\documentclass{article}

\usepackage{authblk}

\usepackage{amssymb}
\usepackage{amsmath}
\usepackage{amsthm}

\usepackage{stmaryrd}
\usepackage{romanbar}
\usepackage{boldfonts}

\usepackage{verbatim}

\usepackage[margin=1in]{geometry}

\usepackage[dvipsnames]{xcolor}

\usepackage{tikz}

\newtheorem{thm}{Theorem}
\newtheorem{prop}{Proposition}
\newtheorem{lem}{Lemma}

\newtheorem{remark}{Remark}

\newcommand{\vx}{\mathbf{x}}
\newcommand{\vu}{\mathbf{u}}
\newcommand{\vv}{\mathbf{v}}
\newcommand{\vw}{\mathbf{w}}
\newcommand{\vf}{\mathbf{f}}
\newcommand{\vg}{\mathbf{g}}
\newcommand{\vn}{\mathbf{n}}

\newcommand{\vc}{\mathbf{c}}
\newcommand{\vb}{\mathbf{b}}
\newcommand{\va}{\mathbf{a}}
\newcommand{\vd}{\mathbf{d}}

\newcommand{\vD}{\boldsymbol{D}}
\newcommand{\vF}{\boldsymbol{F}}
\newcommand{\vI}{\boldsymbol{I}}
\newcommand{\vM}{\boldsymbol{M}}
\newcommand{\vR}{\boldsymbol{R}}
\newcommand{\vT}{{\boldsymbol{T}}}
\newcommand{\vS}{\boldsymbol{S}}

\newcommand{\vB}{\boldsymbol{B}}
\newcommand{\vC}{\boldsymbol{C}}

\newcommand{\avrg}[1]{{\left\{\kern-0.5ex\left\{ #1 \right\}\kern-0.5ex\right\}}} 
\newcommand{\Th}{\mathcal{T}_h} 

\newcommand{\tr}{\mbox{tr}} 
\newcommand{\di}{\mathop{\rm div}\nolimits} 
\newcommand{\rot}{\mathop{\rm rot}\nolimits} 
\newcommand{\Vdiv}{\mathcal{V}}
\newcommand{\M}{\mathcal{M}}
\newcommand{\G}{\mathcal{G}}
\newcommand{\A}{\mathcal{A}}

\def\restriction#1#2{\mathchoice
	{\setbox1\hbox{${\displaystyle #1}_{\scriptstyle #2}$}
		\restrictionaux{#1}{#2}}
	{\setbox1\hbox{${\textstyle #1}_{\scriptstyle #2}$}
		\restrictionaux{#1}{#2}}
	{\setbox1\hbox{${\scriptstyle #1}_{\scriptscriptstyle #2}$}
		\restrictionaux{#1}{#2}}
	{\setbox1\hbox{${\scriptscriptstyle #1}_{\scriptscriptstyle #2}$}
		\restrictionaux{#1}{#2}}}
\def\restrictionaux#1#2{{#1\,\smash{\vrule height .8\ht1 depth .85\dp1}}_{\,#2}} 

\numberwithin{equation}{section}

\title{Finite Element Approximation of Steady Flows\\ of Colloidal Solutions}

\author{Andrea Bonito\thanks{Texas A\&M University, Department of Mathematics, 3368 TAMU, College Station, TX 77843-3368, USA. email: \texttt{bonito@math.tamu.edu}. Supported by the NSF Grant DMS-1817691;}
	, Vivette Girault\thanks{Sorbonne-Universit\'e, CNRS, Universit\'e de Paris, Laboratoire Jacques-Louis Lions (LJLL), F-75005 Paris, France. email: \texttt{girault@ann.jussieu.fr};}
	, Diane Guignard\thanks{University of Ottawa, Department of Mathematics and Statistics, 150 Louis-Pasteur Pvt, Ottawa, ON, Canada K1N 6N5. email: \texttt{dguignar@uottawa.ca};}
	, \\Kumbakonam R. Rajagopal\thanks{Texas A\&M University, Department of Mechanical Engineering, College Station, TX 77843-3123, USA. email: \texttt{krajagopal@tamu.edu};}
	, and Endre S\"uli\thanks{Mathematical Institute, University of Oxford, Andrew Wiles Building, Woodstock Road, Oxford OX2 6GG, United Kingdom. email: \texttt{suli@maths.ox.ac.uk}.}}

\begin{document}

\maketitle

\begin{abstract}
We consider the mathematical analysis and numerical approximation of a system of nonlinear partial differential equations that arises in models that have relevance to  steady isochoric flows of colloidal suspensions. The symmetric velocity gradient is assumed to be a monotone nonlinear function of the deviatoric part of the Cauchy stress tensor. We prove the existence of a weak solution to the problem, and under the additional assumption that the nonlinearity involved in the constitutive relation is Lipschitz continuous we also prove uniqueness of the weak solution. We then construct mixed finite element approximations of the system using both conforming and nonconforming finite element spaces. For both of these we prove the convergence of the method to the unique weak solution of the problem, and in the case of the conforming method we provide a bound on the error between the analytical solution and its finite element approximation in terms of the best approximation error from the finite element spaces. We propose first a Lions--Mercier type iterative method and next a classical fixed-point algorithm to solve the finite-dimensional problems resulting from the finite element discretisation of the system of nonlinear partial differential equations under consideration and present numerical experiments that illustrate the practical performance of the proposed numerical method.
\end{abstract}


\smallskip


\section{Introduction} \label{sec:S_Intro}

The classical incompressible Navier--Stokes constitutive equation and its usual generalisations, the constitutive relations for the incompressible Stokesian fluid, are explicit expressions for the Cauchy stress in terms of the symmetric part of the velocity gradient. The Stokesian fluid is defined by the constitutive expression
\begin{equation} \label{eq:Stokesian}
\vT = -p\vI + \vf(\vD),
\end{equation}
where $\vT$  is the Cauchy stress, $-p\vI$ is the indeterminate part of the stress due to the constraint of incompressibility and $\vD$ is the symmetric part of the velocity gradient, $\vD= \frac{1}{2}(\nabla \vu +
(\nabla \vu)^{t})$. The incompressible Navier--Stokes fluid is a special sub-class of \eqref{eq:Stokesian} that is linear in the symmetric part of the velocity gradient and is defined through:
\begin{equation} \label{eq:Navier--Stokes}
\vT = -p\vI + 2\mu \vD,
\end{equation}
where $\mu$ is the viscosity of the fluid. Power-law fluids are another popular sub-class of \eqref{eq:Stokesian}, the power-law fluid being defined through the constitutive equation
\begin{equation} \label{eq:power-law}
\vT = -p\vI + 2\mu_0 \big(1+\alpha{\rm tr}(\vD^2)\big)^m \vD,
\end{equation}
where $\mu_0$  and  $\alpha$ are positive constants and $m$ is  a  constant; if $m$ is zero we recover the Navier--Stokes fluid model, if it is negative we have a shear-thinning fluid model and if it is positive we have a shear-thickening fluid model.
There are however many fluids that cannot be described by constitutive equations of the  form \eqref{eq:Stokesian} but require ``relations", in the true mathematical sense of the term, between the Cauchy stress and the symmetric part of the velocity gradient. Implicit constitutive relations that involve higher time derivatives of the stress and the symmetric part of the velocity gradient have been proposed to describe the response of non-Newtonian fluids that exhibit  viscoelastic response\footnote{While the Maxwell fluid (see Maxwell (1866)~\cite{ref:Maxwell1866}) is defined through a constitutive relation involving the derivative of the stress, it is not an implicit model in that the symmetric part of the velocity gradient can be explicitly defined in terms of the stress and the time derivative of the stress.}  (see Burgers (1939)~\cite{ref:Burgers39}, Oldroyd (1950)~\cite{ref:Oldroyd50}); that is fluids that exhibit phenomena like stress relaxation. However, purely implicit algebraic relationship between the stress and the symmetric part of the velocity gradient were not considered to describe non-Newtonian fluids until recently. Such models are critical if one is interested in describing the response of fluids which do not exhibit viscoelasticity but whose material properties depend on the mean value of the stress and the shear rate, a characteristic exhibited by many fluids and colloids, as borne out by numerous experiments.
Consider for example an incompressible fluid whose viscosity depends on the mechanical pressure\footnote{The terminology ``pressure"  is often misused, especially in nonlinear fluids; for a detailed discussion of the same see Rajagopal (2015)~\cite{ref:Raj15}.}  (mean value of the stress) and is shear-thinning, whose constitutive relation takes the form
\begin{equation} \label{eq:NLconst_rel}
\vT=-p\vI+2\mu\big(p,{\rm tr}(\vD^2)\big)\vD.
\end{equation}
Since ${\rm tr}(\vD) = \di(\vu) =0$,
\begin{equation} \label{eq:tr_T=p/3}
 {\rm tr} (\vT) = - 3p, \quad \mbox{i.e.,} \quad p=-\frac{1}{3}{\rm tr} (\vT),
\end{equation}
the above equation takes the form
\begin{equation} \label{eq:tr_T=p/3.bis}
\vT =\frac{1}{3}({\rm tr} (\vT)) \vI+ 2\mu\big(\frac{1}{3} {\rm tr} (\vT),{\rm tr}(\vD^2 )\big)\vD.
\end{equation}
(The factor 1/3 is related to the number of space dimensions $d=3$; in two dimensions it would be replaced by 1/2.) The above expression is of the form
\begin{equation} \label{eq:implicit1}
\vf(\vT,\vD)={\bf 0},
\end{equation}
which is an implicit relationship between the stress and the symmetric part of the velocity gradient. Rajagopal (2003)~\cite{Raj2003}, (2006)~\cite{ref:Raj2006} introduced the implicit relationship of the above form (and also the much more general implicit relationship between the history of the stress and the history of the deformation gradient) to describe materials whose properties depend upon the pressure and the shear rate. In fact, the properties of all fluids depend upon the pressure: it is just a matter of how large the variation of the pressure is in order for one to take the variation of the properties into account. The book by Bridgman (1931)~\cite{ref:Bridgman31} entitled ``Physics of High Pressures" provides copious references to the experimental literature before 1931 on the variation of the viscosity of fluids with pressure, and one can find recent references to the experimental  literature  on  the dependence of viscosity on pressure in M\'alek and Rajagopal (2006)~\cite{ref:Malek_Raj06}. Stokes (1845)~\cite{ref:Stokes1845} recognised that the viscosity of fluids varies with pressure, but in the case of sufficiently slow flows in channels and pipes he assumed that the viscosity could be considered a constant. Suffice to say, constitutive relations of the class \eqref{eq:implicit1} are necessary to describe the response of fluids whose viscosity depends on the pressure.
Also as mentioned earlier, the implicit constitutive relation \eqref{eq:implicit1} is useful to describe the behaviour of colloids. Recently, Perl\'acov\'a and Pr\v{u}\v{s}a (2015)~\cite{ref:PerlacPrus15} (see also LeRoux and Rajagopal (2013)~\cite{ref:LeRou_Raj13}) used an implicit model belonging to a sub-class of \eqref{eq:implicit1} to describe the response of colloidal solutions as presented in the experimental work of Boltenhagen {\it et al.} (1997)~\cite{ref:Bol_Hu_Mat_Pi97},
Hu {\it et al.} (1998)~\cite{ref:Hu_Bolt_Mat_Pin98}, Lopez-Diaz {\it et al.} (2010)~\cite{ref:Lop_Sar_Gar_Cas10} among others.
Notice that while one always expresses the incompressible Navier--Stokes fluid by the representation \eqref{eq:Navier--Stokes}, it is perfectly reasonable to describe it as
\begin{equation} \label{eq:NS2}
\vD=\varphi \vI + \frac{1}{2\mu} \vT, \quad \mbox{where}\  \varphi=\frac{p}{2\mu}.
\end{equation}
In fact, it is the representation \eqref{eq:NS2}  that is in keeping with causality as the stress is the cause and the velocity and hence its gradient is the effect, and this fact cannot be overemphasised. Such a representation would imply that the Stokes assumption that is often appealed to is incorrect (see Rajagopal (2013)~\cite{ref:Raj13} for a detailed discussion of the same). M\'alek {et al.} (2010)~\cite{ref:Mal_Pru_Raj} generalised \eqref{eq:NS2} to stress power-law fluids, namely constitutive relations of the form:
\begin{equation} \label{eq:powerlaw2}
\begin{split}
\vT &= -p\vI + \vT^\bd,\\
\vD &= \gamma \big[1+\beta {\rm tr}((\vT^\bd)^2 )\big]^n \vT^\bd,
\end{split}
\end{equation}
where $\vT^\bd$ is the deviatoric part of the Cauchy stress, $\gamma$ and $\beta$ are positive constants, and $n$ is a constant that can be positive, negative or zero. The constitutive relation \eqref{eq:powerlaw2} is capable of describing phenomena that the classical power-law models are incapable of describing. For instance, the constitutive models \eqref{eq:powerlaw2} can describe limiting strain rate as well as fluids which allow the possibility of the strain rate initially increasing with stress and later decreasing with stress; both such responses cannot be described by the classical power-law fluid model \eqref{eq:power-law} (see the discussion in M\'alek {\it et al.} (2010)~\cite{ref:Mal_Pru_Raj} with regard to the difference in the response characteristics of the stress power-law fluid and the classical power-law fluid).
We are interested in a further generalization of the constitutive relation of the form \eqref{eq:powerlaw2} that is appropriate for describing the response of colloidal solutions. This constitutive relation takes the form:
\begin{equation} \label{eq:powerlaw3}
\vD = \Big\{\gamma\big[1+\beta {\rm tr}((\vT^\bd)^2 )\big]^n + \alpha \Big\} \vT^\bd,
\end{equation}
where $\alpha$, $\beta$, and $\gamma$ are positive constants, $n$ is a real number, and $\vT^\bd$ is the deviatoric part of the Cauchy stress. The shear stress in a fluid undergoing simple shear flow, that is described by the constitutive relation given above, increases from zero to a maximum, then decreases to a local minimum, and then increases monotonically as the shear stress increases from zero. As discussed by Le Roux and Rajagopal~\cite{ref:LeRou_Raj13}, and Perl\'acov\'a and Pr\v{u}\v{s}a~\cite{ref:PerlacPrus15}, many colloids exhibit such behavior. The constitutive relation that we introduce first in \eqref{pb:S_ENL2} and next in \eqref{pb:NS_EL} includes \eqref{eq:powerlaw3} as a special sub-class. It can be posed within a Hilbert space setting owing to the presence of the coefficient $ \alpha$ in \eqref{eq:powerlaw3}, but nevertheless,  it is a challenging problem as it involves two nonlinearities: the monotone part in the constitutive relation and the inertial (convective) term. The problem without the inertial term, see Subsection \ref{sec:S_ENL} below, has already been analysed in \cite{ref:BGS18}, while the analysis of the steady-state incompressible Navier--Stokes equations is well-established, see for instance \cite{Tem,ref:GiR}. With both nonlinearities present in the model, proving the existence of a weak  solution, for instance, to the best of our knowledge cannot be done by simply coupling the techniques used for these two problems, namely the Browder--Minty theorem and the Galerkin method combined with Brouwer's fixed point theorem and a weak compactness argument. More refined arguments are needed; they are crucial to the proofs  of Lemmas \ref{lem:conv_product} and \ref{lem:T_to_G(u)} below.

This work is organised as follows. The notation and the functional-analytic setting are recalled in the next subsection. In Section \ref{sec:S_EL}, both linear and fully nonlinear versions of the formulation are briefly analysed for the Stokes system, i.e., without the inertial (convective) term. The theoretical analysis of the complete nonlinear system is carried out in Section \ref{sec:NS_ENL}. The main results of this section are Theorem~\ref{thm:existence1} for the existence of a solution and Proposition~\ref{lem:uniqueness} for the uniqueness of a solution under additional assumptions on the input data. In Section~\ref{sec:NS_ENL_App}, conforming  finite element approximations of these models are proposed and error estimates are derived. The cases of both simplicial and hexahedral elements are discussed.  The analysis of the latter is less satisfactory as it requires subdivisions consisting of parallelepipeds and suffers from a higher computational cost. This motivates the introduction of nonconforming approximations in Section \ref{sec:NS_ENL_AppNC}.  In Section \ref{sec:numerics},  two decoupling algorithms are presented and compared: a Lions--Mercier algorithm adapted to a system with a monotone part and an elliptic part, and a classical fixed-point algorithm alternating between the approximation of a Navier-Stokes system and the nonlinear constitutive relation for the stress. Numerical experiments are performed with conforming finite elements on a square mesh in two dimensions. The theoretically  established convergence of the scheme is confirmed and convergence of both decoupled algorithms is observed.


\subsection{Notation and preliminaries} \label{subsec:notation}

Let $\Omega\subset\mathbb{R}^d$, $d\in\{2,3\}$, be a bounded, open, simply connected Lipschitz domain. We consider the function spaces
\begin{equation} \label{def:spaces}
Q:=L_0^2(\Omega), \quad V:=H_0^1(\Omega)^d \quad \mbox{and} \quad M:=\{\vS\in L^2(\Omega)_{{\rm sym}}^{d\times d}: \, \tr(\vS)=0\},
\end{equation}
for the pressure, the velocity, and the deviatoric stress tensor, respectively. As usual,
\[L^2_0(\Omega) = \left\{q \in L^2(\Omega):\, \int_\Omega q = 0\right\},\]
the zero mean value constraint being introduced to fix the undetermined additive constant in the mechanical pressure.
Here the subscript sym indicates that the $d \times d$ tensors under consideration are assumed to be symmetric. Henceforth, the symmetric gradient of the velocity field $\vv$ (or, briefly, symmetric velocity gradient) will be denoted by
\begin{equation} \label{def:strain}
\vD(\vv) :=\frac{1}{2}(\nabla\vv+(\nabla\vv)^t)
\end{equation}
and the deviatoric part of a  $d \times d$ tensor $\vS$ is defined by
\begin{equation} \label{def:deviator}
\vS^{\bd}:=\vS-\frac{1}{d}\tr(\vS)\vI
\end{equation}
with $\vI$ the $d\times d$ identity tensor;  thus the trace of $\vS^{ \bd}$ is zero.
We denote by $\Vdiv$ the subspace of $V$ consisting of all divergence-free functions contained in $V$; that is,
\begin{equation} \label{def:space_Vdiv}
\Vdiv:=\{\vv\in V: \, \di(\vv)=0\}.
\end{equation}
For vector-valued functions $\vv:\Omega\rightarrow\mathbb{R}^d$, we write
$$\|\vv\|_{L^2(\Omega)}:=\|\, |\vv| \,\|_{L^2(\Omega)} \quad \mbox{and} \quad \|\vv\|_{L^{\infty}(\Omega)}:=\|\, |\vv| \,\|_{L^{\infty}(\Omega)}$$
with $|\cdot |$ signifying the Euclidean norm on $\mathbb{R}^d$, while for tensor-valued functions $\vS:\Omega\rightarrow\mathbb{R}^{d\times d}$, we define
$$\|\vS\|_{L^2(\Omega)}:=\| \, |\vS| \, \|_{L^2(\Omega)},$$
where now
$$|\vS|:=\sqrt{\vS:\vS}$$
is the Frobenius norm of $\vS$. Clearly, $M$ is a Hilbert space with this norm. We recall the Poincar\'e and  Korn inequalities, which are, respectively, the following: there exist positive constants $C_P$ and $C_K$ such that
\begin{equation} \label{def:poincare}
\|v\|_{L^2(\Omega)}\leq C_P\|\nabla v\|_{L^2(\Omega)} \quad \forall\, v\in H_0^1(\Omega)
\end{equation}
and
\begin{equation} \label{def:korn}
\|\nabla \vv\|_{L^2(\Omega)}\leq C_K\|\vD(\vv)\|_{L^2(\Omega)} \quad \forall\, \vv\in V.
\end{equation}
We endow $V$ (and $\Vdiv$) with the norm
\begin{equation} \label{def:equivnorm}
\|\cdot\|_V:=\|\vD(\cdot)\|_{L^2(\Omega)}.
\end{equation}
Both $V$ and $\Vdiv$ are Hilbert spaces with this norm, because $\|\cdot\|_V$ is equivalent to both the $H^1(\Omega)^{d\times d}$ norm and the $H^1(\Omega)^{d \times d}$ semi-norm, thanks to \eqref{def:poincare}, \eqref{def:korn} and the trivial relation $\|\vD(\vv)\|_{L^2(\Omega)}\leq \|\nabla\vv\|_{L^2(\Omega)}$.

\section{Stokes system with linear and nonlinear constitutive relations} \label{sec:S_EL}

In this section we study two preliminary model problems without the inertial term; the first one simply reduces to the Stokes system, while the second model problem involves a monotone nonlinearity treated by the Browder--Minty approach.


\subsection{The Stokes system} \label{subsec:Stokes}

Let us consider the problem
\begin{equation} \label{pb:S_EL}
\left\{\begin{array}{rcll}
-\di(\vT) & =&  \vf &\quad \mbox{in } \Omega, \\
\vD(\vu) & =&   \alpha\vT^\bd & \quad \mbox{in } \Omega, \\
\di(\vu) & =&  0 &\quad \mbox{in } \Omega, \\
\vu & =&  \mathbf{0} &\quad \mbox{on } \partial\Omega,
\end{array}
\right.
\end{equation}
where $\vf:\Omega\rightarrow\mathbb{R}^d$ is a prescribed external force, $\vD(\vu)$ is defined by \eqref{def:strain}, the unknown tensor $\vT$ is symmetric, and $\alpha$ is a given positive constant, the reciprocal of the viscosity coefficient.  Here, we assume that $\vf\in L^2(\Omega)^d$ for simplicity, but  a similar analysis holds for the general case $\vf\in V'=H^{-1}(\Omega)^d$; see for instance Remark \ref{rem:1} in Section \ref{sec:NS_ENL}.  By decomposing the Cauchy stress $\vT$ as $\vT=\vT^{\bd} + \frac{1}{d} \tr(\vT) \vI$ and inserting this in the first equation of \eqref{pb:S_EL} we arrive at the following equivalent problem:
\begin{equation} \label{pb:S_EL2}
\left\{\begin{array}{rcll}
-\di(\vT^\bd)-\frac{1}{d}\nabla\tr(\vT) & = & \vf & \mbox{in } \Omega, \\
\vD(\vu) & = &  \alpha\vT^\bd & \mbox{in } \Omega, \\
\di(\vu) & = & 0 & \mbox{in } \Omega, \\
\vu & = & \mathbf{0} & \mbox{on } \partial\Omega,
\end{array}\right.
\end{equation}
which we recognise to be the Stokes system where the mechanical pressure (mean normal stress) is $p:=-\frac{1}{d}\tr(\vT)$.
Recalling the spaces $M,V,Q$ defined in \eqref{def:spaces} and using the relation
$$\vD(\vv):\vS=\nabla\vv:\vS,$$
which holds\footnote{For any $\vR,\vS\in\mathbb{R}^{d\times d}$, with $\vS$ symmetric, we have that $\vS:\vR=\left(\frac{\vS+\vS^{t}}{2}\right):\vR=\frac{1}{2}\vS:\vR+\frac{1}{2}\vS^{t}:\vR=\frac{1}{2}\vS:\vR+\frac{1}{2}\vS:\vR^{t}=\vS:\left(\frac{\vR+\vR^{t}}{2}\right)$.} for any symmetric tensor $\vS$, the weak formulation of problem \eqref{pb:S_EL2} can be written as follows: find a triple $(\vT^\bd,\vu,p)\in M\times V\times Q$ such that
\begin{alignat}{2} \label{pb:S_EL2_weak}
\begin{aligned}
\int_{\Omega}\vT^\bd:\vD(\vv) - \int_{\Omega}p\di(\vv) & =  \displaystyle{\int_{\Omega}}\vf\cdot\vv && \quad\forall\, \vv\in V, \\
 \alpha\int_{\Omega}\vT^\bd:\vS - \int_{\Omega}\vS:\vD(\vu) & =  0 && \quad \forall\, \vS\in M, \\
-\int_{\Omega}q\di(\vu)  & =  0 && \quad \forall\, q\in Q.
\end{aligned}
\end{alignat}
For any $\vS\in M$, $\vv\in V$, and $q\in Q$, we set
\begin{align*}
	b_1(\vS,\vv) & :=  \int_{\Omega}\vS:\vD(\vv), \\
	b_2(\vv,q) & :=  -\int_{\Omega}q\di(\vv).
\end{align*}
As is usual for the Stokes problem, the unknown pressure can be eliminated from \eqref{pb:S_EL2_weak} by restricting the test functions $\vv$ to $\Vdiv$. In addition, the variable $\vu$ can also be eliminated by treating the first line of \eqref{pb:S_EL2_weak} as a constraint, thus leading to an equivalent (reduced) problem for which the two variables $p$ and $\vu$ are eliminated. The equivalence is based on the following (inf-sup) conditions

\begin{equation} \label{infsup:b1}
\inf_{\vv\in \Vdiv}\sup_{\vS\in M}\frac{b_1(\vS,\vv)}{\|\vS\|_{L^2(\Omega)}\|\vD(\vv)\|_{L^2(\Omega)}}\geq 1
\end{equation}
and
\begin{equation} \label{infsup:b2}
\exists\beta>0: \quad \inf_{q\in Q}\sup_{\vv\in V}\frac{b_2(\vv,q)}{\|q\|_{L^2(\Omega)}\|\vD(\vv)\|_{L^2(\Omega)}}\geq\inf_{q\in Q}\sup_{\vv\in V}\frac{b_2(\vv,q)}{\|q\|_{L^2(\Omega)}\|\nabla\vv\|_{L^2(\Omega)}}\geq \beta,
\end{equation}
where we have used that $\|\vD(\vv)\|_{L^2(\Omega)}\leq \|\nabla\vv\|_{L^2(\Omega)}$. It is well-known that the spaces $V$ and $Q$ defined in \eqref{def:spaces} satisfy the inf-sup condition \eqref{infsup:b2}, see for instance~\cite{ref:GiR}, while the relation \eqref{infsup:b1} can be easily shown by  observing that, for a given $\vv\in \Vdiv$, we have $\vR:=\vD(\vv)\in M$ since $\tr(\vR)=\tr(\vD(\vv))=\di(\vv)=0$ and $\vD(\vv)$ is symmetric. Therefore, $b_1(\vR,\vv)=\|\vD(\vv)\|_{L^2(\Omega)}^2$ and thus
\begin{align*}
\sup_{\vS\in M}\frac{b_1(\vS,\vv)}{\|\vS\|_{L^2(\Omega)}}\geq \frac{b_1(\vR,\vv)}{\|\vR\|_{L^2(\Omega)}}=\|\vD(\vv)\|_{L^2(\Omega)}.
\end{align*}
We can then eliminate the incompressibility constraint by seeking $\vu\in \Vdiv$, yielding the (partially reduced) problem: find $(\vT^\bd,\vu)\in M\times \Vdiv$ such that
\begin{alignat}{2} \label{pb:S_EL2_weak2}
\begin{aligned}
\int_{\Omega} \vT^\bd:\vD(\vv) & =  \displaystyle{\int_{\Omega}}\vf\cdot\vv && \quad \forall\, \vv\in \Vdiv, \\[0.25cm]
 \alpha\int_{\Omega}\vT^\bd:\vS - \int_{\Omega}\vS:\vD(\vu) & =  0 &&\quad \forall\, \vS\in M.
\end{aligned}
\end{alignat}
Clearly, each solution of \eqref{pb:S_EL2_weak} satisfies \eqref{pb:S_EL2_weak2}. Conversely, it follows from the inf-sup condition (\ref{infsup:b2}) that for any solution $(\vT^\bd,\vu)$ of \eqref{pb:S_EL2_weak2} there exists a unique $p\in Q$ such that $(\vT^\bd,\vu,p)$ is the solution of \eqref{pb:S_EL2_weak}; see~\cite{ref:GiR}. Hence these two problems are equivalent. Furthermore, we can eliminate the unknown $\vu$ by proceeding as follows; see~\cite{ref:BGS18}. First, we introduce the decomposition $M=\M\oplus\M^{\perp}$ with
\begin{equation} \label{eq:calM}
\M:=\{\vS\in M \,: \, b_1(\vS,\vv)=0 \quad \forall\, \vv\in\Vdiv\},
\end{equation}
the kernel of $b_1$, and
\begin{equation*}
\M^{\perp}:=\{\vS\in M \,: \, \int_{\Omega}\vS:\vR=0 \quad \forall\, \vR\in\M\}
\end{equation*}
its orthogonal complement in $M$, and we write $\vT^\bd=\vT_0^\bd+\vT_\vf^\bd$ with $\vT_0^\bd\in\M$ and $\vT_\vf^\bd\in\M^{\perp}$. The condition \eqref{infsup:b1} ensures the existence and uniqueness of $\vT_\vf^\bd\in\M^{\perp}$ satisfying
\begin{equation} \label{eq:Tf}
b_1(\vT_\vf^\bd,\vv) = \int_{\Omega}\vf\cdot\vv \quad \forall\, \vv\in\Vdiv \quad \mbox{and} \quad \|\vT_\vf^\bd\|_{L^2(\Omega)}\leq C_PC_K\|\vf\|_{L^2(\Omega)}
\end{equation}
with $C_P$ and $C_K$ the constants in Poincar\'e's and Korn's inequalities \eqref{def:poincare} and \eqref{def:korn}, respectively. We finally get the (fully reduced) problem: find $\vT_0^\bd\in\M$ such that
\begin{equation} \label{pb:S_EL2_weak3}
\begin{array}{rcll}
  \alpha\int_{\Omega}\vT_0^\bd:\vS  =  -\alpha\int_{\Omega}\vT_\vf^\bd:\vS  & \forall\, \vS\in \M.
\end{array}
\end{equation}
The well-posedness of problem \eqref{pb:S_EL2_weak3} follows from the Lax--Milgram lemma,  while its equivalence to the original problem \eqref{pb:S_EL2_weak} is guaranteed by \eqref{infsup:b1} and \eqref{infsup:b2}.

Of course, in this simple model with a linear constitutive relation,  $\vT_0^d=\mathbf{0}$ since the right-hand side of \eqref{pb:S_EL2_weak3} vanishes and $\alpha(\cdot,\cdot)_\Omega $  is an inner product on $\M$. However, the framework developed here will be used in the sequel in a more general setting.


\subsection{Stokes model with a nonlinear constitutive relation} \label{sec:S_ENL}

Next, we consider the following Stokes-like system with a nonlinear relation between the stress and the symmetric velocity gradient:
\begin{equation} \label{pb:S_ENL2}
\left\{\begin{array}{rcll}
-\di(\vT^\bd)-\frac{1}{d}\nabla\tr(\vT) & = & \vf & \mbox{in } \Omega, \\
\vD(\vu) & = &  \alpha\vT^\bd+\gamma\mu(|\vT^\bd|)\vT^\bd & \mbox{in } \Omega, \\
\di(\vu) & = & 0 & \mbox{in } \Omega, \\
\vu & = & \mathbf{0} & \mbox{on } \partial\Omega
\end{array}\right.
\end{equation}
with $\gamma$ a given positive constant, and where $\mu\in\mathcal{C}^1((0,+\infty))\cap\mathcal{C}^0([0,+\infty))$ is a given function satisfying

\begin{equation} \label{eqn:mu_diff}
\frac{{\rm d}}{{\rm d}a}(\mu(a)a)>0 \quad \forall\, a\in\mathbb{R}_{> 0}
\end{equation}
and
\begin{equation} \label{eqn:mu_pos_C1}
\mu(a) >  0 \quad \mbox{and} \quad \mu(a)a\leq C_1 \quad  \forall\, a \in \mathbb{R}_{\geq 0}
\end{equation}
for some positive constant $C_1$.
Since $\mu$ is continuous on any subinterval of $\mathbb{R}_{\geq 0}$, the second part of \eqref{eqn:mu_pos_C1} implies that $\mu$ is bounded above and we denote its maximum by $\mu_{\rm max}$,
\begin{equation} \label{eq:mumax}
0 < \mu(a) \le \mu_{\rm max}\quad \forall \, a \in \mathbb{R}_{\geq 0}.
\end{equation}
Moreover, proceeding as in the proof of \cite[Lemma 4.1]{BMRS14}, we deduce from \eqref{eqn:mu_diff} and \eqref{eqn:mu_pos_C1} that for any $\vR,\vS\in\mathbb{R}^{d\times d}$, the following monotonicity property hold:
\begin{eqnarray} \label{eqn:mu_mon}
(\mu(|\vR|)\vR-\mu(|\vS|)\vS):(\vR-\vS) \geq 0,
\end{eqnarray}
with equality if and only if $\vR=\vS$.

Introducing again $p:=-\frac{1}{d}\tr(\vT)$, the weak formulation of problem \eqref{pb:S_ENL2} reads as follows: find  a  triple $(\vT^\bd,\vu,p)\in M\times V\times Q$ such that
\begin{alignat}{2} \label{pb:S_ENL2_weak}
\begin{aligned}
\int_{\Omega}\vT^\bd:\vD(\vv) - \int_{\Omega}p\di(\vv)& =  \displaystyle{\int_{\Omega}}\vf\cdot\vv &&\quad \forall\, \vv\in V, \\[0.25cm]
\alpha\int_{\Omega}\vT^\bd:\vS + \gamma\int_{\Omega}\mu(|\vT^\bd|)\vT^\bd:\vS - \int_{\Omega}\vS:\vD(\vu) & =  0 &&\quad \forall\, \vS\in M, \\[0.25cm]
-\int_{\Omega}q\di(\vu) & =  0 &&\quad \forall\, q\in Q.
\end{aligned}
\end{alignat}

Proceeding exactly as in Section \ref{subsec:Stokes}, we first eliminate the pressure, and we thus deduce that problem \eqref{pb:S_ENL2_weak} is equivalent to the following problem: find $(\vT^\bd,\vu)\in M\times \Vdiv$ such that
\begin{alignat}{2} \label{pb:S_ENL2_weak2}
\begin{aligned}
\int_{\Omega}\vT^\bd:\vD(\vv) & =  \displaystyle{\int_{\Omega}}\vf\cdot\vv &&\quad \forall\, \vv\in \Vdiv, \\[0.25cm]
\alpha\int_{\Omega}\vT^\bd:\vS + \gamma\int_{\Omega}\mu(|\vT^\bd|)\vT^\bd:\vS - \int_{\Omega}\vS:\vD(\vu) & =  0 &&\quad \forall\, \vS\in M,
\end{aligned}
\end{alignat}
which is further equivalent to the following problem: find $\vT_0^\bd\in\M$ such that
\begin{equation} \label{pb:S_ENL2_weak3}
\begin{array}{rcll}
\alpha\int_{\Omega}(\vT_0^\bd+\vT_\vf^\bd):\vS + \gamma\int_{\Omega}\mu(|\vT_0^\bd+\vT_\vf^\bd|) (\vT_0^\bd+\vT_\vf^\bd):\vS  = 0 & \forall\, \vS\in \M,
\end{array}
\end{equation}
with $\vT_\vf^\bd\in\M^{\perp}$ the solution of \eqref{eq:Tf}.
The Browder--Minty theorem, see for instance~\cite{ref:Minty}, guarantees the existence of a solution to problem \eqref{pb:S_ENL2_weak3}. Indeed, let $\A:M\rightarrow M'$ be defined for $\vR,\vS\in M$ by
\begin{equation} \label{def:mapA}
\langle \A(\vR), \vS\rangle_M:= \alpha\int_{\Omega}\vR:\vS +  \gamma\int_{\Omega}\mu(|\vR|)\vR:\vS,
\end{equation}
where $\langle \cdot,\cdot \rangle_M$ denotes the  duality  pairing between $M$ and its dual space, $M'$.
It then easily follows that the mapping $\vT_0^\bd\mapsto\A(\vT_0^\bd+\vT_\vf^\bd)$ is bounded, monotone, coercive and hemi-continuous. By the Browder--Minty theorem these imply surjectivity of $\A$ and thereby existence of a solution,  while its uniqueness follows from the strict monotonicity of $\A$.

\section{Navier--Stokes with nonlinear constitutive relation} \label{sec:NS_ENL}

Now, we focus on our problem of interest, where a convective term is added  to the first equation of \eqref{pb:S_ENL2}, i.e., we consider the problem
\begin{equation} \label{pb:NS_EL}
\left\{\begin{array}{rcll}
(\vu\cdot\nabla)\vu-\di(\vT) & = & \vf & \mbox{in } \Omega, \\
\vD(\vu) & = &  \alpha\vT^\bd+\gamma\mu(|\vT^\bd|)\vT^\bd & \mbox{in } \Omega, \\
\di(\vu) & = & 0 & \mbox{in } \Omega, \\
\vu & = & \mathbf{0} & \mbox{on } \partial\Omega.
\end{array}\right.
\end{equation}
We prove a priori estimates, construct a solution, and give sufficient conditions for global uniqueness.


\subsection{Reformulation} \label{subsec:forms}

By introducing the pressure $p:=-\frac{1}{d}\tr(\vT)$, problem \eqref{pb:NS_EL} can be rewritten as follows:
\begin{equation} \label{pb:NS_EL_v2}
\left\{\begin{array}{rcll}
(\vu\cdot\nabla)\vu-\di(\vT^{\bd})+\nabla p & = & \vf & \mbox{in } \Omega, \\
\vD(\vu) & = &  \alpha\vT^\bd+\gamma\mu(|\vT^\bd|)\vT^\bd & \mbox{in } \Omega, \\
\di(\vu) & = & 0 & \mbox{in } \Omega, \\
\vu & = & \mathbf{0} & \mbox{on } \partial\Omega.
\end{array}\right.
\end{equation}
In order to bring forth an elliptic term on the left-hand side of the first equation of \eqref{pb:NS_EL_v2}, we rewrite  the second equation in \eqref{pb:NS_EL_v2} as
\begin{equation} \label{eqn:constitutive_rel}
\vT^\bd=\frac{1}{\alpha}\vD(\vu)-\frac{\gamma}{\alpha}\mu(|\vT^\bd|)\vT^\bd,
\end{equation}
and thus by substituting this relation into the first equation of \eqref{pb:NS_EL_v2} we get
\begin{equation} \label{pb:NS_EL_v3}
\left\{\begin{array}{rcll}
(\vu\cdot\nabla)\vu-\frac{1}{\alpha}\di(\vD(\vu))+\nabla p & = & \vf-\frac{\gamma}{\alpha}\di(\mu(|\vT^\bd|)\vT^\bd) & \mbox{in } \Omega, \\[\smallskipamount]
\alpha\vT^\bd+\gamma\mu(|\vT^\bd|)\vT^\bd & = &  \vD(\vu) & \mbox{in } \Omega, \\[\smallskipamount]
\di(\vu) & = & 0 & \mbox{in } \Omega, \\
\vu & = & \mathbf{0} & \mbox{on } \partial\Omega.
\end{array}\right.
\end{equation}
The weak formulation of \eqref{pb:NS_EL_v3} reads: find $(\vT^\bd,\vu,p)\in M\times V\times Q$ such that
\begin{align} \label{pb:weak_epsu}
\begin{aligned}
\int_{\Omega}[(\vu\cdot\nabla)\vu]\cdot\vv + \frac{1}{\alpha}\int_{\Omega}\vD(\vu):\vD(\vv)-\int_{\Omega}p\di(\vv) & =  \int_{\Omega}\vf\cdot\vv+\frac{\gamma}{\alpha}\int_{\Omega}\mu(|\vT^\bd|)\vT^\bd:\vD(\vv), \\
\alpha\int_{\Omega}\vT^\bd:\vS+\gamma\int_{\Omega}\mu(|\vT^\bd|)\vT^\bd:\vS & =  \int_{\Omega}\vD(\vu):\vS, \\
\int_{\Omega}q\di(\vu) & =  0
\end{aligned}
\end{align}
for all $(\vS,\vv,q)\in M\times V\times Q$.

As previously, we eliminate the pressure by restricting the test functions to $\Vdiv$, and we thus obtain the following equivalent reduced problem: find $(\vT^\bd,\vu)\in M\times \Vdiv$ such that
\begin{align}
\int_{\Omega}[(\vu\cdot\nabla)\vu]\cdot\vv + \frac{1}{\alpha}\int_{\Omega}\vD(\vu):\vD(\vv) & =  \int_{\Omega}\vf\cdot\vv+\frac{\gamma}{\alpha}\int_{\Omega}\mu(|\vT^\bd|)\vT^\bd:\vD(\vv), \label{pb:NS_EL_weakR1} \\
\alpha\int_{\Omega}\vT^\bd:\vS+\gamma\int_{\Omega}\mu(|\vT^\bd|)\vT^\bd:\vS & =  \int_{\Omega}\vD(\vu):\vS \label{pb:NS_EL_weakR2}
\end{align}
for all $(\vS,\vv)\in M\times \Vdiv$.

Interestingly, \eqref{pb:NS_EL_weakR1}, \eqref{pb:NS_EL_weakR2} can be further reduced by
observing that, given $\vu$, \eqref{pb:NS_EL_weakR2} uniquely determines $\vT^d$ thanks to the Browder--Minty theorem; see the end of Section \ref{sec:S_ENL}. Thus, we define the mapping $\G:\Vdiv\rightarrow M$ by $\vu\mapsto \vT^\bd$ with $\vT^\bd\in M$ being the unique solution of
\begin{equation} \label{eqn:def_map_G}
\langle \A(  \vT^\bd ), \vS\rangle_M =\int_{\Omega}\vD(\vu):\vS\quad \forall\, \vS\in M,
\end{equation}
where we recall that $\A$ is defined in \eqref{def:mapA}.  With this mapping, \eqref{pb:NS_EL_weakR1}, \eqref{pb:NS_EL_weakR2} is equivalent to the following problem: find $\vu \in \Vdiv$ such that
\begin{equation} \label{eq:reduced3}
\begin{split}
\int_{\Omega}[(\vu\cdot\nabla)\vu]\cdot\vv + \frac{1}{\alpha}\int_{\Omega}\vD(\vu):\vD(\vv)  &=  \int_{\Omega}\vf\cdot\vv
+\frac{\gamma}{\alpha}\int_{\Omega}\mu(|\G(\vu)|)\G(\vu):\vD(\vv).
\end{split}
\end{equation}	

Before embarking on the proof of existence of a solution  to problem \eqref{pb:NS_EL_weakR1}, \eqref{pb:NS_EL_weakR2} we establish a series of \emph{a priori} estimates under the assumption that a solution exists.


\subsection{\emph{A priori} estimates} \label{subsec:aprioribdd}

Assuming that problem \eqref{pb:NS_EL_weakR1}, \eqref{pb:NS_EL_weakR2} has a solution, the following \emph{a priori} estimates hold for any solution $(\vT^\bd,\vu) \in M \times \mathcal{V}$.

\begin{lem}(First \emph{a priori} estimates) \label{lem:apriori}
	Let $|\Omega|$ denote the measure of $\Omega$. Then,
	\begin{equation} \label{apriori_est:u}
	\|\vD(\vu)\|_{L^2(\Omega)} \leq \alpha C_PC_K\|\vf\|_{L^2(\Omega)}+\gamma C_1 |\Omega|^{\frac{1}{2}}
	\end{equation}
	and
	\begin{equation} \label{apriori_est:Td}
	\|\vT^\bd\|_{L^2(\Omega)} \leq C_PC_K\|\vf\|_{L^2(\Omega)}+\frac{\gamma}{\alpha} C_1 |\Omega|^{\frac{1}{2}}
	\end{equation}
	with $C_P$ and $C_K$ signifying the constants in Poincar\'e's and Korn's inequality, respectively, and $C_1$ the constant in \eqref{eqn:mu_pos_C1}.
\end{lem}

\begin{proof}
	Taking $\vS=\vT^\bd$ in \eqref{pb:NS_EL_weakR2} yields
	\begin{equation*}
	\alpha\|\vT^\bd\|_{L^2(\Omega)}^2+\gamma\int_{\Omega}\mu(|\vT^\bd|)|\vT^\bd|^2 = \int_{\Omega}\vD(\vu):\vT^\bd\leq \|\vD(\vu)\|_{L^2(\Omega)}\|\vT^\bd\|_{L^2(\Omega)}.
	\end{equation*}
	Using then the positivity of $\mu$, see \eqref{eqn:mu_pos_C1}, we get
	\begin{equation} \label{eqn:apriori_step1}
	\|\vT^\bd\|_{L^2(\Omega)} \leq \frac{1}{\alpha}\|\vD(\vu)\|_{L^2(\Omega)}.
	\end{equation}
	To obtain a bound for $\vu$, we recall the well-known relation
	\begin{equation} \label{eqn:conv_zero}
	\int_{\Omega}[(\vu\cdot\nabla)\vv]\cdot\vv=0 \quad \forall\, \vu\in \Vdiv, \,\, \forall\, \vv\in V,
	\end{equation}
	which is easily obtained by integration by parts, as follows:
	\begin{equation*}
	 \int_{\Omega}[(\vu\cdot\nabla)\vv]\cdot\vv=\frac{1}{2}\int_{\Omega}\vu\cdot\nabla(|\vv|^2)=-\frac{1}{2}\int_{\Omega}\di(\vu)|\vv|^2=0.
	\end{equation*}
	Therefore, taking $\vv=\vu$ in \eqref{pb:NS_EL_weakR1} and using \eqref{eqn:mu_pos_C1} we obtain
	\begin{align*}
	\frac{1}{\alpha}\|\vD(\vu)\|_{L^2(\Omega)}^2 & =  \int_{\Omega}\vf\cdot\vu+\frac{\gamma}{\alpha}\int_{\Omega}\mu(|\vT^\bd|)\vT^\bd:\vD(\vu) \\
	& \leq  \left(C_PC_K\|\vf\|_{L^2(\Omega)}+\frac{\gamma}{\alpha} C_1 |\Omega|^{\frac{1}{2}}\right)\|\vD(\vu)\|_{L^2(\Omega)},
	\end{align*}
	from which we directly deduce \eqref{apriori_est:u}; \eqref{apriori_est:Td} follows by applying \eqref{apriori_est:u} to \eqref{eqn:apriori_step1}.
\end{proof}

\begin{lem}(Second \emph{a priori} estimates) \label{lem:apriori2}
	Recall that $\mu_{\max}:=\sup_{s\in [0,\infty)}\mu(s)$. We also have
	\begin{equation} \label{apriori_est:u2}
	\|\vD(\vu)\|_{L^2(\Omega)} \leq (\alpha+\gamma\mu_{\max})C_PC_K\|\vf\|_{L^2(\Omega)}
	\end{equation}
	and
	\begin{equation} \label{apriori_est:Td2}
	\|\vT^\bd\|_{L^2(\Omega)} \leq \frac{1}{\alpha}(\alpha+\gamma\mu_{\max})C_PC_K\|\vf\|_{L^2(\Omega)}.
	\end{equation}	
\end{lem}

The advantage of the estimates \eqref{apriori_est:u2} and \eqref{apriori_est:Td2} is that if $\vf=\mathbf{0}$, then we can directly deduce that $\vu=\mathbf{0}$ and $\vT^\bd=\mathbf{0}$ (and consequently $p=0$).

\begin{proof}
	The ingredients of the proof are similar to those used in the proof of Lemma \ref{lem:apriori} and only the derivation of the bound for $\vD(\vu)$ is different. First notice that combining \eqref{pb:NS_EL_weakR1} and \eqref{pb:NS_EL_weakR2} we have
	\begin{equation} \label{pb:NS_EL_weakR12}
	\int_{\Omega}[(\vu\cdot\nabla)\vu]\cdot\vv + \int_{\Omega}\vT^\bd:\vD(\vv) = \int_{\Omega}\vf\cdot\vv \quad \forall\, \vv\in\Vdiv.
	\end{equation}
	Taking $\vv=\vu$ in \eqref{pb:NS_EL_weakR12} we then find that
	\begin{equation} \label{bound_Td_eu}
	\int_{\Omega}\vT^\bd:\vD(\vu)\leq C_PC_K\|\vf\|_{L^2(\Omega)}\|\vD(\vu)\|_{L^2(\Omega)}.
	\end{equation}
	Notice that $\vT^\bd:\vD(\vu)\geq 0$ a.e. in $\Omega$. Indeed, from \eqref{pb:NS_EL_weakR2} we have that
	\begin{equation} \label{eq:symgrad=T}
	\left(\alpha+\gamma\mu(|\vT^\bd|)\right) \vT^\bd = \vD(\vu) \quad \mbox{in } M'
	\end{equation}
	and thus
	\begin{equation*}
	\underbrace{\left(\alpha+\gamma\mu(|\vT^\bd|)\right)}_{> 0} \vT^\bd:\vD(\vu) = |\vD(\vu)|^2\geq 0 \quad \mbox{a.e. in } \Omega.
	\end{equation*}
	Therefore, taking $\vS=\vD(\vu)$ in \eqref{pb:NS_EL_weakR2} and using the upper bound $\mu_{\max}$ for $\mu$ and the bound \eqref{bound_Td_eu} we have
	\begin{align*}
	\|\vD(\vu)\|_{L^2(\Omega)}^2 & =  \alpha\int_{\Omega}\vT^\bd:\vD(\vu)+\gamma\int_{\Omega}\mu(|\vT^\bd|)\vT^\bd:\vD(\vu) \\
	 & \leq  (\alpha+\gamma\mu_{\max})\int_{\Omega}\vT^\bd:\vD(\vu) \\
	 & \leq  (\alpha+\gamma\mu_{\max})C_PC_K\|\vf\|_{L^2(\Omega)}\|\vD(\vu)\|_{L^2(\Omega)},
	\end{align*}
	which yields \eqref{apriori_est:u2}. Finally, the bound \eqref{apriori_est:Td2} for $\vT^\bd$ is obtained by substituting \eqref{apriori_est:u2} in \eqref{eqn:apriori_step1}.
\end{proof}

\begin{remark} \label{rem:1}
	{\rm Similar \emph{a priori} estimates can be derived in the case when $\vf\in V'$ (with $V' = H^{-1}(\Omega)^d$). More precisely, all occurrences of $C_PC_K\|\vf\|_{L^2(\Omega)}$ can be replaced by $\|\vf\|_{\Vdiv'}$, where
	\begin{equation} \label{def:f_in_Vprime}
	\|\vf\|_{\Vdiv'}:=\sup_{\vv\in\Vdiv}\frac{\langle\vf,\vv\rangle_{ V }}{\|\vv\|_V}=\sup_{\vv\in\Vdiv}\frac{\langle\vf,\vv\rangle_{ V }}{\|\vD(\vv)\|_{L^2(\Omega)}},
	\end{equation}
	and $\langle\cdot,\cdot\rangle_{ V }$ denotes the duality pairing  between $V'$ and $V$. The same observation holds for all that follows.}
\end{remark}

\begin{remark} \label{rem:otherapriori}
	{\rm By a direct argument we can also prove that
	\begin{equation} \label{apriori_est:u3bis}
	\|\vD(\vu)\|_{L^2(\Omega)}\leq \frac{1}{\alpha}(\alpha+\gamma \mu_{\rm max})^2C_PC_K\|\vf\|_{L^2(\Omega)}.
	\end{equation}
	This leads to the same \emph{a priori} bound \eqref{apriori_est:Td2} for $\vT^\bd$, 	
	$$
	\|\vT^\bd\|_{L^2(\Omega)}\leq \frac{1}{\alpha}(\alpha+\gamma \mu_{\rm max})C_PC_K\|\vf\|_{L^2(\Omega)}.
	$$
	Indeed, the choice $\vS=\vD(\vu)$ in \eqref{pb:NS_EL_weakR2} gives directly  (without invoking \eqref{eq:symgrad=T})
	$$
	\|\vD(\vu)\|_{L^2(\Omega)}^2= \alpha\int_{\Omega}\vT^\bd:\vD(\vu)+\gamma\int_{\Omega}\mu(|\vT^\bd|)\vT^\bd:\vD(\vu) \le (\alpha+\gamma\mu_{\max}) \|\vT^\bd\|_{L^2(\Omega)} \|\vD(\vu)\|_{L^2(\Omega)}.
	$$
	Hence
	$$\|\vD(\vu)\|_{L^2(\Omega)} \le (\alpha+\gamma\mu_{\max}) \|\vT^\bd\|_{L^2(\Omega)}$$
	and
	$$\|\vD(\vu)\|_{L^2(\Omega)}^2 \le (\alpha+\gamma\mu_{\max})^2 \|\vT^\bd\|_{L^2(\Omega)}^2.$$
	Then \eqref{apriori_est:u3bis} follows by substituting the bound
	$$\alpha\|\vT^\bd\|_{L^2(\Omega)}^2 \le \int_{\Omega}\vD(\vu):\vT^\bd = \int_{\Omega} \vf\cdot \vu \le C_PC_K\|\vf\|_{L^2(\Omega)}\|\vD(\vu)\|_{L^2(\Omega)}$$
	into the preceding inequality. This also yields \eqref{apriori_est:Td2}.}
\end{remark}	


\subsection{Construction of a solution} \label{subsec:existence}

In this subsection we prove the existence of a solution in a bounded Lipschitz domain without any restrictions on the data, other than those stated at the beginning of Section \ref{sec:S_ENL}. The first part of the construction is fairly standard: a suitable sequence of (finite-dimensional) Galerkin approximations to the infinite-dimensional problem is constructed, followed by Brouwer's fixed  point theorem to prove that each finite-dimensional problem in the sequence has a solution; uniform \emph{a priori} estimates, similar to those derived in Lemma \ref{lem:apriori}, are  established for the  Galerkin solutions, which are then used  for passing to the (weak) limit, via a weak compactness argument.  However, because of the combined effect of the nonlinearities,  identifying the limit as a solution to the infinite-dimensional problem requires a more refined argument.

For the sake of clarity, the argument  is split into several steps.

\textbf{Step 1} (Finite-dimensional approximation). Formulation \eqref{eq:reduced3} lends itself readily to a Galerkin discretisation. Since the only unknown is $\vu$ in $\Vdiv$, a separable Hilbert space, we introduce a  countably infinite basis $\left\{\vw_1,\vw_2,\ldots\right\}$ of  orthonormal functions of $\Vdiv$ with respect to the inner product
\begin{equation} \label{def:inner_prod_D}
(\vu,\vv):=\int_{\Omega}\vD(\vu):\vD(\vv),
\end{equation}
whose span is dense in $\Vdiv$. Next, we truncate this basis, i.e., for each $m\geq 1$ we define
 $$\Vdiv_m:=\mbox{span}\{\vw_1,\ldots,\vw_m\},$$
and for $\vu_m\in\Vdiv_m$ we denote by $\hat{\vu}_m\in\mathbb{R}^m$ its representation with respect to this basis. 	
Finally, we fix $m\geq 1$ and consider the following finite-dimensional problem: find $\vu_m\in\Vdiv_m$ such that, for all $1\leq j\leq m$,
\begin{equation} \label{pb:um}
\int_{\Omega}[(\vu_m\cdot\nabla)\vu_m]\cdot\vw_j + \frac{1}{\alpha}\int_{\Omega}\vD(\vu_m):\vD(\vw_j)=\int_{\Omega}\vf\cdot\vw_j+\frac{\gamma}{\alpha}\int_{\Omega}\mu(|\vT_m^\bd|)\vT_m^\bd:\vD(\vw_j)
\end{equation}
with $\vT_m^\bd:=\G(\vu_m)$. In other words, $\vT_m^\bd\in M$ solves
\begin{equation} \label{pb:Tmd}
\alpha\int_{\Omega}\vT_m^\bd:\vS+\gamma\int_{\Omega}\mu(|\vT_m^\bd|)\vT_m^\bd:\vS=\int_{\Omega}\vD(\vu_m):\vS \quad \forall\, \vS\in M.
\end{equation}
Problem \eqref{pb:um}, which can be seen as the \emph{projection} of \eqref{eq:reduced3} onto $\Vdiv_m$, is equivalent to the following: find $\hat{\vu}_m\in\mathbb{R}^m$ such that
\begin{equation*}
\mathbf{F}(\hat{\vu}_m)= \mathbf{0},
\end{equation*}
where $\mathbf{F}=(F_1,\ldots,F_m)^{t}:\mathbb{R}^m\rightarrow\mathbb{R}^m$ is the continuous function defined, for $j=1,\ldots,m$, by
\begin{equation*}
F_j(\hat{\vu}_m):=\int_{\Omega}[(\vu_m\cdot\nabla)\vu_m]\cdot\vw_j + \frac{1}{\alpha}\int_{\Omega}\vD(\vu_m):\vD(\vw_j)-\int_{\Omega}\vf\cdot\vw_j-\frac{\gamma}{\alpha}\int_{\Omega}\mu(|\vT_m^\bd|)\vT_m^\bd:\vD(\vw_j).
\end{equation*}
	
\textbf{Step 2} (Existence of a discrete solution).	Problem \eqref{pb:um} is a system of $m$ nonlinear equations in $m$ unknowns. The existence of a solution to this problem can be established by the following variant of Brouwer's fixed point theorem (see e.g.~\cite{Evans,ref:GiR}).

\begin{lem} \label{lem:zero_in_ball}
	Let $\mathbf{F}:\mathbb{R}^m\rightarrow\mathbb{R}^m$ be a continuous function that satisfies
	\begin{equation*}
	\mathbf{F}(\vx)\cdot\vx\geq 0 \quad \mbox{if} \quad |\vx|=r
	\end{equation*}
	for some $r>0$. Then, there exists  a point $\vx\in B_m(\mathbf{0},r):=\{\vx\in\mathbb{R}^m:\,\, |\vx|\le r\}$ such that
	\begin{equation*}
	\mathbf{F}(\vx)=\mathbf{0}.
	\end{equation*}
\end{lem}

\begin{prop} \label{pro:exist_um}
	Problem \eqref{pb:um} has at least one solution  $\vu_m\in\Vdiv_m$ that satisfies the uniform bound
	\begin{equation} \label{apriori_um}
	\|\vD(\vu_m)\|_{L^2(\Omega)} \leq \alpha C_PC_K\|\vf\|_{L^2(\Omega)}+\gamma C_1 |\Omega|^{\frac{1}{2}}.
	\end{equation}
	Moreover,  $\vT_m^\bd=\G(\vu_m)$ satisfies the uniform bound
	\begin{equation} \label{apriori_Tmd}
	\|\vT_m^\bd\|_{L^2(\Omega)} \leq \frac{1}{\alpha}	\|\vD(\vu_m)\|_{L^2(\Omega)}\leq C_PC_K\|\vf\|_{L^2(\Omega)}+\frac{\gamma}{\alpha} C_1 |\Omega|^{\frac{1}{2}}.
	\end{equation}
\end{prop}

\begin{proof}	
	We infer from Lemma \ref{lem:zero_in_ball} that $\mathbf{F}$ has a zero in the ball $B_m(\mathbf{0},r)$ with
	\begin{equation*}
	r:=\alpha C_PC_K\|\vf\|_{L^2(\Omega)}+\gamma C_1 |\Omega|^{\frac{1}{2}}.
	\end{equation*}
	Indeed, using the antisymmetry property \eqref{eqn:conv_zero}, which holds because  $\vu_m \in \Vdiv_m \subset \Vdiv$, we get
	\begin{align*}
	\mathbf{F}(\hat{\vu}_m)\cdot\hat{\vu}_m & =  \frac{1}{\alpha}\int_{\Omega}|\vD(\vu_m)|^2-\int_{\Omega}\vf\cdot\vu_m-\frac{\gamma}{\alpha}\int_{\Omega}\mu(|\vT_m^\bd|)\vT_m^\bd:\vD(\vu_m) \\
	& \geq  \left(\frac{1}{\alpha}\|\vD(\vu_m)\|_{L^2(\Omega)}-C_PC_K\|\vf\|_{L^2(\Omega)}-\frac{\gamma}{\alpha} C_1 |\Omega|^{\frac{1}{2}}\right)\|\vD(\vu_m)\|_{L^2(\Omega)},
	\end{align*}
	where we have used Poincar\'e's and Korn's inequalities \eqref{def:poincare} and \eqref{def:korn}, respectively, to bound the second term and the relation \eqref{eqn:mu_pos_C1} for the third one. As $\|\vD(\vu_m)\|_{L^2(\Omega)}=|\hat{\vu}_m|$, we deduce from the last inequality that if $|\hat{\vu}_m|=r$ with $r$ as defined above, then
	\begin{equation*}
	\mathbf{F}(\hat{\vu}_m)\cdot\hat{\vu}_m\geq \left(\frac{1}{\alpha}|\hat{\vu}_m|-C_PC_K\|\vf\|_{L^2(\Omega)}-\frac{\gamma}{\alpha} C_1 |\Omega|^{\frac{1}{2}}\right)|\hat{\vu}_m|=0.
	\end{equation*}
	Thanks to Lemma \ref{lem:zero_in_ball}, there exists  a point  $\hat{\vu}_m\in B_m(\mathbf{0},r)$ such that $\mathbf{F}(\hat{\vu}_m)= {\bf 0}$, i.e., problem \eqref{pb:um} has a solution $\vu_m\in\Vdiv_m$ that satisfies the uniform bound \eqref{apriori_um}. Finally, it is easily shown that $\vT_m^\bd:=\G(\vu_m)$ satisfies the bound \eqref{apriori_Tmd}.
\end{proof}

\textbf{Step 3} (Passage to the limit $m\rightarrow\infty$ and identification of the limit). We consider the sequences $(\vu_m)_{m\geq 1}$ and $(\vT_m^\bd)_{m\geq 1}$ with $\vu_m\in\Vdiv_m$ and  $\vT_m^\bd=\G(\vu_m)\in M$. Thanks to the uniform  estimates \eqref{apriori_um} and \eqref{apriori_Tmd} there exist two subsequences (not relabelled) such that
\begin{alignat*}{2}
 \lim_{m \to \infty} \vu_m & =\bar{\vu} & & \quad\mbox{weakly in } H_0^1(\Omega)^d \,\, (\mbox{and thus also in }\Vdiv), \\
 \lim_{m \to \infty} \vu_m & =\bar{\vu} & & \quad\mbox{strongly in } L^q(\Omega)^d\,\,  \mbox{with }1\leq q < \infty\; \mbox{ if }\, d=2, \mbox{ and }1 \leq q < 6\; \mbox{ if }\, d=3,  \\
 \lim_{m \to \infty} \vT_m^\bd & =\bar{\vT}^\bd & & \quad \mbox{weakly in } L^2(\Omega)^{d\times d} \,\, (\mbox{and thus also in }M),
\end{alignat*}
for some $\bar{\vu}\in\Vdiv$ and $\bar{\vT}^\bd\in M$. Our objective is to show that the  pair  $(\bar{\vT}^\bd,\bar{\vu}) \in M\times \Vdiv$  is a solution to the problem under consideration by passing to the limit in \eqref{pb:um}, \eqref{pb:Tmd}.

Passing  to the limit in \eqref{pb:um}, \eqref{pb:Tmd} is however not straightforward because of the lack of strong convergence of $\vT_m^\bd$ in $M$.  Identifying the pair  $(\bar{\vT}^\bd,\bar{\vu}) \in M\times \Vdiv$ as a solution will be achieved by means of the following two lemmas, the first of which (Lemma \ref{lem:conv_product}) relies on the equations and the strong convergence of the sequence $(\vu_m)_{m\geq 1}$ in $L^q(\Omega)^d$ shown above, and the second lemma (Lemma \ref{lem:T_to_G(u)}) follows from the monotonicity property \eqref{eqn:mu_mon}.

The proof, included below, that the pair $(\bar{\vT}^\bd,\bar{\vu})$ satisfies \eqref{pb:NS_EL_weakR2}  is inspired by the arguments in \cite{BGMRS12}, where a more general constitutive relation than \eqref{eqn:constitutive_rel} was considered. Specifically, the conclusion of Lemma \ref{lem:T_to_G(u)} follows from \cite[Lemma 2.4.1]{BGMRS12}, the hypothesis (2.12) of \cite[Lemma 2.4.1]{BGMRS12} being fulfilled thanks to Lemma \ref{lem:conv_product}; however we provide a proof here that is directly tailored to our problem.

\begin{lem} \label{lem:conv_product}
	The following limit holds:
	\begin{equation} \label{eqn:lim_sup}
	\lim_{m\rightarrow\infty}\int_{\Omega}\vT_m^\bd:\vD(\vu_m)= \int_{\Omega}\bar{\vT}^\bd:\vD(\bar{\vu}).
	\end{equation}
\end{lem}

\begin{proof}
	By testing equation \eqref{pb:Tmd} with $\vS=\vD(\vw_j)$ and substituting into \eqref{pb:um} we deduce that
	\begin{equation} \label{eqn:um_wj}
	\int_{\Omega}[(\vu_m\cdot\nabla)\vu_m]\cdot\vw_j + \int_{\Omega}\vT_m^\bd:\vD(\vw_j) =\int_{\Omega}\vf\cdot\vw_j \quad \forall\, 1\leq j\leq m.
	\end{equation}
	Multiplying \eqref{eqn:um_wj} by $(\hat{\vu}_m)_j$, summing over $j$, and applying \eqref{eqn:conv_zero}, we derive
	\begin{equation} \label{eqn:Tmd_um}
	\int_{\Omega}\vT_m^\bd:\vD(\vu_m) =\int_{\Omega}\vf\cdot\vu_m.
	\end{equation}
	Thus we obtain on the one hand
	\begin{equation} \label{eq:lim_product1}
	\lim_{m\rightarrow\infty}\int_{\Omega}\vT_m^\bd:\vD(\vu_m) = \lim_{m\rightarrow\infty}\int_{\Omega}\vf\cdot\vu_m = \int_{\Omega}\vf\cdot\bar{\vu}.
	\end{equation}	
	On the other hand, letting $m$ tend to infinity in \eqref{eqn:um_wj} for fixed $j$ and considering the strong convergence of  $\vu_m$, we infer that
	\begin{equation*}
	\int_{\Omega}[(\bar{\vu}\cdot\nabla)\bar{\vu}]\cdot\vw_j + \int_{\Omega}\bar{\vT}^\bd:\vD(\vw_j) =\int_{\Omega}\vf\cdot\vw_j \quad \forall\, j\geq 1,
	\end{equation*}
	and the density of $\bigcup_{m \geq 1}\Vdiv_m$ in $\Vdiv$ therefore implies that
	\begin{equation} \label{eqn:Td_u}
	\int_{\Omega}[(\bar{\vu}\cdot\nabla)\bar{\vu}]\cdot\vv + \int_{\Omega}\bar{\vT}^\bd:\vD(\vv) =\int_{\Omega}\vf\cdot\vv \quad \forall\, \vv\in\Vdiv.
	\end{equation}
	In view of \eqref{eqn:conv_zero}, the choice $\vv = \bar{\vu}$ in \eqref{eqn:Td_u} yields
	\begin{equation} \label{eq:product2}
	\int_{\Omega}\bar{\vT}^\bd:\vD(\bar{\vu}) =\int_{\Omega}\vf\cdot\bar{\vu},
	\end{equation}
	and \eqref{eqn:lim_sup}	then follows from \eqref{eq:lim_product1} and \eqref{eq:product2}.
\end{proof}

\begin{lem} \label{lem:T_to_G(u)}
	We have that
	\begin{equation} \label{eqn:barT_G(baru)}
	\bar\vT^\bd= \G(\bar{\vu}).
	\end{equation}
\end{lem}

\begin{proof}
	Let $\tilde{\vT}^\bd:=\G(\bar{\vu})$; since $\vT_m^\bd:=\G(\vu_m)$, we have by definition
	\begin{equation*}
	\alpha\int_{\Omega}(\vT_m^\bd-\tilde{\vT}^\bd):\vS+\gamma\int_{\Omega}(\mu(|\vT_m^\bd|)\vT_m^\bd-\mu(|{\tilde \vT}^{\bd}|)\tilde{\vT}^\bd):\vS=\int_{\Omega}(\vD(\vu_m-\bar{\vu})):\vS
	\end{equation*}
	for all $\vS\in M$. Taking then $\vS=\vT_m^\bd-\tilde{\vT}^\bd$ and using the monotonicity property \eqref{eqn:mu_mon} we get
	\begin{align*}
	\alpha\int_{\Omega}|\vT_m^\bd-\tilde{\vT}^\bd|^2 & \leq  \int_{\Omega}(\vD(\vu_m-\bar{\vu})):(\vT_m^\bd-\tilde{\vT}^\bd) \\
	& =  \int_{\Omega}\left[\vD(\vu_m):\vT_m^\bd-\vD(\vu_m):\tilde{\vT}^\bd-\vD(\bar{\vu}):(\vT_m^\bd-\tilde{\vT}^\bd)\right].
	\end{align*}
	Finally, we take the limit $m\rightarrow\infty$ of both sides and apply  \eqref{eqn:lim_sup} to obtain
	\begin{equation*}
	\alpha\lim_{m\rightarrow\infty}\int_{\Omega}|\vT_m^\bd-\tilde{\vT}^\bd|^2 \leq \int_{\Omega}\left[\vD(\bar{\vu}):\bar{\vT}^\bd-\vD(\bar{\vu}):\tilde{\vT}^\bd-\vD(\bar{\vu}):(\bar{\vT}^\bd-\tilde{\vT}^\bd)\right]=0,
	\end{equation*}
	which implies \eqref{eqn:barT_G(baru)} as well as the \emph{strong} convergence in $M$ of $\vT_m^\bd$ to $\bar{\vT}^\bd$.
\end{proof}

\begin{thm}(Existence of a solution) \label{thm:existence1}
	The pair  $(\bar\vT^\bd,\bar\vu)\in M\times\Vdiv$ solves
	\eqref{pb:NS_EL_weakR1}, \eqref{pb:NS_EL_weakR2}.
\end{thm}

\begin{proof}
	It follows from Lemma \ref{lem:T_to_G(u)} that on the one hand $(\bar\vT^\bd,\bar\vu)$ solves \eqref{pb:NS_EL_weakR2} and on the other hand,
	$$ \lim_{m \to \infty}\mu(|\vT_m^\bd|)\vT_m^\bd = \mu(|\bar{\vT}^\bd|)\bar{\vT}^\bd\ \mbox{weakly in } M.
	$$
	Indeed, passing to the limit in \eqref{pb:Tmd} gives, for any $\vS\in M$,
	\begin{align*}
	\lim_{m\rightarrow\infty}\gamma\int_{\Omega}\mu(|\vT_m^\bd|)\vT_m^\bd:\vS & =  \lim_{m\rightarrow\infty}\left( \int_{\Omega}\vD(\vu_m):\vS-\alpha\int_{\Omega}\vT_m^\bd:\vS \right) \\
	& =   \int_{\Omega}\vD(\bar{\vu}):\vS-\alpha\int_{\Omega}\bar{\vT}^\bd:\vS \\
	& =  \gamma\int_{\Omega}\mu(|\bar{\vT}^\bd|)\bar{\vT}^\bd:\vS.
	\end{align*}
	Therefore, taking the limit as $m\rightarrow\infty$ in \eqref{pb:um} we get
	\begin{equation*}
	\int_{\Omega}[(\bar{\vu}\cdot\nabla)\bar{\vu}]\cdot\vw_j + \frac{1}{\alpha}\int_{\Omega}\vD(\bar{\vu}):\vD(\vw_j)=\int_{\Omega}\vf\cdot\vw_j+\frac{\gamma}{\alpha}\int_{\Omega}\mu(|\bar{\vT}^\bd|)\bar{\vT}^\bd:\vD(\vw_j)
	\end{equation*}
	for each $j=1,2,\ldots$, and thus the density of $\bigcup_{m \geq 1}\Vdiv_m$ in $\Vdiv$ implies that	 
	\begin{equation*}
	\int_{\Omega}[(\bar{\vu}\cdot\nabla)\bar{\vu}]\cdot\vv + \frac{1}{\alpha}\int_{\Omega}\vD(\bar{\vu}):\vD(\vv)=\int_{\Omega}\vf\cdot\vv+\frac{\gamma}{\alpha}\int_{\Omega}\mu(|\bar{\vT}^\bd|)\bar{\vT}^\bd:\vD(\vv) \quad \forall\,  \vv\in \Vdiv,
	\end{equation*}
	which is precisely \eqref{pb:NS_EL_weakR1}.
\end{proof}


\subsection{Global conditional uniqueness} \label{subsec:globaluniqu}

We now prove global uniqueness of the solution under additional assumptions on the function $\mu$ and the input data.
The notion of uniqueness we establish is global and conditional in the sense that
it holds under suitable restrictions on the data, but it is also global because
no other solution exists.

Let $\mathbb{R}_{\rm{sym},0}^{d\times d}$  denote the space of symmetric $d \times d$ matrices with vanishing trace and let $C_S$ be the smallest positive constant in the following Sobolev
embedding:
\begin{equation} \label{def:SobolevL4}
\|\vv\|_{L^4(\Omega)}\leq C_S \|\nabla \vv\|_{L^2(\Omega)} \quad \forall\, \vv\in V.
\end{equation}

\begin{prop}(Uniqueness) \label{lem:uniqueness}
	Assume that the function $\vT^\bd\mapsto\mu(|\vT^\bd|)\vT^\bd$ is Lipschitz continuous in $\mathbb{R}_{\rm{sym},0}^{d\times d}$, i.e., there exists a positive constant $\Lambda$ such that
	\begin{equation} \label{eqn:mu_Lip}
	|\mu(|\vT^\bd|)\vT^\bd-\mu(\vS^\bd)\vS^\bd|\leq \Lambda |\vT^\bd-\vS^\bd| \quad \forall\, \vT^\bd,\vS^\bd\in\mathbb{R}_{\rm{sym},0}^{d\times d}.
	\end{equation}
  	If the input data satisfy
	\begin{equation} \label{hyp:input_data}
	\frac{\gamma}{\alpha}\Lambda+\alpha^2C_S^2C_PC_K^4\|\vf\|_{L^2(\Omega)}+\alpha\gamma C_S^2C_K^3 C_1 |\Omega|^{\frac{1}{2}}<1
	\end{equation}
	then the solution of problem \eqref{pb:NS_EL_weakR1}, \eqref{pb:NS_EL_weakR2} is unique.
\end{prop}

\begin{proof}
	We use a variational argument. Suppose that $(\vT_1^\bd,\vu_1),(\vT^\bd_2,\vu_2)\in M\times \Vdiv$ are solutions of \eqref{pb:NS_EL_weakR1}, \eqref{pb:NS_EL_weakR2}. Let us write $\delta\vT^\bd:=\vT_1^\bd-\vT_2^\bd$ and $\delta\vu:=\vu_1-\vu_2$. Subtracting the equations solved by $(\vT^\bd_2,\vu_2)$
	from those solved by $(\vT^\bd_1,\vu_1)$ we get for all $(\vS,\vv)\in M\times\Vdiv$ the following pair of equalities:
	\begin{align}
	\int_{\Omega}\left[(\vu_1\cdot\nabla)\vu_1-(\vu_2\cdot\nabla)\vu_2\right]\cdot\vv+\frac{1}{\alpha}\int_{\Omega}\vD(\delta\vu):\vD(\vv) & =  \frac{\gamma}{\alpha}\int_{\Omega}\big(\mu(|\vT_1^\bd|)\vT_1^\bd-\mu(|\vT_2^\bd|)\vT_2^\bd\big):\vD(\vv),\label{eqn:deltau} \\
	\alpha\int_{\Omega}\delta\vT^\bd:\vS+\gamma\int_{\Omega}\big(\mu(|\vT_1^\bd|)\vT_1^\bd-\mu(|\vT_2^\bd|)\vT_2^\bd\big):\vS & =  \int_{\Omega}\vD(\delta\vu):\vS. \label{eqn:deltaT}
	\end{align}
	The choice $\vS=\delta\vT^\bd$ in \eqref{eqn:deltaT}, thanks to the monotonicity property \eqref{eqn:mu_mon}, leads to
	\begin{equation} \label{eqn:bound_deltaT}
	\|\delta\vT^\bd\|_{L^2(\Omega)}\leq \frac{1}{\alpha}\|\vD(\delta\vu)\|_{L^2(\Omega)}.
	\end{equation}
	Then, by noting that
	\begin{equation*}
	\int_{\Omega}\left[(\vu_1\cdot\nabla)\vu_1-(\vu_2\cdot\nabla)\vu_2\right]\cdot\vv = \int_{\Omega}\left[(\delta\vu\cdot\nabla)\vu_1\right]\cdot\vv+\int_{\Omega}\left[(\vu_2\cdot\nabla)\delta\vu\right]\cdot\vv,
	\end{equation*}
	by testing \eqref{eqn:deltau} with  $\vv=\delta\vu$, and recalling \eqref{eqn:conv_zero} we obtain
	\begin{eqnarray*}
	\frac{1}{\alpha}\|\vD(\delta\vu)\|_{L^2(\Omega)}^2 & = &  \frac{\gamma}{\alpha}\int_{\Omega}\big(\mu(|\vT_1^\bd|)\vT_1^\bd-\mu(|\vT_2^\bd|)\vT_2^\bd\big):\vD(\delta\vu)-\int_{\Omega}[(\delta\vu\cdot\nabla)\vu_1]\cdot\delta\vu \\
	& \stackrel{\eqref{eqn:mu_Lip}}{\leq} & \frac{\gamma}{\alpha}\Lambda\|\delta\vT^\bd\|_{L^2(\Omega)}\|\vD(\delta\vu)\|_{L^2(\Omega)}+\|\delta\vu\|_{L^4(\Omega)}^2\|\nabla\vu_1\|_{L^2(\Omega)} \\
	& \stackrel{\eqref{def:SobolevL4},\!~\eqref{def:korn}}{\leq} & \frac{\gamma}{\alpha}\Lambda\|\delta\vT^\bd\|_{L^2(\Omega)}\|\vD(\delta\vu)\|_{L^2(\Omega)}+C_S^2C_K^3\|\vD(\vu_1)\|_{L^2(\Omega)}\|\vD(\delta\vu)\|_{L^2(\Omega)}^2 \\
	& \stackrel{\eqref{eqn:bound_deltaT},\!~\eqref{apriori_est:u}}{\leq} & \left[\frac{\gamma}{\alpha^2}\Lambda+C_S^2C_K^3\left(\alpha C_PC_K\|\vf\|_{L^2(\Omega)}+\gamma C_1 |\Omega|^{\frac{1}{2}}\right)\right]\|\vD(\delta\vu)\|_{L^2(\Omega)}^2. \\
	\end{eqnarray*}
	The assumption \eqref{hyp:input_data} on the data guarantees that the factor
	on the right-hand side of the last inequality is strictly smaller than
	$\frac{1}{\alpha}$, thus implying that  $\|\vD(\delta\vu)\|_{L^2(\Omega)}=0$, i.e., $\vu_1=\vu_2$. Finally, applying this result to \eqref{eqn:bound_deltaT} yields $\vT_1^\bd=\vT_2^\bd$.
\end{proof}

\begin{remark} \label{rem:rem3}
	{\rm The strategy used in deriving the second \emph{a priori} estimate stated in
	Lemma \ref{lem:apriori2} leads to uniqueness when  \eqref{hyp:input_data}
	is replaced by
	\begin{equation} \label{hyp:input_data2}
	\frac{\gamma}{\alpha}\Lambda+\alpha C_S^2C_PC_K^4(\alpha+\gamma\mu_{\max})\|\vf\|_{L^2(\Omega)}<1.
	\end{equation}
	In fact both strategies lead to the same condition \eqref{hyp:input_data2}; namely, we also get \eqref{hyp:input_data2} by using \eqref{apriori_est:u2} instead of \eqref{apriori_est:u} to bound $\|\vD(\vu_1)\|_{L^2(\Omega)}$ in the proof of Proposition \ref{lem:uniqueness}.}
\end{remark}
Note that both \eqref{hyp:input_data} and \eqref{hyp:input_data2} hold when $\gamma$ and $\vf$ are sufficiently small.

\begin{remark} \label{rem:rem3bis}
	{\rm Under the Lipschitz condition \eqref{eqn:mu_Lip}, the proof of \eqref{apriori_est:u3bis} and \eqref{apriori_est:Td2} is valid with $\mu_{{\rm \max}}$ replaced by $\Lambda$, and the $L^2(\Omega)^d$ norm of $\vf$ (multiplied by $C_P C_K$) replaced by its norm in $\Vdiv'$,  see Remark \ref{rem:1}. More precisely,
	\begin{equation} \label{apriori_est:u3}
	\|\vD(\vu)\|_{L^2(\Omega)}\leq \frac{1}{\alpha}(\alpha+\gamma\Lambda)^2\|\vf\|_{\Vdiv'},
	\end{equation}
	\begin{equation} \label{apriori_est:Td3}
	\|\vT^\bd\|_{L^2(\Omega)}\leq \frac{1}{\alpha}(\alpha+\gamma\Lambda)\|\vf\|_{\Vdiv'}.
	\end{equation}}
\end{remark}		

	
\subsection{Comparison of the \emph{a priori} estimates} \label{subsec:comp-apriori}

At this stage, it is useful to compare the \emph{a priori} estimates derived in the previous sections. We have

\begin{equation} \label{eq:C_u}
\|\vD(\vu)\|_{L^2(\Omega)}  \leq  C_{\vu}:=\min\left\{(\alpha+\gamma\mu_{\max})\|\vf\|_{\Vdiv'},\alpha\|\vf\|_{\Vdiv'}+\gamma C_1 |\Omega|^{\frac{1}{2}},\frac{1}{\alpha}(\alpha+\gamma\Lambda)^2\|\vf\|_{\Vdiv'}\right\},
\end{equation}
\begin{equation} \label{eq:C_T}
\|\vT^\bd\|_{L^2(\Omega)}  \leq  C_{\vT^\bd}:=\min\left\{\frac{1}{\alpha}(\alpha+\gamma\mu_{ \max})\|\vf\|_{\Vdiv'},\|\vf\|_{\Vdiv'}+\frac{\gamma}{\alpha} C_1 |\Omega|^{\frac{1}{2}},\frac{1}{\alpha}(\alpha+\gamma\Lambda)\|\vf\|_{\Vdiv'}\right\},
\end{equation}
where  $\Lambda$ is replaced by $\mu_{ \max}$ if we do not make the Lipschitz assumption \eqref{eqn:mu_Lip}. For $p$ we have
\begin{equation*}
\|p\|_{L^2(\Omega)} \leq \frac{1}{\beta}\left(C_S^2C_K^2C_{\vu}^2+\|\vf\|_{(\Vdiv^{\perp})'}+\min\left\{\min\left(1,\frac{\gamma\mu_{\max}}{\alpha},\frac{\gamma\Lambda}{\alpha}\right)C_{\vT^\bd},\frac{\gamma}{\alpha} C_1 |\Omega|^{\frac{1}{2}}\right\}\right),
\end{equation*}
where $\Vdiv^{\perp}$ denotes the orthogonal complement of $\Vdiv$ in $V$ with respect to the inner product \eqref{def:inner_prod_D}.

\begin{remark} \label{rem:5}
	{\rm We can replace $C_S^2$ by the product $C_pC_r$  of the smallest constants $C_p$ and $C_r$ from the Sobolev embedding of $H^1(\Omega)^d$ into $L^p(\Omega)^d$ and $L^r(\Omega)^d$, respectively, with $p=6$ and $r=3$.  We could also use the best constant $\hat{C}$ such that
	$$\int_{\Omega}[(\vu\cdot\nabla)\vv]\cdot\vw\leq \hat{C} \|\nabla\vu\|_{L^2(\Omega)}\|\nabla\vv\|_{L^2(\Omega)}\|\nabla\vw\|_{L^2(\Omega)},$$
	or even
	$$\int_{\Omega}[(\vu\cdot\nabla)\vv]\cdot\vw\leq \hat{C} \|\vD(\vu)\|_{L^2(\Omega)}\|\vD(\vv)\|_{L^2(\Omega)}\|\vD(\vw)\|_{L^2(\Omega)}.$$
	In the former case, $\hat{C}\leq C_pC_r$ while in the latter case, $\hat{C}\leq C_K^3C_pC_r$.}
\end{remark}

\section{ Conforming finite element approximation} \label{sec:NS_ENL_App}

In this section, we study conforming finite element approximations of problem \eqref{pb:NS_EL_v2}, where conformity refers to the discrete velocity space. To facilitate the implementation, it is useful to relax the zero trace restriction on the discrete tensor space, but this is not quite a nonconformity because the theoretical analysis of the preceding sections holds without this condition. In particular, the inf-sup condition \eqref{infsup:b1} is still valid (supremum over a larger space).

We start with the numerical analysis of general conforming approximations, including existence of discrete solutions, convergence, and error estimates, and give specific examples further on.

	
\subsection{General conforming approximation} \label{subsec:Gen_conf_app}

As stated above, here $M=L^2(\Omega)_{{\rm sym}}^{d\times d}$. Up to this modification, we propose to discretise the formulation  derived from \eqref{pb:NS_EL_v2}: find $(\vT^\bd,\vu,p)\in M\times V \times Q$ such that
\begin{alignat}{2} \label{pb:weak_cont}
\begin{aligned}
\displaystyle{\int_{\Omega}}\left[(\vu\cdot\nabla)\vu\right]\cdot\vv+\displaystyle{\int_{\Omega}}\vT^\bd:\vD(\vv)+b_2(\vv,p) & =  \displaystyle{\int_{\Omega}}\vf\cdot\vv &&\quad \forall\, \vv\in V, \\
\alpha\displaystyle{\int_{\Omega}}\vT^\bd:\vS+\gamma\displaystyle{\int_{\Omega}}\mu(|\vT^\bd|)\vT^\bd:\vS & =  \displaystyle{\int_{\Omega}}\vD(\vu):\vS &&\quad \forall\, \vS\in M, \\
b_2(\vu,q) & =  0 &&\quad \forall\, q\in Q.
\end{aligned}
\end{alignat}
Note that, since ${\rm div}(\vu) =0$, by taking $\vS = \vI$ the second line of \eqref{pb:weak_cont} implies that the solution $\vT^\bd$ of \eqref{pb:weak_cont} satisfies $\tr(\vT^\bd)=0$ a.e. in $\Omega$,  even though this condition was not explicitly imposed on elements of $M$.

Let $h >0$ be a discretisation parameter that will tend to zero and, for each $h$, let $V_h\subset V$, $Q_h\subset Q$ and $M_h\subset M$ be three finite-dimensional spaces satisfying the following basic approximation properties, for all $\vS \in M$, $\vv \in V$ and $q \in Q$:
$$\lim_{h \to 0} \inf_{\vS_h \in M_h}\|\vS_h-\vS\|_{L^2(\Omega)} = 0,\quad  \lim_{h \to 0} \inf_{\vv_h \in V_h}\|\vD(\vv_h-\vv)\|_{L^2(\Omega)} =0, \quad \lim_{h \to 0} \inf_{q_h \in Q_h}\|q_h-q\|_{L^2(\Omega)} =0.
$$
Moreover, let
\begin{equation} \label{def:Vh0}
V_{h,0}:=\{\vv_h\in V_h: \quad b_2(\vv_h,q_h)=0 \quad \forall\, q_h\in Q_h\}.
\end{equation}
We assume  on the one hand that  the pair $(V_h,Q_h)$ is uniformly stable for the divergence, i.e.,
\begin{equation} \label{infsup:b1h}
\inf_{q_h\in Q_h}\sup_{\vv_h\in V_h}\frac{b_2(\vv_h,q_h)}{\|\vD(\vv_h)\|_{L^2(\Omega)}\|q_h\|_{L^2(\Omega)}}\geq \beta^{*}
\end{equation}
for some constant $\beta^{*} >0$, independent of $h$, and  on the other hand that $M_h$ and $V_{h,0}$ are compatible in the sense that
\begin{equation} \label{eqn:eps_in_Mh}
\vD(\vv_h)\in M_h \quad \forall\, \vv_h\in V_{h,0}.
\end{equation}
Note that the latter assumption may be prohibitive when considering conforming finite elements on quadrilateral ($d=2$) or hexahedral ($d=3$) meshes, see Subsection \ref{subsec:Ex_conf_app}; this motivates the study of non-conforming finite elements considered in Section \ref{sec:NS_ENL_AppNC}.
The inf-sup condition \eqref{infsup:b1h} guarantees that
\begin{equation} \label{eq:appVh0}
\lim_{h \to 0} \inf_{\vv_h \in V_{h,0}}\|\vD(\vv_h-\vu)\|_{L^2(\Omega)} =0.
\end{equation}
Indeed, \eqref{infsup:b1h} implies  the relation
\begin{equation} \label{eq:bdd_vh0}
\inf_{\vv_{h,0}\in V_{h,0}}\| \vD(\vu-\vv_{h,0})\|_{ L^2(\Omega)} \leq \left(1+\frac{c_b}{\beta^*}\right)\inf_{\vv_h\in V_h}\| \vD(\vu-\vv_h)\|_{ L^2(\Omega)},
\end{equation}
which can be shown using a standard argument; see for instance~\cite{ref:GiR}. Here, $c_b$ denotes the continuity constant of $b_2(\cdot,\cdot)$ on $V\times Q$.

As the divergence of functions of $V_{h,0}$ is not necessarily zero, the antisymmetry property
\eqref{eqn:conv_zero} does not hold in the discrete spaces. Since this property is
a crucial ingredient in the analysis of our problem, it is standard (see for instance~\cite{Tem,ref:GiR}) to introduce the trilinear form
$d:V\times V\times V\rightarrow\mathbb{R}$ defined by
\begin{equation} \label{def:form_d}
d(\vu;\vv,\vw):= \frac{1}{2}\int_{\Omega}\left[(\vu\cdot\nabla)\vv\right]\cdot\vw-\frac{1}{2}\int_{\Omega}\left[(\vu\cdot\nabla)\vw\right]\cdot\vv.
\end{equation}
The trilinear form $d$ is obviously antisymmetric and it is consistent thanks to the fact that
\begin{equation*}
d(\vu;\vv,\vw)=\int_{\Omega}\left[(\vu\cdot\nabla)\vv\right]\cdot\vw \quad \forall\, \vu\in \Vdiv,\; \forall\, \vv,\vw\in V.
\end{equation*}
Moreover, a standard computation shows that there exists a constant $\hat{D}\leq \min(C_S^2,C_3C_6)C_K^3$ such that
\begin{equation} \label{bound:form_d}
d(\vu;\vv,\vw)\leq \hat{D}\| \vD(\vu)\|_{ L^2(\Omega)}\| \vD(\vv)\|_{ L^2(\Omega)}\| \vD(\vw)\|_{ L^2(\Omega)} \quad \forall\, \vu,\vv,\vw\in V.
\end{equation}
We then consider the following approximation of problem \eqref{pb:weak_cont}: find $(\vT_h,\vu_h,p_h)\in M_h\times V_h\times Q_h$ such that

\begin{alignat}{2} \label{pb:weak_disc}
\begin{aligned}
d(\vu_h;\vu_h,\vv_h)+\displaystyle{\int_{\Omega}}\vT_h:\vD(\vv_h)+b_2(\vv_h,p_h) & =  \displaystyle{\int_{\Omega}}\vf\cdot\vv_h &&\quad \forall\, \vv_h\in V_h, \\
\alpha\displaystyle{\int_{\Omega}}\vT_h:\vS_h+\gamma\displaystyle{\int_{\Omega}}\mu(|\vT_h|)\vT_h:\vS_h & =  \displaystyle{\int_{\Omega}}\vD(\vu_h):\vS_h &&\quad \forall\, \vS_h\in M_h, \\
b_2(\vu_h,q_h) & =  0 &&\quad \forall\, q_h\in Q_h.
\end{aligned}
\end{alignat}


\subsubsection{Existence  of a discrete solution}

Existence of a solution to problem \eqref{pb:weak_disc} without restrictions on the data is established by Brouwer's fixed point theorem, as in Section \ref{subsec:existence}. To begin with, for any function $\vv \in V$, we define the discrete analogue of the mapping $\G$, see \eqref{eqn:def_map_G}; namely, $\G_h(\vv) \in M_h$ is the unique solution of
\begin{equation} \label{eqn:def_map_Gh}
\alpha\int_{\Omega}\G_h(\vv):\vS_h+\gamma\int_{\Omega}\mu(|\G_h(\vv)|)\G_h(\vv):\vS_h
=\int_{\Omega}\vD(\vv):\vS_h\quad \forall\, \vS_h\in M_h.
\end{equation}
This finite-dimensional square system has one and only one solution $\G_h(\vv)$ thanks to the properties of the left-hand side: the first term is elliptic and the second term is monotone. As in  Section \ref{subsec:existence}, in view of the inf-sup condition \eqref{infsup:b1h}, problem \eqref{pb:weak_disc} is equivalent to finding $\vu_h \in V_{h,0}$ solution of
\begin{equation} \label{eq:reduce-discrete}
d(\vu_h;\vu_h,\vv_h) + \int_{\Omega}\vT_h: \vD({\vv_h}) = \int_{\Omega}\vf\cdot\vv_h\quad  \forall\, \vv_h\in V_{h,0},
\end{equation}
where $\vT_h:=\G_h(\vu_h)$.
By proceeding as in Proposition \ref{pro:exist_um}, it is easy to prove that problem \eqref{eq:reduce-discrete} has at least one solution $\vu_h \in V_{h,0}$, and by the above equivalence, each solution $\vu_h$ determines a  pair $(\vT_h,p_h) \in M_h\times Q_h$ so that $(\vT_h,\vu_h,p_h)$ solves problem \eqref{pb:weak_disc}. Moreover, each solution of problem \eqref{pb:weak_disc} satisfies the same estimates as in \eqref{apriori_est:u} and \eqref{apriori_est:Td}. For the sake of simplicity, since the approximation is conforming, we state them in terms of the norm of $\vf$ in $H^{-1}(\Omega)^d$,
\begin{equation} \label{apriori_est:uh}
\|\vD(\vu_h)\|_{L^2(\Omega)} \leq \alpha \|\vf\|_{H^{-1}(\Omega)}+\gamma C_1 |\Omega|^{\frac{1}{2}}
\end{equation}
and
\begin{equation} \label{apriori_est:Tdh}
\|\vT_h\|_{L^2(\Omega)} \leq \|\vf\|_{H^{-1}(\Omega)}+\frac{\gamma}{\alpha} C_1 |\Omega|^{\frac{1}{2}}.
\end{equation}
Regarding the other \emph{a priori} bounds, \eqref{apriori_est:u3bis} and  \eqref{apriori_est:Td2} are satisfied by  $\vu_h$ and $\vT_h$ and, if \eqref{eqn:mu_Lip} holds,  so are \eqref{apriori_est:Td3} and  \eqref{apriori_est:u3}, all up to the above norm for $\vf$. In contrast, however, we do not have enough information to claim that  \eqref{apriori_est:u2} is valid because it relies on the nonnegativity of $\vT_h: \vD(\vu_h)$ almost everywhere in $\Omega$; the integral average is positive but this does not always guarantee pointwise nonnegativity. Thus we replace the constant $C_{\vu}$ of \eqref{eq:C_u} by the constant $\widetilde{C_{\vu}}$ in the following inequality:
\begin{equation} \label{eq:C_uh}
\|\vD(\vu_h)\|_{L^2(\Omega)} \leq \widetilde{C_{\vu}} : = \min\left\{\frac{1}{\alpha}(\alpha+\gamma\mu_{\max})^2\|\vf\|_{H^{-1}(\Omega)},\alpha\|\vf\|_{H^{-1}(\Omega)}+\gamma C_1 |\Omega|^{\frac{1}{2}},\frac{1}{\alpha}(\alpha+\gamma\Lambda)^2\|\vf\|_{H^{-1}(\Omega)}\right\},
\end{equation}
where the last  term  is included when \eqref{eqn:mu_Lip} holds. Because $C_{\vu} \le \widetilde{C_{\vu}}$, we shall use $\widetilde{C_{\vu}}$ to bound both $\vu$ and $\vu_h$ in order to simplify the constants in the computations that will now follow.


Finally, let us establish the convergence of the sequence of discrete solutions in the limit of $h \rightarrow 0$. The above uniform \emph{a priori} estimates imply that, up to a subsequence of the discretisation parameter $h$,
\begin{alignat*}{2}
 \lim_{h \to 0} \vu_h & =\bar{\vu} & & \quad\mbox{weakly in } H_0^1(\Omega)^d ,\\
 \lim_{h \to 0} \vu_h & =\bar{\vu} & & \quad\mbox{strongly in } L^q(\Omega)^d\,\,  \mbox{with }1\leq q < \infty\; \mbox{ if }\, d=2, \mbox{ and }1 \leq q < 6\; \mbox{ if }\, d=3,  \\
 \lim_{h \to 0} \vT_h & =\bar{\vT} & & \quad \mbox{weakly in } L^2(\Omega)^{d \times d},
\end{alignat*}
for some $\bar{\vu}\in H_0^1(\Omega)^d$ and $\bar{\vT}\in L^2(\Omega)^{d \times d}$. Clearly, the symmetry of $\vT_h$ implies that of $\bar{\vT}$ and ${\rm div}(\bar\vu) = 0 $ follows from the fact that $\vu_h$ belongs to $V_{h,0}$. Then the approximation properties of the discrete spaces and  \eqref{eq:appVh0} permit to replicate the steps of the proof of Lemma \ref{lem:conv_product} and yield
\begin{equation} \label{eqn:lim_suph}
\lim_{h\rightarrow 0}\int_{\Omega}\vT_h:\vD(\vu_h)= \int_{\Omega}\bar{\vT}:\vD(\bar{\vu}).
\end{equation}
To fully identify the limit, in addition to $\tilde\vT^\bd := \G(\bar{\vu})$,  which has trace zero since $\di(\bar \vu)=0$, we introduce the auxiliary tensor $\tilde\vT_h := \G_h(\bar{\vu})$. On the one hand
$$\alpha \int_\Omega (\tilde\vT_h-\tilde\vT^\bd):\vS_h + \gamma\int_{\Omega}\big(\mu(|\tilde\vT_h|)\tilde\vT_h -\mu(|\tilde\vT^\bd|)\tilde\vT^\bd\big) :\vS_h = 0 \quad \forall\, \vS_h \in M_h, $$
thus implying that, for all $\vS_h$ in $M_h$,
\begin{align*}
\alpha\|\tilde\vT_h-&\tilde\vT^\bd\|^2_{ L^2(\Omega)} + \gamma\int_{\Omega}\big(\mu(|\tilde\vT_h|)\tilde\vT_h -\mu(|\tilde\vT^\bd|)\tilde\vT^\bd\big): (\tilde\vT_h-\tilde\vT^\bd)\\
& = \alpha \int_\Omega (\tilde\vT_h-\tilde\vT^\bd):(\vS_h-\tilde\vT^\bd) + \gamma\int_{\Omega}\big(\mu(|\tilde\vT_h|)\tilde\vT_h -\mu(|\tilde\vT^\bd|)\tilde\vT^\bd\big): (\vS_h-\tilde\vT^\bd).
\end{align*}
Since both $\tilde\vT_h$ and $\tilde\vT^\bd$ are bounded in $M$ uniformly with respect to $h$, and
$$\|\mu(|\tilde\vT_h|)\tilde\vT_h -\mu(|\tilde\vT^\bd|)\tilde\vT^\bd\|_{ L^2(\Omega)} \le 2 C_1 |\Omega|^{\frac{1}{2}},$$
again a uniform bound, then the approximation properties of $M_h$ and
the monotonicity property \eqref{eqn:mu_mon} imply that
\begin{equation} \label{eq:tildeTh-tildeT}
\lim_{h \to 0}  \|\tilde\vT_h-\tilde\vT^\bd\|_{ L^2(\Omega)} =0.
\end{equation}
On the other hand, the auxiliary tensor $\tilde\vT_h$ permits us to argue as in the proof of Lemma \ref{lem:T_to_G(u)}. Indeed, the monotonicity property \eqref{eqn:mu_mon} yields
\begin{align*}
\alpha \|\vT_h- \tilde\vT_h\|^2_{ L^2(\Omega)}  &\le \int_{\Omega}\vD(\vu_h) :(\vT_h - \tilde\vT_h) -
\int_{\Omega}\vD(\bar \vu) :(\vT_h - \tilde\vT_h)\\
& = \int_{\Omega}\vD(\vu_h) :\vT_h - \int_{\Omega}\vD(\vu_h) :\tilde\vT_h
-  \int_{\Omega}\vD(\bar \vu) :(\vT_h - \tilde\vT_h).
\end{align*}
From \eqref{eqn:lim_suph} and  \eqref{eq:tildeTh-tildeT}, we easily derive that the above right-hand side tends to zero. Hence
$$\lim_{h \to 0}  \|\vT_h-\tilde\vT_h\|_{ L^2(\Omega)} =0, $$
and then combining this with \eqref{eq:tildeTh-tildeT} we infer that
\begin{equation} \label{eq:Th-tildeT}
\lim_{h \to 0}  \|\vT_h-\tilde\vT^\bd\|_{ L^2(\Omega)} =0.
\end{equation}
Hence uniqueness of the limit implies that $\bar \vT = \tilde\vT^\bd = \G(\bar \vu)$. This, and  \eqref{infsup:b1h}, permit to identify the limit as in Lemma \ref{lem:T_to_G(u)} and Theorem \ref{thm:existence1}, and proves convergence to a weak solution without restrictions on the data.
Thus we have proved the following result.

\begin{thm}(Convergence for all data) \label{thm:uncondconvh}
	Under the above approximation properties and compatibility of the discrete spaces, up to a subsequence,
	\begin{alignat*}{2}
	\lim_{h \to 0} \vu_h &= \vu  &&\quad\mbox{weakly in } H_0^1(\Omega)^d, \, \\
	\lim_{h \to 0}\vu_h &= \vu  &&\quad\mbox{strongly in } L^q(\Omega)^d\,\,  \mbox{with }1\leq q < \infty\; \mbox{ if }\, d=2, \mbox{ and }1 \leq q < 6\; \mbox{ if }\, d=3,  \\
	\lim_{h \to 0}\vT_h &= \vT^\bd  &&\quad \mbox{strongly in } L^2(\Omega)^{d \times d},\\
	\lim_{h \to 0} p_h &= p   &&\quad\mbox{weakly in } L^2(\Omega),
	\end{alignat*}
	where $(\vT^\bd,\vu,p)$ is a solution of \eqref{pb:NS_EL_weakR1}, \eqref{pb:NS_EL_weakR2}.
\end{thm}


\subsubsection{Error estimate} \label{subsec:error_estimate}

We now prove  an \emph{a priori} error estimate between $(\vT^\bd,\vu,p)$ and $(\vT_h,\vu_h,p_h)$,  under the assumption \eqref{eqn:mu_Lip}  that has not been used so far, and the small data condition \eqref{eq:smalldata} below. Note that this small data condition is in fact the same as the uniqueness condition \eqref{hyp:input_data},  upon replacing $C_{\vu}$ by $\widetilde{C_{\vu}}$.  To simplify the notation and compress some of the long displayed lines of mathematics, we shall write $\|\cdot\|_V$, $\|\cdot\|_M$ and $\|\cdot\|_Q$
instead of $\|D(\cdot)\|_{L^2(\Omega)}$ (as a norm on $V$), 
$\|\cdot\|_{L^2(\Omega)}$ (as a norm on $M$) and $\|\cdot\|_{L^2(\Omega)}$ (as a norm on $Q$), respectively.

\begin{thm} \label{thm:apriorierror}
	In addition to \eqref{eqn:mu_Lip}, let the input data satisfy
	\begin{equation} \label{eq:smalldata}
	\frac{\gamma}{\alpha}\Lambda +\alpha\hat{D}\widetilde{C_{\vu}}\leq\theta<1,
	\end{equation}
	where $0<\theta <1$ and $\hat{D}$ is the constant from \eqref{bound:form_d}.
 	Then, there exists a constant $C>0$ independent of $h$ such that the difference between the solution $(\vT_h,\vu_h,p_h)$ of \eqref{pb:weak_disc} and $(\vT^\bd,\vu,p)$ of  \eqref{pb:weak_cont} satisfies
	\begin{equation} \label{eqn:error_a_priori}
	\|\vu-\vu_h\|_V+\|\vT^\bd-\vT_h\|_M+\|p-p_h\|_Q \leq C\left[\inf_{\vv_h\in V_h}\|\vu-\vv_h\|_V+\inf_{\vS_h\in M_h}\|\vT^\bd-\vS_h\|_M+\inf_{q_h\in Q_h}\|p-q_h\|_Q\right].
	\end{equation}
\end{thm}

\begin{proof}
	Since we are using conforming finite element spaces, taking $(\vS,\vv,q)=(\vS_h,\vv_h,q_h)$ in \eqref{pb:weak_cont} and subtracting the equations of \eqref{pb:weak_disc} we easily get
    \begin{alignat}{2} \label{pb:weak_Galerkin}
	\hspace{-0.3cm}	d(\vu;\vu,\vv_h)-d(\vu_h;\vu_h,\vv_h)+\displaystyle{\int_{\Omega}}(\vT^\bd-\vT_h):\vD(\vv_h)+b_2(\vv_h,p-p_h) & =  0 &&\;\, \forall\, \vv_h\in V_h, \nonumber\\
	\hspace{-0.3cm}	 \alpha\displaystyle{\int_{\Omega}}(\vT^\bd-\vT_h):\vS_h+\gamma\displaystyle{\int_{\Omega}}(\mu(|\vT^\bd|)\vT^\bd-\mu(|\vT_h|)\vT_h):\vS_h & =  \displaystyle{\int_{\Omega}}\vD(\vu-\vu_h):\vS_h &&\;\, \forall\, \vS_h\in M_h, \\
	\hspace{-0.3cm}	b_2(\vu-\vu_h,q_h) & =  0 &&\;\, \forall\, q_h\in Q_h.\nonumber
	\end{alignat}
	The rest of the proof is divided into three steps.
	
	\textbf{Step 1} ({E}rror bound for the pressure). By the triangle inequality we have, for any $q_h \in Q_h$,
	\begin{equation*}
	\|p-p_h\|_Q\leq \|p-q_h\|_Q+\|q_h-p_h\|_Q,
	\end{equation*}
	and it therefore suffices to derive a bound on $\|q_h-p_h\|_Q$. From the (discrete) inf-sup condition we have
	\begin{equation*}
	\beta^*\|p_h-q_h\|_Q\leq \sup_{\vv_h\in V_h}\frac{b_2(\vv_h,p_h-q_h)}{\|\vv_h\|_V}.
	\end{equation*}
 	Again, using the first equation of \eqref{pb:weak_Galerkin} we have
	\begin{align*}
	b_2(\vv_h,p_h-q_h) & =  b_2(\vv_h,p_h-p)+b_2(\vv_h,p-q_h) \\
	& =  d(\vu;\vu,\vv_h)-d(\vu_h;\vu_h,\vv_h)+\int_{\Omega}(\vT^\bd-\vT_h):\vD(\vv_h)+b_2(\vv_h,p-q_h) \\
	& \leq  \left[(\hat{D}\|\vu\|_V+\hat{D}\|\vu_h\|_V)\|\vu-\vu_h\|_V+\|\vT^\bd-\vT_h\|_M+c_b\|p-q_h\|_Q\right]\|\vv_h\|_V \\
	& \leq  \left[2\hat{D} \widetilde{C_{\vu}} \|\vu-\vu_h\|_V+\|\vT^\bd-\vT_h\|_M+c_b\|p-q_h\|_Q\right]\|\vv_h\|_V,
	\end{align*}
	where we can take $c_b=C_K$ using the relation $\|\di(\vv)\|_{L^2(\Omega)}^2+\|\rot(\vv)\|_{L^2(\Omega)}^2=\| \nabla \vv
	\|_{L^2(\Omega)}^2$ that holds because we have homogeneous Dirichlet boundary conditions (otherwise take $c_b=\sqrt{d}C_K$). Thus, we obtain
	\begin{equation} \label{errest_step1_p}
	\|p-p_h\|_Q \leq \frac{2\hat{D} \widetilde{C_{\vu}}}{\beta^*}\|\vu-\vu_h\|_V+\frac{1}{\beta^*}\|\vT^\bd-\vT_h\|_M+\left(1+\frac{c_b}{\beta^*}\right)\|p-q_h\|_Q
	\end{equation}
	for any $q_h\in Q_h$.

	\textbf{Step 2} ({E}rror bound for the stress tensor). Again, we start with the triangle inequality; for any $\vS_h \in M_h$ we have that
	\begin{equation*}
	\|\vT^\bd-\vT_h\|_M \leq \|\vT^\bd-\vS_h\|_M+\|\vS_h-\vT_h\|_M,
	\end{equation*}
	and we then bound $\|\vT_h-\vS_h\|_M$. Thanks to the monotonicity property \eqref{eqn:mu_mon} and the second equation of \eqref{pb:weak_Galerkin}, we have
	\begin{align*}
	\alpha\|\vT_h-\vS_h\|_M^2 & \leq  \alpha\int_{\Omega}|\vT_h-\vS_h|^2+\gamma\int_{\Omega}(\mu(|\vT_h|)\vT_h-\mu(|\vS_h|)\vS_h):(\vT_h-\vS_h) \\
	&   \hspace*{-2.1cm}=\alpha\int_{\Omega}(\vT_h-\vT^\bd+\vT^\bd-\vS_h):(\vT_h-\vS_h)+\gamma\int_{\Omega}(\mu(|\vT_h|)\vT_h-\mu(|\vS_h|)\vS_h):(\vT_h-\vS_h) \\
	&   \hspace*{-2.1cm}=\int_{\Omega}\vD(\vu_h-\vu):(\vT_h-\vS_h)+\gamma\int_{\Omega}(\mu(|\vT^\bd|)\vT^\bd-\mu(|\vS_h|)\vS_h):(\vT_h-\vS_h)+\alpha\int_{\Omega}(\vT^\bd-\vS_h):(\vT_h-\vS_h) \\
	&  \hspace*{-2.1cm} \leq  \left[\|\vu-\vu_h\|_V+\alpha\|\vT^\bd-\vS_h\|_M+\gamma\Lambda\|\vT^\bd-\vS_h\|_M\right]  \|\vT_h-\vS_h\|_M,
	\end{align*}
	and thus
	\begin{equation} \label{errest_step2_Td}
	\|\vT^\bd-\vT_h\|_M \leq \frac{1}{\alpha}\|\vu-\vu_h\|_V+\left(2+\frac{\gamma\Lambda}{\alpha}\right)\|\vT^\bd-\vS_h\|_M
	\end{equation}
	for any $\vS_h\in M_h$.
	
	\textbf{Step 3} ({E}rror bound for the velocity). Recalling the definition of $V_{h,0}$ in \eqref{def:Vh0}, let $\vv_{h,0}\in V_{h,0}$ and let $\vv_h:=\vv_{h,0}-\vu_h\in V_{h,0}$. We will first show the relation \eqref{eqn:error_a_priori} by taking the infimum over $V_{h,0}$ instead of $V_h$. As before, we use the triangle inequality to get
	\begin{equation*}
	\|\vu-\vu_h\|_V \leq \|\vu-\vv_{h,0}\|_V + \|\vv_{h,0}-\vu_h\|_V.
	\end{equation*}
	Thanks to the assumption \eqref{eqn:eps_in_Mh}, we can take $\vS_h=\vD(\vv_h)$ in the second equation of \eqref{pb:weak_Galerkin} yielding
	\begin{equation*}
	\int_{\Omega}\vD(\vu-\vu_h):\vD(\vv_h) = \alpha\int_{\Omega}(\vT^\bd-\vT_h):\vD(\vv_h)+\gamma\int_{\Omega}\big(\mu(|\vT^\bd|)\vT^\bd-\mu(|\vT_h|)\vT_h\big):\vD(\vv_h).
	\end{equation*}
	Using the first equation of \eqref{pb:weak_Galerkin}, we can easily derive the equality
	\begin{align*}
	\| \vv_{h,0}-\vu_h \|_{V}^2 =\int_{\Omega}\vD(\vv_{h,0}-\vu_h):\vD(\vv_h) & =  \int_{\Omega}\vD(\vv_{h,0}-\vu):\vD(\vv_h)-\alpha(d(\vu;\vu,\vv_h)-d(\vu_h;\vu_h,\vv_h)) \\
	&\quad  
	-\alpha b_2(\vv_h,p-q_h) +\gamma\int_{\Omega}(\mu(|\vT^\bd|)\vT^\bd-\mu(|\vT_h|)\vT_h):\vD(\vv_h),
	\end{align*}
	thanks to the fact that $b_2(\vv_h,q_h-p_h) = 0$. To bound the convective term, we use
	\begin{align*}
	d(\vu;\vu,\vv_h)-d(\vu_h;\vu_h,\vv_h) & =  d(\vu-\vu_h;\vu,\vv_h)+d(\vu_h;\vu-\vu_h,\vv_h) \\
	& =  d(\vu-\vv_{h,0};\vu,\vv_h)+d(\vv_{h,0}-\vu_h;\vu,\vv_h) \\
	&  \quad +d(\vu_h;\vu-\vv_{h,0},\vv_h)+\underbrace{d(\vu_h;\vv_{h,0}-\vu_h,\vv_h)}_{=0} \\
	&  \hspace*{-4cm} \leq \left[\hat{D}\|\vu\|_V\|\vu-\vv_{h,0}\|_V+\hat{D}\|\vu\|_V\|\vv_{h,0}-\vu_h\|_V+\hat{D}\|\vu_h\|_V\|\vu-\vv_{h,0}\|_V\right]\|\vv_{h,0}-\vu_h\|_V \\
	&  \hspace*{-4cm}  \leq \left[2\hat{D} \widetilde{C_{\vu}} \|\vu-\vv_{h,0}\|_V+\hat{D} \widetilde{C_{\vu}}\|\vv_{h,0}- \vu_h \|_V\right]\|\vv_{h,0}-\vu_h\|_V,
	\end{align*}
	from which we get
	\begin{align*}
	\|\vv_{h,0}-\vu_h\|_V & \leq  \|\vu-\vv_{h,0}\|_V+2\alpha\hat{D} \widetilde{C_{\vu}}\|\vu-\vv_{h,0}\|_V \\
	& \quad +\alpha c_b\|p-q_h\|_Q+\gamma\Lambda\|\vT^\bd-\vT_h\|_M+\alpha\hat{D} \widetilde{C_{\vu}}\|\vv_{h,0}-\vu_h\|_V.
	\end{align*}
	Now using \eqref{errest_step2_Td} we arrive at
	\begin{align*}
	\|\vv_{h,0}-\vu_h\|_V & \leq  (1+2\alpha\hat{D} \widetilde{C_{\vu}} )\|\vu-\vv_{h,0}\|_V+\alpha c_b\|p-q_h\|_Q+\frac{\gamma\Lambda}{\alpha}\|\vu-\vu_h\|_V \\
	& \quad +  \gamma\Lambda\left(2+\frac{\gamma\Lambda}{\alpha}\right)\|\vT^\bd- \vS_h  \|_M+\alpha\hat{D} \widetilde{C_{\vu}} \|\vv_{h,0}-\vu_h\|_V \\
	& \leq  \left(1+2\alpha\hat{D} \widetilde{C_{\vu}} +\frac{\gamma\Lambda}{\alpha}\right)\|\vu-\vv_{h,0}\|_V+\gamma\Lambda\left(2+\frac{\gamma\Lambda}{\alpha}\right)\|\vT^d- \vS_h \|_M \\
	& \quad +\alpha c_b\|p-q_h\|_Q+\left(\frac{\gamma\Lambda}{\alpha}+\alpha\hat{D} \widetilde{C_{\vu}}\right)\| \vv_{h,0}-\vu_h\|_V.
	\end{align*}
	Therefore, using the assumption \eqref{eq:smalldata} on the input data, we obtain
	\begin{equation*}
	\|\vv_{h,0}-\vu_h\|_V \leq \frac{1}{1-\theta}\left[\left(1+2\alpha\hat{D} \widetilde{C_{\vu}} +\frac{\gamma\Lambda}{\alpha}\right)\|\vu-\vv_{h,0}\|_V+\gamma\Lambda\left(2+\frac{\gamma\Lambda}{\alpha}\right)\|\vT^\bd- \vS_h \|_M+\alpha c_b\|p-q_h\|_Q\right]
	\end{equation*}
	and thus
	\begin{equation} \label{errest_step3_u}
	\|\vu-\vu_h\|_V \leq \left(1+\frac{1+2\alpha\hat{D}\widetilde{C_{\vu}}+\frac{\gamma\Lambda}{\alpha}}{1-\theta}\right)\|\vu-\vv_{h,0}\|_V+\frac{\gamma\Lambda\left(2+\frac{\gamma\Lambda}{\alpha}\right)}{1-\theta}\|\vT^\bd- \vS_h \|_M+\frac{\alpha c_b}{1-\theta}\|p-q_h\|_Q
	\end{equation}
	for any $\vv_{h,0}\in V_{h,0}$. Finally, combining \eqref{errest_step1_p}, \eqref{errest_step2_Td} and \eqref{errest_step3_u} we obtain
	\begin{equation*}
	\|\vu-\vu_h\|_V+\|\vT^\bd-\vT_h\|_M+\|p-p_h\|_Q \leq c_1\inf_{\vv_{h,0}\in V_{h,0}}\|\vu-\vv_{h,0}\|_V+c_2\inf_{\vS_h\in M_h}\|\vT^\bd-\vS_h\|_M+c_3\inf_{q_h\in Q_h}\|p-q_h\|_Q
	\end{equation*}
	with
	\begin{align*}
	c_1 & :=  \left[1+\frac{2\hat{D} \widetilde{C_{\vu}} }{\beta^*}+\frac{1}{\alpha}\left(1+\frac{1}{\beta^*}\right)\right]\left(1+\frac{1+2\alpha\hat{D} \widetilde{C_{\vu}} +\frac{\gamma\Lambda}{\alpha}}{1-\theta}\right), \\
	c_2 & :=  \left(1+\frac{1}{\beta^*}\right)\left(2+\frac{\gamma\Lambda}{\alpha}\right)+\left[1+\frac{2\hat{D} \widetilde{C_{\vu}} }{\beta^*}+\frac{1}{\alpha}\left(1+\frac{1}{\beta^*}\right)\right]\frac{\gamma\Lambda\left(2+\frac{\gamma\Lambda}{\alpha}\right)}{1-\theta} ,\\
	c_3 & :=  \left(1+\frac{c_b}{\beta^*}\right)+\left[1+\frac{2\hat{D} \widetilde{C_{\vu}} }{\beta^*}+\frac{1}{\alpha}\left(1+\frac{1}{\beta^*}\right)\right]\left(\frac{\alpha c_b}{1-\theta}\right).
	\end{align*}
	We can then conclude the proof using \eqref{eq:bdd_vh0}.
\end{proof}

	
\subsection{Examples of conforming approximation} \label{subsec:Ex_conf_app}

From now on, we assume that the boundary of the Lipschitz domain $\Omega \subset \mathbb{R}^d$ is a polygonal line (when $d=2$) or a polyhedral surface (when $d=3$), so that it can be exactly meshed. For each $h$, let $\Th$ be a conforming mesh on $\overline\Omega$ consisting of elements $E$, triangles or quadrilaterals in two dimensions, tetrahedra or hexahedra (all planar-faced) in three dimensions, conforming in the sense that the mesh has no hanging nodes. As usual, the diameter of $E$ is denoted by $h_E$,
$$h = \sup_{E \in \Th} h_E, $$
and $\varrho_E$ is the diameter of the  largest ball inscribed in $E$.


\subsubsection{The simplicial case}

In the case of simplices, the family of meshes $\Th$ is assumed to be {\em regular} in the sense of Ciarlet~\cite{Cia}: i.e., it is assumed that there exists a constant $\sigma >0$, independent of $h$, such that
\begin{equation} \label{eq:reg_mesh}
\frac{h_E}{\varrho_E} \le \sigma\quad \forall\, E \in \Th.
\end{equation}
This condition guarantees that there is an invertible affine mapping $\mathcal F_E$ that maps the unit reference simplex  onto $E$.

For any integer $k \ge 0$, let $\polP_k$ denote the space of polynomials in $d$ variables of degree at most $k$. In each element $E$, the functions will be approximated in the spaces $\polP_k$. The specific choice of finite element spaces is dictated by two considerations. First, conditions \eqref{infsup:b1h} and \eqref{eqn:eps_in_Mh} must be satisfied. Next, since the number of unknowns in \eqref{pb:weak_disc} is large, the degree $k$ of the finite element functions should be small. It is well-known that the lowest degree of conforming approximation of $(\vu,p)$ satisfying \eqref{infsup:b1h}, without modification of the bilinear forms, is the Taylor-Hood $\polP_2^d$--$\polP_1$ element, see~\cite{ref:GiR,ref:Boffi97}, provided each element has at least one interior vertex. In view of \eqref{eqn:eps_in_Mh}, this implies that $\vT^{\bd}$ is approximated by $\polP_1^{d\times d}$.
Thus the corresponding finite element spaces are
\begin{align*}
V_h  &:=\{ \vv_h \in H^1_0(\Omega)^d: \quad \restriction{\vv_h}{E}  \in \polP_2^d \quad \forall\, E \in \Th\},
\\
Q_h  &:=\{ q_h \in H^1(\Omega)\cap L^2_0(\Omega): \quad \restriction{q_h}{E}  \in \polP_1 \quad \forall\, E \in \Th\},
\\
M_h  &:=\{ \vS_h \in L^2(\Omega)_{{ \rm sym}}^{d\times d}: \quad \restriction{\vS_h}{E}  \in (\polP_1)_{{ \rm sym}}^{d\times d} \quad \forall\, E \in \Th\}.
\end{align*}
It is easy to check that with these spaces on a simplicial mesh, under condition \eqref{eq:reg_mesh},  problem \eqref{pb:weak_disc} has at least one solution. Furthermore, if the data satisfy \eqref{eq:smalldata}, then
Theorem \ref{thm:apriorierror} yields
\begin{equation} \label{eqn:errorTH2}
\| \vD( \vu-\vu_h)\|_{ L^2(\Omega)} +\|\vT^\bd-\vT_h\|_{ L^2(\Omega)}+\|p-p_h\|_{ L^2(\Omega)} \leq C\,h^2,
\end{equation}
provided that the solution is sufficiently smooth, namely  $\vu \in H^3(\Omega)^d\cap H^1_0(\Omega)^d$,  $\vT^{\bd} \in H^2(\Omega)^{d\times d}$, and $p \in H^2(\Omega)\cap L^2_0(\Omega)$. Therefore the scheme has order two for an optimal number of degrees of freedom, i.e., this order of convergence cannot be achieved with fewer degrees of freedom.


\subsubsection{The quadrilateral/hexahedral case} \label{ss:quad}

The notion of regularity is more complex for quadrilateral and much more complex for hexahedral elements. In the case of quadrilaterals~\cite{ref:GiR},  the family of  meshes is regular if the elements are convex and, moreover, the subtriangles associated to each vertex (there is one per vertex) all satisfy \eqref{eq:reg_mesh}.  In the case of hexahedra with plane faces, convexity and the validity of \eqref{eq:reg_mesh} for the subtetrahedra associated to each vertex are necessary but not sufficient. This difficulty has been investigated by many authors, see for instance~\cite{ref:Zhang2005,ref:Knabner2003}; the most relevant publication concerning hexahedra with plane faces is however~\cite{ref:Johnen17}, where the minimum of the Jacobian in the reference cube  $\hat E$  is bounded below by the minimum of the coefficients of its B\'ezier expansion and this minimum is determined by an efficient algorithm. The details of this are beyond the scope of this work, and we shall simply assume here that the minimum of these B\'ezier coefficients is strictly positive and that furthermore, denoting by ${\mathcal J}_E$ the Jacobian determinant of $\mathcal F_E$,
\begin{equation} \label{eq:reg_hexa}
{\mathcal J}_E(\hat \vx) \ge \hat c \varrho_E^3\quad \forall\,\hat \vx \in \hat E
\end{equation}
with a constant $\hat c$ independent of $E$ and $h$.
If these conditions hold, there is an invertible bi-affine mapping $\mathcal F_E$ in two dimensions or tri-affine in three dimensions that maps the unit reference square or cube  onto $E$.

We let $\polQ_k$ be the space of polynomials in $d$ variables of degree at most $k$ in {\em each} variable. In contrast to the case of simplicial meshes, the space $\polQ_k$ is not invariant under the composition with $\mathcal F_E$, which makes the compatibility condition \eqref{eqn:eps_in_Mh} between $\vD(V_h)$ and $M_h$ problematic. 
To circumvent this issue, we restrict ourselves to affine maps $\mathcal F_E$, thereby allowing subdivisions consisting of parallelograms/parallelepipeds. In addition, the situation is less satisfactory when a quadrilateral or hexahedral mesh is used, because although the Taylor-Hood $\polQ_2^d$--$\polQ_1$ element satisfies \eqref{infsup:b1h}, the second condition \eqref{eqn:eps_in_Mh} does not hold if $\vT^{\bd}$ is approximated by $ {\polQ}_1^{d\times d}$ since the components of the gradient of $\polQ_2$ functions belong to a space, intermediate between  $\polQ_2$ and $\polQ_1 $, that is strictly larger than both $\polQ_1$ and $\polP_2$. Therefore, in order to satisfy \eqref{eqn:eps_in_Mh}, the simplest option is to discretise each component of $\vT^{\bd}$ by ${\polQ}_2$. The corresponding finite element spaces are
\begin{align*}
V_h  &:=\{ \vv_h \in H^1_0(\Omega)^d: \quad \restriction{\vv_h}{E}  \in \polQ_2^d \quad \forall\, E \in \Th\},
\\
Q_h  &:=\{ q_h \in H^1(\Omega)\cap L^2_0(\Omega): \quad \restriction{q_h}{E}  \in \polQ_1 \quad \forall\, E \in \Th\},
\\
M_h  &:=\{ \vS_h \in L^2(\Omega)_{{ \rm sym}}^{d\times d}: \quad \restriction{\vS_h}{E}  \in (\polQ_2)_{{ \rm sym}}^{d\times d} \quad \forall\, E \in \Th\}.
\end{align*}
With these spaces and under the above regularity conditions, problem \eqref{pb:weak_disc}  has at least one solution and the error estimate \eqref{eqn:errorTH2} holds if the data satisfy \eqref{eq:smalldata}. However, this triple of spaces is no longer optimal, because the degree two approximation of $\vT^{\bd}$ now requires far too many degrees of freedom with no gain in accuracy. For instance, when $d=3$, its approximation by $({\polQ}_2)_{{ \rm sym}}^{3\times 3}$ requires $27 \times 6 =162$ unknowns inside each element instead of $8\times 6 =48$ unknowns for $({\polQ}_1)_{{ \rm sym}}^{3\times 3}$. 

The nonconforming finite element approximations discussed in Section~\ref{sec:NS_ENL_AppNC} do not require an affine mapping $\mathcal F_E$ and, by considering $\polP$-type approximations on the physical element $E$, do not suffer from the computational cost overhead mentioned above.

\section{ Nonconforming finite element approximation} \label{sec:NS_ENL_AppNC}

The nonconforming approximations developed here will not only allow the use of elements of degree one for $\vu$, but will also lead to locally mass-conserving schemes. Because of the discontinuity of the finite element  functions, the proofs are in some cases more complex; this is true in particular for the proof of the inf-sup condition for the discrete divergence.


\subsection{The quadrilateral/planar-faced hexahedral case} \label{subsec:quads}

Here we consider quadrilateral/hexahedral grids $\Th$ with planar faces, satisfying the regularity assumptions stated in Section \ref{subsec:Ex_conf_app}. There is a wide choice of possible approximations with nonconforming finite elements. Here we  propose   globally discontinuous velocities in $\polP^d_{ k}$, ${ k}\ge 1$, in each cell associated with globally discontinuous pressures and stresses both of degree at most $k-1$. Thus we consider $V_h \subset L^2(\Omega)^d$, $Q_h \subset L^2_0(\Omega)$ and $M_h \subset L^2(\Omega)_{{\rm sym}}^{d\times d}$ defined by
\begin{align}
\label{eq:Vhnon-k}
V_h  &:=\{ \vv_h \in L^2(\Omega)^d: \quad \restriction{\vv_h}{E}  \in \polP^d_{ k }\quad \forall\,E\in \Th\},
\\
\label{eq:Qhnon-k}
Q_h  &:=\{ q_h \in L^2_0(\Omega): \quad \restriction{q_h}{E} \in \polP_{{ k }-1} \quad \forall\, E \in \Th\},
\\
\label{eq:Mhnon-k}
M_h  &:=\{ \vS_h \in L^2(\Omega)_{{\rm sym}}^{d\times d}: \quad \restriction{\vS_h}{E}  \in (\polP_{{ k }-1})_{{\rm sym}}^{d\times d} \quad \forall\, E\in \Th\}.
\end{align}
As usual, the  full nonconformity of $V_h$ is compensated by adding to the forms consistent jumps and averages on edges  when $d=2$ or faces  when $d=3$; see for instance~\cite{Riviere2008}. Let $\Gamma_h  = \Gamma_h^i\cup\Gamma_h^b$ denote the set of all edges when $d=2$ or all faces when $d=3$ with $\Gamma_h^i$ and $\Gamma_h^b$ signifying the set of all interior and the set of all boundary edges ($d=2$) or faces ($d=3$), respectively. A unit normal vector $\vn_e$ is attributed to each $e \in \Gamma_h$; its direction can be freely chosen. Here, the  following rule is applied: if $e\in \Gamma_h^b$, then $\vn_e = \vn_\Omega$, the exterior unit normal to $\Omega$;
if $e \in\Gamma_h^i$, then $\vn_e$ points from  $E_i$ to $E_j$, where $E_i$ and $E_j$ are the two elements of $\Th$ adjacent to $e$ and the number $i$ of $E_i$ is smaller than that of $E_j$.
The jumps and averages of any function $f$ on $e$ (smooth enough to have a trace)  are defined by
$$[f(x)]_e := \restriction{f(x)}{E_i} - \restriction{f(x)}{E_j},\quad \mbox{when}\ \vn_e\  \mbox{points from}\   E_i\  \mbox{to}\  E_j,
$$
$$\{f(x)\}_e := \frac{1}{2} \big(\restriction{f(x)}{E_i} + \restriction{f(x)}{E_j}\big).
$$
When $e\in\Gamma_h^b$, the jump and average are defined to coincide with the trace on $e$.

The terms involving jumps and averages that are added to each form are not unique; here we make the following fairly standard choice:
\begin{equation} \label{eq:epshnon}
\int_\Omega \vS :\vD(\vv) \simeq b_{1h}(\vS_h,\vv_h) := \sum_{E \in \Th} \int_E \vS_h:\vD(\vv_h) - \sum_{e \in \Gamma_h} \int_e  \{\vS_h\}_e\vn_e \cdot [\vv_h]_e.
\end{equation}
The trilinear form $d$ is approximated by a centred discretisation, as follows:
\begin{equation} \label{eq:dhnon}
\begin{split}
d_h(\vu_h;\vv_h,\vw_h)  : =& \sum_{E \in \Th} \int_E  \left[ (\vu_h\cdot\nabla)\vv_h  \right]\cdot \vw_h + \frac{1}{2}\sum_{E \in \Th} \int_E \di(\vu_h) (\vv_h \cdot \vw_h)\\
& - \frac{1}{2}\sum_{e \in \Gamma_h}\int_e [\vu_h]_e\cdot \vn_e \{\vv_h \cdot \vw_h\}_e -  \sum_{e \in \Gamma_h^i} \int_e\{\vu_h\}_e\cdot \vn_e [\vv_h]_e\cdot\{ \vw_h\}_e.
\end{split}
\end{equation}
The divergence form $b_2$ is approximated by
\begin{equation} \label{eq:b2hnon}
b_{2h}(\vv_h,q_h)  := -\sum_{E \in \Th} \int_E q_h \di(\vv_h)+ \sum_{e \in \Gamma_h}\int_e [\vv_h]_e\cdot \vn_e \{q_h\}_e.
\end{equation}
Clearly, the jump terms in \eqref{eq:epshnon} and \eqref{eq:b2hnon} vanish when $\vv_h$ belongs to $H^1_0(\Omega)^d$. Likewise, the jump and divergence terms in \eqref{eq:dhnon} vanish when $\vu_h$ and $\vv_h$ belong to $H^1_0(\Omega)^d$ and $\di(\vu_h) = 0$. Moreover, \eqref{eq:dhnon} is constructed so that $d_h$ is antisymmetric,
\begin{equation} \label{eq:antisym-non}
d_h(\vu_h;\vv_h,\vw_h) = - d_h(\vu_h;\vw_h,\vv_h) \quad \forall\, \vu_h,\vv_h,\vw_h \in V_h.
\end{equation}
Finally, the following positive definite form acts as a penalty to compensate the nonconformity of $\vu_h$:
\begin{equation} \label{eq:Jh}
J_h(\vu_h,\vv_h)  := \sum_{e \in \Gamma_h}\frac{\sigma_e}{h_e} \int_e [\vu_h]_e \cdot [\vv_h]_e,
\end{equation}
where $h_e$ is the average of the diameter of the two elements adjacent to $e$, if $e \in\Gamma_h^i$, or the diameter of the element adjacent to $e$ otherwise. The parameters $\sigma_e > 0$ will be chosen below to guarantee stability of the scheme,  see \eqref{eqn:sigma.e} and \eqref{eqn:bdd.C1}.
This form is also used to define the norm on $V_h$ by
\begin{equation} \label{eq:normVh}
\|\vv_h\|_{V_h}:= \Big(\|\vD(\vv_h)\|^2_h + J_h(\vv_h,\vv_h) \Big)^{\frac{1}{2}},
\end{equation}
where
\begin{equation} \label{eq:seminormVh}
\|\vD(\vv_h)\|_h:= \Big(\sum_{E \in \Th} \|\vD(\vv_h)\|^2_{L^2(E)}\Big)^{\frac{1}{2}}
\end{equation}
denotes the associated semi-norm. Also, in view of \eqref{eq:b2hnon}, we define the space of discretely divergence-free functions,
\begin{equation} \label{eq:V0hnon.conf}
V_{h,0} := \{ \vv_h \in V_h: \quad b_{2h}(\vv_h,q_h) = 0 \quad \forall\, q_h\in Q_h\}.
\end{equation}

The discrete scheme reads: find $(\vT_h,\vu_h,p_h) \in  M_h \times  V_h \times Q_h$ solution of
\begin{alignat}{2} \label{pb:nonconfscheme}
\begin{aligned}
d_h(\vu_h;\vu_h,\vv_h) + b_{1h}(\vT_h,\vv_h) + b_{2h}(\vv_h,p_h)+ J_h(\vu_h,\vv_h) & =  \int_{\Omega}\vf\cdot\vv_h &&\quad \forall\, \vv_h\in V_h, \\
\alpha\int_{\Omega}\vT_h:\vS_h+\gamma\int_{\Omega}\mu(|\vT_h|)\vT_h:\vS_h & =  b_{1h}(\vS_h,\vu_h) &&\quad \forall\, \vS_h\in M_h, \\
b_{2h}(\vu_h,q_h) & =  0 &&\quad \forall\, q_h\in Q_h.
\end{aligned}
\end{alignat}
As expected, $b_{2h}(\vv_h,1) = 0$, and therefore the system \eqref{pb:nonconfscheme} is unchanged when the zero mean value constraint is  lifted from the functions of $Q_h$.

\subsubsection{Properties of the norm and forms} \label{subsec:Nonconf-forms}

All constants below depend on the regularity of the mesh but are independent of $h$. In particular, we shall use $C$ to denote such generic constant independent of $h$.
In addition, we shall use the following ``edge to interior" inequality. There exists a constant $\hat C$, depending only on the dimension $d$ and the degree of the polynomials, such that for all $\vv_h \in V_h$, all $e \in \Gamma_h$ and any element $E$, adjacent to $e$,
\begin{equation} \label{eq:e.to.E}
\|\vv_h\|_{L^2(e)} \le \hat C \bigg(\frac{|e|}{|E|}\bigg)^{\frac{1}{2}} \|\vv_h\|_{L^2(E)}.
\end{equation}
It is easy to check that \eqref{eq:normVh} defines a norm on $V_h$. Next, the results in~\cite{ref:SB03,ref:SB04} yield the following consequences of a discrete Korn inequality:
\begin{equation} \label{eq:disc-Korn}
\|\vv_h\|_{L^2(\Omega)} \le C \|\vv_h\|_{V_h} \quad \forall\, \vv_h \in V_h,
\end{equation}
and
\begin{equation} \label{eq:disc-Kornbis}
\|\nabla_h \vv_h\|_{L^2(\Omega)} \le C \|\vv_h\|_{V_h} \quad \forall\, \vv_h \in V_h,
\end{equation}
where $\nabla_h\vv_h$ is the broken gradient (i.e., the local gradient in each element).
Moreover, by following the work  in~\cite{girRivLi2016-2,ref:lasisSulli03,buffa-ortner}, this can be generalised for all finite $p\ge  1$ when $d=2$ and all $p \in [1,6]$ when $d=3$, to
\begin{equation} \label{eq:Lpbd}
\|\vv_h\|_{L^p(\Omega)} \le C(p) \|\vv_h\|_{V_h} \quad \forall\, \vv_h \in V_h.
\end{equation}
With this norm, the following compactness result holds for any sequence $\vv_h$ in $V_h$, see~\cite{buffa-ortner,girRivLi2016-2,ref:BartelsJensenMuller}: if there exists a constant $C$ independent of $h$ such that
$$\|\vv_h\|_{V_h} \le C,$$
then there exists a function $\bar \vv \in H^1_0(\Omega)^d$ such that for all finite $p\ge 1$ when $d=2$ and all $p \in [1,6)$ when $d=3$,
\begin{equation} \label{eq:strng.lim}
\lim_{h \to 0} \|\vv_h - \bar \vv\|_{L^p(\Omega)} = 0.
\end{equation}

Regarding the forms, a straightforward finite-dimensional argument shows that, for all $\vu_h,\vv_h,\vw_h \in V_h$,
\begin{equation} \label{eq:bdd.dh3}
\Big| \sum_{e \in \Gamma_h}\int_e [\vu_h]_e\cdot \vn_e \{\vv_h \cdot \vw_h\}_e\Big| \le C \big(J_h(\vu_h,\vu_h)\big)^{\frac{1}{2} } \|\vv_h\|_{L^4(\Omega)} \|\vw_h\|_{L^4(\Omega)},
\end{equation}
\begin{equation} \label{eq:bdd.dh4}
\Big|\sum_{e \in \Gamma_h^i} \int_e\{\vu_h\}_e\cdot \vn_e [\vv_h]_e \cdot \{\vw_h\}_e \Big| \le  C \|\vu_h\|_{L^4(\Omega)}\big(J_h(\vv_h,\vv_h)\big)^{\frac{1}{2}}\|\vw_h\|_{L^4(\Omega)}.
\end{equation}
Hence we have, for all $\vu_h,\vv_h,\vw_h \in V_h$,
\begin{align} \label{eq:bdd.dh}
\begin{aligned}
&\big|d_h(\vu_h;\vv_h,\vw_h)\big| \le
C \|\vu_h\|_{L^4(\Omega)}\big(J_h(\vv_h,\vv_h)\big)^{\frac{1}{2}}  \|\vw_h\|_{L^4(\Omega)}\\
&+\Big[\|\nabla_h\vv_h\|_{L^2(\Omega)} \|\vu_h\|_{L^4(\Omega)}  + \frac{1}{2} \big(\big(\sum_{E \in \Th} \|\di(\vu_h)\|^2_{L^2(E)}\big)^{\frac{1}{2} } +  C \big(J_h(\vu_h,\vu_h)\big)^{\frac{1}{2} }\Big) \|\vv_h\|_{L^4(\Omega)}\Big] \|\vw_h\|_{L^4(\Omega)}.
\end{aligned}
\end{align}
Similarly,
\begin{alignat}{2}
\label{eq:bdd.b2h}
\Big|b_{2h}(\vv_h,q_h)\Big|  &\le  \Big(\big(\sum_{E \in \Th} \|\di(\vv_h)\|^2_{L^2(E)}\big)^{\frac{1}{2} } + C  \big(J_h(\vv_h,\vv_h)\big)^{\frac{1}{2} } \Big)\|q_h\|_{L^2(\Omega)} &&\quad \forall\, \vv_h \in V_h, q_h \in Q_h,
\\
\label{eq:bdd.ah}
\Big|b_{1h}(\vS_h,\vv_h)\Big|  &\le  \Big(\|\vD(\vv_h)\|_h + C \big(J_h(\vv_h,\vv_h)\big)^{\frac{1}{2} }\Big) \|\vS_h\|_{L^2(\Omega)} && \quad \forall\, \vv_h \in V_h, \vS_h \in M_h.
\end{alignat}
Finally, the inequality below is used in choosing  $\sigma_e$. Its proof is fairly straightforward, but it is included here for the reader's convenience.

\begin{prop} \label{prop:balance}
	For any $\vu_h \in V_h$, any choice of $\sigma_e >0$ and any real number $\delta >0$, we have
	\begin{equation} \label{eqn:balance1}
	\Big|\sum_{e \in \Gamma_h} \int_e  \{\vD(\vu_h)\}_e\vn_e \cdot [\vu_h]_e \Big| \le \frac{1}{2} \Big( \frac{1}{ \delta} J_h(\vu_h,\vu_h) + \delta \frac{C_h}{{\rm min}_{e \in \Gamma_h}\sigma_e} \|\vD(\vu_h)\|_h^2\Big),
	\end{equation}
	where
	\begin{equation} \label{eqn:bdd.C1}
 	C_h:= 2d\, \hat C^2 \max\left(\max_{e \in \Gamma_h^i} \bigg(h_e \max_{j=1,2}\frac{|e|}{|E_j|}\bigg), \max_{e \in \Gamma_h^b}\left(h_e\frac{|e|}{|E|}\right)\right),
	\end{equation}
	$E_1$ and $E_2$ are the elements that share the face $e\in \Gamma_h^i$, $E$ is the element that has face $e \in \Gamma_h^b$, and $\hat C$ is the constant appearing in inequality \eqref{eq:e.to.E} solely depending on $d$ and the polynomial degree.
\end{prop}

\begin{proof}
	For a face $e\in\Gamma_h^i$, which is shared by elements $E_1$ and $E_2$, we have
	\begin{align*}
	\Big| \int_e  \{\vD(\vu_h)\}_e\vn_e \cdot [\vu_h]_e \Big| &\le \bigg(\frac{\sigma_e}{h_e}\bigg)^{\frac{1}{2}} \|[\vu_h]_e\|_{L^2(e)} \bigg(\frac{h_e}{\sigma_e}\bigg)^{\frac{1}{2}} \frac{\hat C}{2} \sum_{j=1}^2  \bigg(\frac{|e|}{|E_j|}\bigg)^{\frac{1}{2}} \|\vD(\vu_h)\|_{L^2(E_i)} \\
	&  \le \frac{1}{2}\bigg(\frac{1}{\delta}\frac{\sigma_e}{h_e} \|[\vu_h]_e\|^2_{L^2(e)}+\delta\frac{h_e}{2\sigma_e} \hat C^2\max_{j=1,2}\frac{|e|}{|E_j|}\big( \|\vD(\vu_h)\|_{L^2(E_1)}^2
	+ \|\vD(\vu_h)\|_{L^2(E_2)}^2 \big)\bigg).
	\end{align*}
 	Similarly, for $e\in\Gamma_h^b$, which is the face of an element $E$ adjacent to $\partial \Omega$, we have
	\begin{equation*}
	\Big| \int_e  \{\vD(\vu_h)\}_e\vn_e \cdot [\vu_h]_e \Big| \le \frac{1}{2}\left(\frac{1}{\delta}\frac{\sigma_e}{h_e} \|[\vu_h]_e\|^2_{L^2(e)}+\delta\frac{h_e}{\sigma_e} \hat C^2\frac{|e|}{|E|}\|\vD(\vu_h)\|_{L^2(E)}^2\right).
	\end{equation*}

	By using the last two inequalities in
	\begin{equation*}
	\Big|\sum_{e \in \Gamma_h} \int_e  \{\vD(\vu_h)\}_e\vn_e \cdot [\vu_h]_e \Big|  \le \sum_{e\in\Gamma_h}\Big| \int_e  \{\vD(\vu_h)\}_e\vn_e \cdot [\vu_h]_e \Big|
	\end{equation*}
	and splitting the sum on the right-hand side into sums over the disjoint sets $\Gamma_h^b$ and $\Gamma_h^i$, we have that
	\begin{align*}
	&\Big|\sum_{e \in \Gamma_h} \int_e  \{\vD(\vu_h)\}_e\vn_e \cdot [\vu_h]_e \Big|  
	\le \frac{1}{2}\bigg(\frac{1}{\delta}\sum_{e \in \Gamma_h^b}\frac{\sigma_e}{h_e} \|[\vu_h]_e\|^2_{L^2(e)}+\delta\sum_{e\in \Gamma_h^b}\frac{h_e}{\sigma_e} \hat C^2\frac{|e|}{|E|}\|\vD(\vu_h)\|_{L^2(E)}^2\bigg)\\
	&\quad+ \frac{1}{2}\bigg(\frac{1}{\delta}\sum_{e \in \Gamma_h^i}\frac{\sigma_e}{h_e} \|[\vu_h]_e\|^2_{L^2(e)}+\delta\sum_{e \in \Gamma_h^i}\frac{h_e}{2\sigma_e} \hat C^2\max_{j=1,2}\frac{|e|}{|E_j|}\big( \|\vD(\vu_h)\|_{L^2(E_1)}^2
	+ \|\vD(\vu_h)\|_{L^2(E_2)}^2 \big)\bigg)
	\end{align*}
	with the notational convention that when summing over $e \in \Gamma_h^b$ the element $E$ under the summation sign is the element adjacent to $\partial\Omega$ with face $e$, and when summing over $e \in \Gamma_h^i$ the elements $E_1$ and $E_2$ under the summation sign are the ones that share the face $e$. Hence, 
	\begin{align*}
	&\Big|\sum_{e \in \Gamma_h} \int_e  \{\vD(\vu_h)\}_e\vn_e \cdot [\vu_h]_e \Big|  
	\le \frac{1}{2}\left(\frac{1}{\delta} J_h(\vu_h,\vu_h) +\frac{\delta}{\min_{e \in \Gamma_h^b}\sigma_e}  \hat C^2 \max_{e \in \Gamma_h^b}\left(h_e\frac{|e|}{|E|}\right) \sum_{e\in \Gamma_h^b} \|\vD(\vu_h) \|_{L^2(E)}^2\right)\\
	&\quad+ \frac{1}{2}\bigg(\frac{\delta}{\min_{e \in \Gamma_h^i}\sigma_e} \hat C^2{\rm max}_{e \in \Gamma_h^i} \bigg(h_e \max_{j=1,2}\frac{|e|}{|E_j|}\bigg)\sum_{e \in \Gamma_h^i}\frac{1}{2} \big( \|\vD(\vu_h)\|_{L^2(E_1)}^2
	+ \|\vD(\vu_h)\|_{L^2(E_2)}^2 \big)\bigg).
	\end{align*}
	The asserted result \eqref{eqn:balance1} follows from the last inequality by noting that, for each $E\in\Th$, the factor $\|\vD(\vu_h)\|_{L^2(E)}^2$ appears at most $2d$ times.
\end{proof}

Concerning the expression appearing in \eqref{eqn:bdd.C1} we note that, thanks to the regularity assumption on the family of meshes, we have that $h_e \frac{|e|}{|E|} \le C$ and so 
\begin{equation} \label{e:Ch}
C_h \le C.
\end{equation}

\subsubsection{First \emph{a priori} estimates} \label{subsec:1staprioriDG}

By testing the first equation of \eqref{pb:nonconfscheme} with $\vv_h = \vu_h$, applying the third equation and the antisymmmetry \eqref{eq:antisym-non} of $d_h$, we obtain
$$ b_{1h}(\vT_h,\vu_h) + J_h(\vu_h,\vu_h) = \int_\Omega \vf \cdot \vu_h. $$
Next, by testing the second equation of \eqref{pb:nonconfscheme} with $\vS_h = \vT_h$ and substituting the above equality, we deduce that
$$ \alpha \|\vT_h\|^2_{L^2(\Omega)} +  J_h(\vu_h,\vu_h)  \le \int_\Omega \vf \cdot \vu_h. $$
Thus, in view of \eqref{eq:disc-Korn}, we have our first bound:
\begin{equation} \label{eqn:1stbdd}
\alpha \|\vT_h\|^2_{L^2(\Omega)} +  J_h(\vu_h,\vu_h) \le C\|\vf\|_{L^2(\Omega)}\|\vu_h\|_{V_h}.
\end{equation}
A further bound is arrived at by testing the second equation of \eqref{pb:nonconfscheme} with $\vS_h = \vD(\vu_h)$; hence,
$$\alpha\int_{\Omega}\vT_h:\vD(\vu_h)+\gamma\int_{\Omega}\mu(|\vT_h|)\vT_h:\vD(\vu_h)  =  \|\vD(\vu_h)\|^2_h - \sum_{e \in \Gamma_h} \int_e  \{\vD(\vu_h)\}_e\vn_e \cdot [\vu_h]_e.
$$
Then Proposition \ref{prop:balance} gives, for any $\delta >0$,
\begin{equation} \label{eqn:2ndbdd}
\|\vD(\vu_h)\|^2_h \le \alpha\|\vT_h\|_{L^2(\Omega)}\|\vD(\vu_h)\|_h + \gamma C_1 |\Omega|^{\frac{1}{2}}\|\vD(\vu_h)\|_h +  \frac{1}{2} \Big( \frac{1}{ \delta} J_h(\vu_h,\vu_h) + \delta \frac{C_h}{{\rm min}_{e \in \Gamma_h}\sigma_e} \|\vD(\vu_h)\|_h^2\Big).
\end{equation}

We choose $\delta=1$ and, upon recalling \eqref{e:Ch}, assume that $\sigma_e$ is chosen so that
\begin{equation} \label{eqn:sigma.e}
{\rm min}_{e \in \Gamma_h}\sigma_e \ge C_h.
\end{equation}
Next, by adding $J_h(\vu_h,\vu_h)$ to both sides of \eqref{eqn:2ndbdd}, applying \eqref{eqn:1stbdd} to bound this term, and using the norm of $V_h$, we infer that
$$ \|\vu_h\|^2_{V_h} \le \alpha\|\vT_h\|_{L^2(\Omega)}\|\vD(\vu_h)\|_h + \gamma C_1 |\Omega|^{\frac{1}{2}}\|\vD(\vu_h)\|_h + C\|\vf\|_{L^2(\Omega)}\|\vu_h\|_{V_h} +
\frac{1}{2} \|\vu_h\|^2_{V_h} $$
and thus
\begin{equation} \label{eqn:3rdbdd}
\frac{1}{2}\|\vu_h\|_{V_h} \le \alpha\|\vT_h\|_{L^2(\Omega)} + \gamma C_1 |\Omega|^{\frac{1}{2}} + C \|\vf\|_{L^2(\Omega)}.
\end{equation}
To close the estimates, we return to \eqref{eqn:1stbdd} and get
$$ \alpha \|\vT_h\|^2_{L^2(\Omega)} +  J_h(\vu_h,\vu_h) \le \frac{1}{2}\bigg(\delta_2  \|\vu_h\|_{V_h}^2 + \frac{C^2}{\delta_2} \|\vf\|_{L^2(\Omega)}^2\bigg) $$
for any $\delta_2 >0$.
Thus
$$ \alpha \|\vT_h\|_{L^2(\Omega)} \le \frac{\sqrt{\alpha}}{\sqrt{2 \delta_2}}C\|\vf\|_{L^2(\Omega)} +
\frac{\sqrt{\alpha \delta_2}}{\sqrt{2}} \|\vu_h\|_{V_h}, $$
and the choice $\delta_2 = \frac{1}{8\alpha}$ yields
$$ \alpha \|\vT_h\|_{L^2(\Omega)} \le 2 \alpha C \|\vf\|_{L^2(\Omega)} + \frac{1}{4}\|\vu_h\|_{V_h}. $$
Thus we have shown the following uniform and unconditional bounds:
\begin{equation} \label{eqn:4rthbdd}
\|\vu_h\|_{V_h} \le 4 C (1+2 \alpha) \|\vf\|_{L^2(\Omega)} +  4 \gamma C_1 |\Omega|^{\frac{1}{2}},\quad
\|\vT_h\|_{L^2(\Omega)} \le  \frac{C}{\alpha}(4\alpha +1)   \|  \vf\|_{L^2(\Omega)} + \frac{\gamma}{\alpha} C_1 |\Omega|^{\frac{1}{2}}.
\end{equation}

An \emph{a priori} estimate for the pressure requires an inf-sup condition. This is the subject of the next subsection.

\subsubsection{An inf-sup condition} \label{subsec:inf-sup DG}

In the nonconforming case considered here, the analogue of \eqref{infsup:b1h} reads
\begin{equation} \label{infsup:b2h}
\inf_{q_h\in Q_h}\sup_{\vv_h\in V_h}\frac{b_{2h}(\vv_h,q_h)}{\|\vv_h\|_{V_h}\|q_h\|_{L^2(\Omega)}}\geq \beta^{*}
\end{equation}
with a constant $\beta^*>0$ independent of $h$. To check this condition, recall Fortin's lemma; see for instance~\cite{ref:GiR}.
\begin{lem} \label{thm:lemFort}
	The discrete condition \eqref{infsup:b2h} holds uniformly with respect to $h$ if, and only if, there exists an approximation operator $\Pi_h \in {\mathcal L}( H^1_0(\Omega)^d ;V_h)$ such that, for all $\vv \in H^1_0(\Omega)^d$,
	\begin{equation} \label{eqn:Fort1}
	b_{2h}(\Pi_h(\vv) -\vv,q_h) = 0 \quad \forall\, q_h \in Q_h,
	\end{equation}
	and
	\begin{equation} \label{eqn:Fort2}
 	\|\Pi_h(\vv)\|_{V_h} \le C |\vv|_{H^1(\Omega)}
 	\end{equation}
	with a constant $C$ independent of $h$.
\end{lem}
Originally, Fortin's lemma was stated for discrete functions in subspaces of $H^1_0(\Omega)^d$, but the extension to spaces of discontinuous functions is straightforward, as long as the form $b_{2h}(\cdot,\cdot)$ is consistent with the divergence, which is the case here.

As the proof of \eqref{eqn:Fort1}, \eqref{eqn:Fort2} is fairly technical, we restrict the discussion to the first order case,  i.e., $k=1$, in hexahedra. The  quadrilateral case is much simpler.


\subsubsection{The inf-sup condition in planar-faced  hexahedra for $k=1$} \label{subsubsec:hexainf-sup}

The construction of a suitable operator $\Pi_h$ is usually done by correcting a good approximation operator  $R_h$. For instance, we can use the $L^2$ projection onto the space of polynomials of degree one defined locally in each element, so that  $R_h(\vv)$ belongs to $V_h$ and satisfies optimal approximation properties; see for instance~\cite{BreSco}. Then $R_h(\vv)$ is corrected by constructing $\vc_h \in V_h$ such that
\begin{equation} \label{eqn:corrct1}
b_{2h}(\vc_h,q_h) =  b_{2h}(R_h(\vv) -\vv,q_h)\quad \forall\, q_h \in Q_h.
\end{equation}
By expanding $b_{2h}$ and denoting by $q_E$ the value of $q_h$ in $E$, \eqref{eqn:corrct1} reads
$$ -\sum_{E \in \Th} q_E \int_E \di(\vc_h) + \sum_{e \in \Gamma_h} \int_e [\vc_h]_e\cdot \vn_e \{q_h\}_e =
-\sum_{E \in \Th} q_E \int_E \di(R_h(\vv) -\vv) + \sum_{e \in \Gamma_h} \int_e [R_h(\vv) -\vv]_e\cdot \vn_e \{q_h\}_e. $$
Green's formula in each element yields
\begin{equation} \label{eqn:corrct2}
-\sum_{E \in \Th} q_E \int_{\partial E} \vc_h\cdot  \vn_E + \sum_{e \in \Gamma_h} \int_e [\vc_h]_e\cdot \vn_e \{q_h\}_e =
-\sum_{E \in \Th} q_E \int_{\partial E} (R_h(\vv) -\vv)\cdot  \vn_E + \sum_{e \in \Gamma_h} \int_e [R_h(\vv) -\vv]_e\cdot \vn_e \{q_h\}_e
\end{equation}
with $\vn_E$ the unit exterior normal to $E$. Consider now an interior face $e$ shared by $E_1$ and $E_2$, so that $\vn_e$ is interior to $E_2$;  the contribution of $e$ to the left-hand side of \eqref{eqn:corrct2} is
$$ -q_{E_1} \int_e \restriction{\vc_h}{E_1}\cdot \vn_e + q_{E_2} \int_e \restriction{\vc_h}{E_2}\cdot \vn_e  + \int_e \frac{1}{2}
(q_{E_1} + q_{E_2}) (\restriction{\vc_h}{E_1}-\restriction{\vc_h}{E_2})\cdot \vn_e = - \int_e [q_h]_e\{\vc_h\}_e \cdot \vn_e $$
with a similar contribution to the right-hand side. Notice also that the contribution of a boundary face $e\in\Gamma_h^b$ is equal to zero on both sides of \eqref{eqn:corrct2}.
Therefore a sufficient condition for \eqref{eqn:corrct2} is that
\begin{equation}
\label{eqn:corrct3}
\int_e \restriction{\vc_h}{E} \cdot \vn_e = \int_e \restriction{(R_h(\vv) - \vv)}{E} \cdot \vn_e.
\end{equation}
We will thus construct $\vc_h\in V_h$ by imposing \eqref{eqn:corrct3} for each element $E\in\Th$ and each face $e\in\partial E$. To simplify the notation, we will write  from now on  $\vc_h$ and $(R_h(\vv)-\vv)$ instead of $\restriction{\vc_h}{E}$ and $\restriction{(R_h(\vv) - \vv)}{E}$, respectively.


Let $E$ be an arbitrary hexahedral element of  $\Th$ with faces $e_i$, centre of face $\vb_i$, and exterior unit normal $\vn_i$, $1 \le i \le 6$. To be specific, let $\va_i$, $i=1,2,3,4$, be the vertices of $e_1$, $\va_i$, $i=1,3,5,6$, the vertices of $e_2$, $\va_i$, $i=1,2,5,7$, the vertices of $e_3$, $\va_i$, $i=5,6,7,8$, the vertices of $e_4$, $\va_i$, $i=2,4,7,8$, the vertices of $e_5$, and $\va_i$, $i=3,4,6,8$, the vertices of $e_6$. The ordering of the nodes is illustrated in Figure \ref{fig:3D_reference}. Note that for $i=1,2,3$, $e_{i+3}$ is the face opposite to $e_i$, opposite in the sense that its intersection with $e_i$ is empty.

Without loss of generality, we assume that the vertex $\va_1$ is located at the origin and that the face $e_1$ lies on the $x_3=0$ plane. Indeed, this situation can be obtained via a rigid motion (translation plus rotation), which preserves all normal vectors. Therefore, the normal to the face $e_1$ is parallel to the $x_3$ axis.  Now, the idea is to transform $E$ onto a ``reference'' element $\hat E$ by an affine mapping ${\mathcal F}_E$ so that the subtetrahedron $S_1$ of $E$ based on $e_1$ and containing the origin $\va_1$ is mapped onto the unit tetrahedron $\hat S_1$. More precisely,  as $e_2$  and $e_3$ are both adjacent to $e_1$,  $S_1$ is the subtetrahedron with vertices $\va_1$, $\va_2$, $\va_3$, and $\va_5$, and
$\hat S_1$ has vertices $\hat \va_1 = (0,0,0)$, $\hat \va_2 = (0,1,0)$, $\hat \va_3 = (1,0,0)$, $\hat \va_5 = (0,0,1)$,  see Figure \ref{fig:3D_reference} for an illustration and some notation. This transformation and notation will be used till the end of this subsection.
It stems from the regularity of the family of triangulations that
there exists a constant $M$, independent of $E$ and $h$, such that
\begin{equation} \label{eqn:diam}
\mbox{diameter} (\hat E)  \le M.
\end{equation}

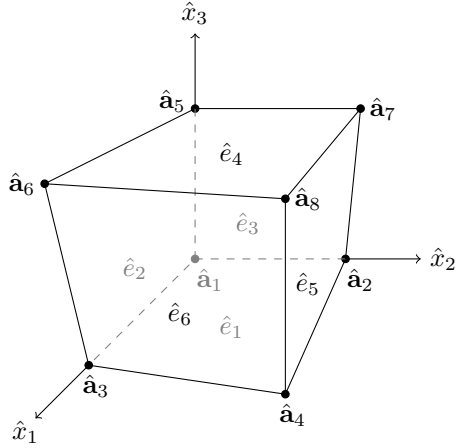
\begin{figure}[htbp]
\begin{center}
\begin{tikzpicture}
\draw [fill,gray] (0,0) circle [radius=0.05];
\draw (0.2,0) node[below,gray]{$\hat \va_1$};

\draw [fill] (2,0) circle [radius=0.05];
\draw (2.2,0) node[below]{$\hat \va_2$};

\draw [fill] (-1.41421,-1.41421) circle [radius=0.05];
\draw (-1.31421,-1.41421) node[below]{$\hat \va_3$};

\draw [fill] (1.2,-1.8) circle [radius=0.05];
\draw (1.3,-1.8) node[below]{$\hat \va_4$};

\draw [fill] (0,2) circle [radius=0.05];
\draw (0,2.1) node[left]{$\hat \va_5$};

\draw [fill] (-2.0,1.0) circle [radius=0.05];
\draw (-2.0,1.0) node[left]{$\hat \va_6$};

\draw [fill] (2.2,2.0) circle [radius=0.05];
\draw (2.2,2.0) node[right]{$\hat \va_7$};

\draw [fill] (1.2,0.8) circle [radius=0.05];
\draw (1.2,0.8) node[right]{$\hat \va_8$};

\draw (-1.41421,-1.41421) -- (1.2,-1.8) -- (2,0);
\draw[dashed,gray] (2,0) -- (0,0) -- (-1.41421,-1.41421); 
\draw (0.2,-0.9) node[right,gray]{$\hat e_1$};

\draw (-1.41421,-1.41421) -- (-2.0,1.0) -- (0,2);
\draw[dashed,gray] (0,2) -- (0,0); 
\draw (-0.8,-0.4) node[above,gray]{$\hat e_2$};

\draw (0,2) -- (2.2,2.0) -- (2,0); 
\draw (0.7,0.2) node[above,gray]{$\hat e_3$};

\draw (-2.0,1.0) -- (1.2,0.8) -- (2.2,2.0); 

\draw (1.2,-1.8) -- (1.2,0.8); 

\draw (0.2,1.4) node[right]{$\hat e_4$};
\draw (1.5,-0.6) node[above]{$\hat e_5$};
\draw (-0.2,-1.0) node[above]{$\hat e_6$};

\draw[->] (-1.41421,-1.41421) -- (-2.12132,-2.12132);
\draw (-1.95,-2.3) node[left]{$\hat x_1$};
	
\draw[->] (2,0) -- (3,0);
\draw (3,0) node[right]{$\hat x_2$};

\draw[->] (0,2) -- (0,3);
\draw (0,3) node[above]{$\hat x_3$};

\end{tikzpicture}
\caption{Some notation for the ``reference" element $\hat E$.} \label{fig:3D_reference}
\end{center}
\end{figure}

The  affine mapping  ${\cal F}_E$ has the expression
$$ \vx = {\cal F}_E(\hat \vx) = \vB \hat \vx, $$
where  the constant term is zero since $\va_1$ is the origin, and the matrix $\vB$ is nonsingular;  its columns are respectively $\va_3=(a_3^1,a_3^2,0)^t$, $\va_2=(a_2^1,a_2^2,0)^t$ and  $\va_5 = (a_5^1,a_5^2,a_5^3)^t$.
The image of the remaining vertices of $E$ are $\hat \va_i = {\cal F}^{-1}_E(\va_i)$, $i=4, 6, 7, 8$. As ${\cal F}_E$ is an affine transformation, it transforms faces onto faces, edges onto edges, and vertices onto vertices. Thus, since $\va_4$ is in the plane $x_3 = 0$, then $\hat\va_4$ is in the plane $\hat x_3 = 0$. Likewise, $\hat \va_6$ is in the plane $\hat x_2 = 0$, $\hat \va_7$ in the plane $\hat x_1 = 0$, and $\hat \va_8$ in the plane determined by $\hat \va_4$, $\hat \va_2$, $\hat \va_7$, as well as the plane determined by $\hat \va_7$, $\hat \va_5$, $\hat \va_6$, and  the plane determined by $\hat \va_6$, $\hat \va_3$, $\hat \va_4$, hence in the intersection of these three planes.  Therefore $\hat E$ is located in the first octant of $\mathbb{R}^3$.
Let $\hat\vn_i$ denote the unit exterior normal vector to $\hat e_i$. It is related to $\vn_i$ by the general formula
\begin{equation} \label{eqn:def_hat_n_3D}
\hat \vn_i = \frac{ \vB^t \vn_i }{|\vB^t \vn_i|}.
\end{equation}

\noindent The advantage of having $e_1$ on the plane $x_3 =0$ is that $\hat \vn_1 =\vn_1=(0,0,-1)^t$.
We also have $\hat \vn_2=(0,-1,0)^t$, and $\hat \vn_3=(-1,0,0)^t$.  Thus
\begin{equation} \label{eq:nu03D1}
|\hat n_1^3| =  |\hat n_2^2|= |\hat n_3^1| =1,
\end{equation}
and the regularity of the family $\Th$ implies that there exists a constant $\nu_0$, independent of $h$ and  $E$, such that
\begin{equation} \label{eq:nu03D}
 |\hat n_4^3|,  |\hat n_5^2|,  |\hat n_6^1| \ge \nu_0.
\end{equation}

With this transformation, and after cancelling $| {\rm det} \vB|$ on both sides, \eqref{eqn:corrct3} reads locally
$$ \int_{\hat e} \hat \vc_{ h } \cdot (\vB^t)^{-1}\hat \vn_{\hat e} = \int_{\hat e} (\widehat{R_h}(\hat \vv) - \hat \vv) \cdot (\vB^t)^{-1}\hat \vn_{\hat e} ,
$$
where the hat denotes composition with ${\cal F}_E$. Thus, by performing the change of variable
$$\hat \vd_h = \vB^{-1}\hat \vc_{ h }
$$
and defining the face moment
$$ m_{\hat e} (f)  := \frac{1}{|\hat e|} \int_{\hat e} f,
$$
\eqref{eqn:corrct3} is equivalent to
\begin{equation}
\label{eqn:corrct43D}
m_{\hat e_i}(\hat \vd_{ h }) \cdot \hat \vn_i = \hat g_i : = \frac{1}{|\hat e_i|} \int_{\hat e_i} \vB^{-1}(\widehat{R_h}(\hat \vv) - \hat \vv) \cdot \hat \vn_i,\quad 1 \le i \le 6.
\end{equation}
This is a linear system of six equations in twelve unknowns, the coefficients of  $\hat \vd_{ h }$.  Therefore, we can freely choose six coefficients and we have the following existence lemma.

\begin{lem}
	\label{lem:const-correction}
	There exists exactly one polynomial vector $\hat \vd_h = (\hat d_1,\hat d_2,\hat d_3)^{ t }$ that satisfies \eqref{eqn:corrct43D} and the following six conditions:
	\begin{equation}
	\label{eq:zero.d_h}
	m_{\hat e_1} (\hat d_1)  = m_{\hat e_5} (\hat d_1)= m_{\hat e_1} (\hat d_2)= m_{\hat e_6} (\hat d_2)=m_{\hat e_5} (\hat d_3)=m_{\hat e_6} ( \hat d_3) =0.
	\end{equation}
\end{lem}

\begin{proof}
	Once the six conditions \eqref{eq:zero.d_h} are prescribed, we are left with a square linear system of six equations in six unknowns. Therefore it suffices to prove that the only solution of the corresponding homogeneous system is the zero solution.
	To begin with, we consider the lines $i=5$ and $i=6$ in \eqref{eqn:corrct43D}. In view of \eqref{eq:nu03D1} and \eqref{eq:nu03D}, the strategy for the choice 
	\eqref{eq:zero.d_h}	is to set to zero the coefficients of $\hat n_6^2$ and $\hat n_6^3$ and those of $\hat n_5^1$ and $\hat n_5^3$, i.e., prescribe $m_{\hat e_5} (\hat d_1)= m_{\hat e_6} (\hat d_2)=m_{\hat e_5} (\hat d_3)=m_{\hat e_6} ( \hat d_3) =0$. With this assumption, the lines  $i=6$ and $i=5$ reduce respectively to
	\begin{equation}
	\label{eq:i=6-5}
	m_{\hat e_6}(\hat d_1) = 0,\quad  m_{\hat e_5}(\hat d_2) = 0.
	\end{equation}
	Next, we consider the line $i=1$. As $\hat n_1^3 =  -  1$ is the only nonzero component, it reduces to
	\begin{equation}
	\label{eq:i=1}
	m_{\hat e_1}(\hat d_3) =  0.
	\end{equation}
	Similarly, when $i=2$ and $i=3$ we have, respectively
	\begin{equation}
	\label{eq:i=2}
	m_{\hat e_2}(\hat d_2) =  0, \quad m_{\hat e_3}(\hat d_1) = 0.
	\end{equation}
	Collecting these results  and the two extra assumptions $m_{\hat e_1} (\hat d_1)  = m_{\hat e_1} (\hat d_2)=0$ in \eqref{eq:zero.d_h}, we find that
	$$m_{\hat e_1}(\hat d_1) = m_{\hat e_5}(\hat d_1) = m_{\hat e_6}(\hat d_1) =
	m_{\hat e_3}(\hat d_1) =  0$$
	$$m_{\hat e_1}(\hat d_2) = m_{\hat e_5}(\hat d_2) = m_{\hat e_6}(\hat d_2) =
	m_{\hat e_2}(\hat d_2)  =0.
	$$
	The three faces $\hat e_1$, $\hat e_5$, $\hat e_6$ share the vertex  $\hat\va_4$, and the regularity of the hexahedron implies that the three vectors along the segments $[\hat\va_4,\hat\va_3]$,  $[\hat\va_4,\hat\va_8]$, and $[\hat\va_4,\hat\va_2]$ is a set of three linearly independent vectors of $\polR^3$.  Then the regularity of the hexahedron implies that a polynomial of degree one is uniquely determined by its moments on the four faces $\hat e_1$, $\hat e_5$, $\hat e_6$,  $\hat e_i$ for any $i$ in the set $\{2,3,4\}$.  
	Hence, as $\hat d_1$ (respectively, $\hat d_2$)  is a polynomial of degree one, the first set (respectively, second set) of equalities and the regularity of the hexahedron imply that $\hat d_1 = 0$, respectively,  $\hat d_2=0$. When $i=4$, this leads to $m_{\hat e_4}(\hat d_3) = 0$. Consequently,
	$$m_{\hat e_1}(\hat d_3) = m_{\hat e_5}(\hat d_3) =  m_{\hat e_6}(\hat d_3) = m_{\hat e_4}(\hat d_3)  =0,
	$$
	and $\hat d_3=0$. Thus $\hat \vd_h = { \bf 0}$ and the system has a unique solution.
\end{proof}

Let $\vM_{\hat E}$ be the $6 \times 6$ matrix of the system \eqref{eqn:corrct43D} under the restriction \eqref{eq:zero.d_h}. It stems from Lemma \ref{lem:const-correction} that $\vM_{\hat E}$ is 
nonsingular. Furthermore, the regularity of the hexahedron implies that $\vM_{\hat E}$ is a continuous function of $\hat E$, thus continuous in a compact set of $\polR^3$. Hence the norm of its inverse is bounded by a constant $\hat C$, independent of $\hat E$,
\begin{equation}
\label{eqn:norm.invM}
|\vM^{-1}_{\hat E} | \le \hat C.
\end{equation}

The stability of the correction follows now easily.

\begin{lem}
	\label{lem:stab-correction}
	There exists a constant $\hat C$, independent of $h$ and $E$, such that for all $E$ in $\Th$ and all $e$ in $\Gamma_h$,
	\begin{equation}
	\label{eq:localhat_c}
	\|\vc_{ h }\|_{L^2(E)} \le \hat C\,h_E\, |\vv|_{H^1(E)},\quad  | \vc_{ h }|_{H^1(E)}\le \hat C |\vv|_{H^1(E)}, \quad  \bigg( \frac{\sigma_e}{h_e} \bigg)^{\frac{1}{2}} \|[\vc_{ h }]_e\|_{L^2(e)} \le  \hat C  \big(|\vv|_{H^1(E_1)} + |\vv|_{H^1(E_2)}\big),
	\end{equation}
	where $E_1$ and $E_2$ are the two elements sharing $e$, when $e$ is an interior face, and the sum is reduced to one term, namely the element $E$ adjacent to $e$, when $e$ is a boundary face.
\end{lem}

\begin{proof}
	The notation $\hat C$ below refers to different constants that are all independent of $h$ and $E$.
	Recalling \eqref{eqn:corrct43D},  \eqref{eqn:diam}  and the transformation from $\hat S_1$ onto $S_1$, we observe that, for any $i$,
	$$|\hat g_i| \le  \frac{\hat C}{\varrho_{S_1}} \|\widehat{R_h}(\hat \vv) - \hat \vv\|_{L^1(\hat e_i)}\le \frac{\hat C}{\varrho_{S_1}} \|\widehat{R_h}(\hat \vv) - \hat \vv\|_{L^2(\hat e_i)}. $$
	By a trace inequality in $\hat E$ and the approximation property of $\widehat{R_h}$ in $\hat E$, we have
	$$
	\sum_{i=1}^6 \|\widehat{R_h}(\hat \vv) - \hat \vv\|_{L^2(\hat e_i)} \le
	\hat C \|\widehat{R_h}(\hat \vv) - \hat \vv\|_{H^1(\hat E)} \le
	\hat C |\hat \vv|_{H^1(\hat E)}.
	$$
	Then, by reverting to $E$,
	$$\sum_{i=1}^6 \|\widehat{R_h}(\hat \vv) - \hat \vv\|_{L^2(\hat e_i)} \le  \hat C\frac{h_{S_1}}{|E|^{\frac{1}{2}}} |\vv|_{H^1( E)}.
	$$
	In view of \eqref{eqn:norm.invM} and the regularity of the family $\Th$,  the above relations lead to the following bound on $\hat \vd_h$:
	$$
	\|\hat \vd_h\|_{L^\infty( \hat E)} \le \frac{\hat C}{|E|^{\frac{1}{2}}}\frac{h_{S_1}}{\varrho_{S_1}} |\vv|_{H^1( E)} \le \frac{\hat C}{|E|^{\frac{1}{2}}}  |\vv|_{H^1( E)};
	$$
	with $\hat \vc_{ h } = \vB\hat \vd_h$,  this yields 
	\begin{equation}
	\label{eq:bddhatc}
	\|\vc_h\|_{L^\infty(E)} = \|\hat \vc_h\|_{L^\infty(\hat E)} \le \hat C\frac{h_{S_1}}{|E|^{\frac{1}{2}}} |\vv|_{H^1( E)}.
	\end{equation}
	Since $h_{S_1} < h_E$, we immediately deduce from \eqref{eq:bddhatc} the first two inequalities in \eqref{eq:localhat_c}. Finally, the third inequality follows from \eqref{eq:bddhatc} and
	$$ \bigg( \frac{\sigma_e}{h_e}  \bigg)^{\frac{1}{2}} \|\vc_{ h }\|_{L^2(e)} \le \bigg( \frac{\sigma_e}{h_e}  \bigg)^{\frac{1}{2}} |e|^{\frac{1}{2}} \|\vc_h\|_{L^\infty(E)}.
	$$
	That completes the proof of the lemma.
\end{proof}

As a consequence of Lemma \ref{lem:stab-correction} we have the following bounds:
\begin{equation}
\label{eq:globalhat_c}
\|\vc_{ h }\|_{L^2(\Omega)} \le \hat C\,h\, |\vv|_{H^1(\Omega)},\quad  \| \vc_{h }\|_{V_h}\le \hat C|\vv|_{H^1(\Omega)}.
\end{equation}
Finally, since the construction  of Lemma \ref{lem:const-correction}
yields a unique  correction, it is easy to check that the mapping $\vv \mapsto \vc_{ h }$  defines a linear operator from $V_{ h }$ into itself, i.e., $\vc_{ h } = \vc_{ h }(\vv)$.

On the other hand, we infer from standard approximation properties of $R_h$  and the regularity of the mesh, that
\begin{align}
\label{eq:localRh}
\begin{aligned}
\| \vv - R_h(\vv)\|_{L^2(E)} \le \hat C\,h_E\, &|\vv|_{H^1(E)},\quad  |R_h(\vv)|_{H^1(E)}\le \hat C|\vv|_{H^1(E)},\\
\bigg(\frac{\sigma_e}{h_e} \bigg)^{\frac{1}{2}}\|[R_h(\vv)]_e\|_{L^2(e)} &\le  \hat C  \big(|\vv|_{H^1(E_1)} + |\vv|_{H^1(E_2)}\big)
\end{aligned}
\end{align}
and
\begin{equation}
\label{eq:globalRh}
\| \vv- R_h(\vv)\|_{L^2(\Omega)} \le \hat C\,h\, |\vv|_{H^1(\Omega)},\quad  \| R_h(\vv)\|_{V_h}\le \hat C|\vv|_{H^1(\Omega)}.
\end{equation}
Thus $\Pi_h(\vv) = R_h(\vv) - \vc_h(\vv)$ satisfies the conditions \eqref{eqn:Fort1} and \eqref{eqn:Fort2} of Lemma \ref{thm:lemFort}. This proves the inf-sup condition  as stated in the next theorem.

\begin{thm}
	\label{thm:inf-sup3D}
	Let the family of hexahedra $\Th$ be regular in the sense defined above. Then the form $b_{2h}$ defined in \eqref{eq:b2hnon} with the pair spaces $V_h$ and $Q_h$ for $k=1$, see \eqref{eq:Vhnon-k} and
	\eqref{eq:Qhnon-k}, satisfies the inf-sup condition \eqref{infsup:b2h} with a constant $\beta^*>0$ independent of $h$.
\end{thm}


\subsubsection{A bound on the pressure} \label{subsubsec:DG-pressurebound}

As usual, the inf-sup condition \eqref{infsup:b2h} yields a bound on the pressure. Indeed, it follows from the first equation of \eqref{pb:nonconfscheme} together with \eqref{eq:bdd.ah}, \eqref{eq:bdd.dh}, \eqref{eq:disc-Kornbis} and \eqref{eq:Lpbd} that
$$ |b_{2h}(\vv_h,p_h)| \le C\big(\|\vT_h\|_{L^2(\Omega)} +  \|\vu_h\|_{V_h}+  \|\vu_h\|_{V_h}^2 + \|\vf \|_{L^2(\Omega)}\big)  \|\vv_h\|_{V_h}. $$
Then \eqref{eqn:4rthbdd} implies, with a constant $C$ independent of $h$ (but depending on $\alpha$), that
$$ |b_{2h}(\vv_h,p_h)| \le C\, \|\vv_h\|_{V_h} \quad \forall\, \vv_h \in V_h. $$
With the inf-sup condition \eqref{infsup:b2h}, this implies that
\begin{equation} \label{eq:bddpDG}
\|p_h\|_{L^2(\Omega)} \le C,
\end{equation}
for another constant $C$ independent of $h$.

\subsubsection{Existence and convergence} \label{subsec:conv_DG}

The proof of existence of a solution of \eqref{pb:nonconfscheme} is the same as in the conforming case. Recall that the case of interest is $k=1$, which is assumed for the remainder of this subsection, but all of what follows can be straightforwardly extended to a general polynomial degree $k\ge 1$ as long as the inf-sup condition \eqref{infsup:b2h} holds. First, the problem is reduced to one equation by testing the first equation of \eqref{pb:nonconfscheme} with $\vv_h \in V_{h,0}$ and  by observing that the second equation determines  for each $\vu_h \in V_h$  a unique $\vT_h$ in $M_h$. This is expressed by writing $\vT_h = {\mathcal G}_{h,DG}(\vu_h)$. Then, \eqref{pb:nonconfscheme} is equivalent to finding a $\vu_h \in V_{h,0}$ such that
\begin{equation} \label{eq:reducedDG}
d_h(\vu_h;\vu_h,\vv_h) + b_{1h}({\mathcal G}_{h,DG}(\vu_h),\vv_h) + J_h(\vu_h,\vv_h)  =  \int_{\Omega}\vf\cdot\vv_h\quad  \forall\, \vv_h\in V_{h,0}.
\end{equation}
By means of the \emph{a priori} estimates \eqref{eqn:4rthbdd}, existence of a solution is deduced by Brouwer's fixed point theorem.

Regarding convergence, the \emph{a priori} estimates \eqref{eqn:4rthbdd} and \eqref{eq:bddpDG} together with \eqref{eq:strng.lim} imply that there exist functions $\bar \vT \in L^2(\Omega)_{{\rm sym}}^{d\times d}$, $\bar \vu \in H^1_0(\Omega)^d$, and $\bar p \in L^2_0(\Omega)$ such that, up to subsequences,
$$ \lim_{h \to 0} \|\vu_h - \bar \vu\|_{L^q(\Omega)} = 0\quad  \mbox{with }1\leq q < \infty\; \mbox{ if }\, d=2, \mbox{ and }1 \leq q < 6\; \mbox{ if }\, d=3,$$
$$ \lim_{h \to 0}\vT_h =  \bar \vT \quad \mbox{weakly in } L^2(\Omega)^{d \times d}, $$
and
$$ \lim_{h \to 0} p_h = \bar p \quad \mbox{weakly in } L^2(\Omega). $$
However, in order to pass to the limit in the equations of the scheme, following~\cite{ref:DipietroErn}, we need to introduce discrete differential operators related to distributional differential operators. These are $G_h^{\rm sym}(\vv_h)\in M_h$ and $G_h^{\rm div}(\vv_h) \in \Theta_h$,  defined for all $\vv_h \in V_h$  by, respectively,
\begin{equation} \label{eq:Ghsym}
\int_\Omega G_h^{\rm sym}(\vv_h) :  \vR_h = b_{1h}(\vR_h,\vv_h) = \sum_{E\in \Th}\int_E \vD(\vv_h):\vR_h - \sum_{e \in \Gamma_h}  \int_e [\vv_h]_e\cdot\{ \vR_h \}_e \vn_e \quad \forall\,  \vR_h \in M_h,
\end{equation}
\begin{equation} \label{eq:Ghdiv}
\int_\Omega G_h^{\rm div}(\vv_h)\,  r_h = b_{2h}(\vv_h, r_h) = - \sum_{E\in \Th}\int_E  r_h \di(\vv_h) + \sum_{e \in \Gamma_h}  \int_e [\vv_h]_e \cdot \vn_e \{ r_h \}_e \quad  \forall\, r_h \in \Theta_h,
\end{equation}
where
$$ \Theta_h =\{ \theta_h \in L^2(\Omega): \, \restriction{ \theta_h }{E} \in \polP_1\quad \forall\, E \in \Th\}. $$
The polynomial degree one in this space is convenient for proving the convergence of the nonlinear term; see \eqref{eq:Gh.deg1}.
The straightforward scaling argument used in proving Proposition \ref{prop:balance} shows that
\begin{equation} \label{eq:bddGhsym}
\|G_h^{\rm sym}(\vv_h)\|_{L^2(\Omega)} \le  C\,\|\vv_h\|_{V_h}\quad \forall\, \vv_h \in V_h,
\end{equation}
and
$$ \|G_h^{\rm div}(\vv_h)\|_{L^2(\Omega)} \le \Big(\sum_{E\in \Th} \|\di(\vv_h)\|^2_{L^2( E )}\Big)^{\frac{1}{2}} +  C\, J_h(\vv_h,\vv_h)^{\frac{1}{2}}, $$
and thus by \eqref{eq:disc-Kornbis}
\begin{equation} \label{eq:bddGhdiv}
\|G_h^{\rm div}(\vv_h)\|_{L^2(\Omega)} \le C\|\vv_h\|_{V_h}\quad \forall\, \vv_h \in V_h
\end{equation}
with different constants $C$ independent of $h$.
At the same time, this gives existence of these two operators. The next proposition relates  $G_h^{\rm sym}(\vu_h)$ and $\vD(\bar\vu)$. The proof is an easy extension of that written in~\cite{ref:DipietroErn}, but we include it below for the reader's convenience.

\begin{prop} \label{prop: Ghsym-eps(u)}
	Up to a subsequence, we have
	\begin{equation} \label{eq:convGhsym}
	\lim_{h \to 0}G_h^{\rm sym}(\vu_h) = \vD(\bar\vu)  \quad  \emph{weakly in } L^2(\Omega)^{d \times d}.
	\end{equation}
\end{prop}

\begin{proof}
	On the one hand, the bounds \eqref{eq:bddGhsym} and \eqref{eqn:4rthbdd} imply that there exists a function $\bar \vw \in L^2(\Omega)_{\rm sym}^{ d\times d}$ such that,  up to a subsequence,
	\begin{equation} \label{eq:vw}
	\lim_{h \to 0}G_h^{\rm sym}(\vu_h) = \bar\vw  \quad \mbox{weakly in } L^2(\Omega)^{d \times d}.
 	\end{equation}
	On the other hand, take any tensor $\vF$ in $H^1(\Omega)^{d \times d}_{\rm sym}$ and let $P_h^0(\vF)$ be its orthogonal $L^2(\Omega)^{d \times d}$ projection  on constants in each $E$. We have
	$$ \Big| \int_\Omega G_h^{\rm sym}(\vu_h): \big(\vF - P_h^0(\vF)\big) \Big| \le  C \|\vu_h\|_{V_h} \|\vF - P_h^0(\vF)\|_{L^2(\Omega)}, $$
	that tends to zero with $h$.  Therefore, the definition \eqref{eq:Ghsym} of $G_h^{\rm sym}(\vu_h)$ implies that
	$$ \lim_{h \to 0} \int_\Omega G_h^{\rm sym}(\vu_h): \vF  = \lim_{h \to 0} b_{1h}(P_h^0(\vF),\vu_h) = \lim_{h \to 0}\big(b_{1h}(P_h^0(\vF)- \vF,\vu_h) + b_{1h}(\vF,\vu_h)\big) $$
	and a straightforward argument yields that the first term tends to zero.
	Hence
	$$ \lim_{h \to 0} \int_\Omega G_h^{\rm sym}(\vu_h): \vF =  \lim_{h \to 0}  b_{1h}(\vF,\vu_h) \quad \forall\, \vF \in H^1(\Omega)^{d \times d}_{\rm sym}. $$
	Now, an application of Green's formula in each $E$ gives
	$$ b_{1h}(\vF,\vu_h) = - \sum_{E\in \Th}\int_E \vu_h \cdot \di(\vF). $$
	Therefore
	$$ \lim_{h \to 0} \int_\Omega G_h^{\rm sym}(\vu_h): \vF = -\int_\Omega \bar \vu \cdot \di(\vF) =  \int_{\Omega} \vD(\bar\vu) :  \vF\quad \forall\, \vF \in H^1_0(\Omega)^{d \times d}_{\rm sym}. $$
	A comparison  with \eqref{eq:vw} and uniqueness of the limit yield
	$$ \vD(\bar\vu) =  \bar \vw, $$
	thus proving \eqref{eq:convGhsym}.
\end{proof}

\begin{remark} \label{rem:baruH10}
	{\rm The fact that $ \bar \vu $ belongs to $H^1_0(\Omega)^d$ is an easy consequence of the above proof.}
\end{remark}

A similar argument to the one in Proposition \ref{prop: Ghsym-eps(u)} gives that
\begin{equation} \label{eq:limGhdiv}
\lim_{h \to 0}G_h^{\rm div}(\vu_h) = \di(\bar\vu)  \quad \mbox{weakly in } L^2(\Omega).
\end{equation}
Hence, by passing to the limit in the last equation of \eqref{pb:nonconfscheme}, we immediately deduce that $\di( \bar \vu ) = 0$; thus $ \bar \vu $ belongs to $\Vdiv$  and satisfies the third equation of \eqref{pb:weak_cont}.

In the next theorem, these results are used to show that the limit satisfies the remaining equations of \eqref{pb:weak_cont}.

\begin{thm} \label{thm:identification}
	Let the family of hexahedra $\Th$ be regular in the sense defined above. Then the triple $( \bar \vT, \bar \vu , \bar p)$ solves \eqref{pb:weak_cont}.
\end{thm}

\begin{proof}
	The proof proceeds in two steps.

	\textbf{Step 1}. Let us start with the first equation of \eqref{pb:nonconfscheme}. Take a function $\vv \in \calD(\Omega)^d$ and let $\vv_h \in V_h$ be the $L^2(\Omega)^d$ orthogonal projection of $\vv$ on $\polP_1^d$ in each element. It is easy to check that
	$$ \lim_{h \to 0}G_h^{\rm sym}(\vv_h) = \vD(\vv) \quad \mbox{strongly in } L^2(\Omega)^{d\times d}. $$
	Therefore the weak convergence of $\vT_h$ and the definition of $G_h^{\rm sym}(\vv_h)$ imply that
	$$ \lim_{h \to 0} b_{1h}(\vT_h,\vv_h)  = \int_{\Omega} \vD(\vv):\bar \vT. $$
	Similarly,
	$$ \lim_{h \to 0}b_{2h}(\vv_h,p_h) = -\int_\Omega \bar p \di(\vv). $$
	Also
	$$ \lim_{h \to 0} J_h(\vu_h,\vv_h) = 0. $$
	As the right-hand side tends to $\int_{\Omega}\vf\cdot\vv$, it remains to examine $d_h(\vu_h;\vu_h,\vv_h)$. Recall that
	$$ d_h(\vu_h;\vu_h,\vv_h) = \sum_{E \in \Th} \int_E  \left[ (\vu_h\cdot\nabla)\vu_h  \right] \cdot \vv_h -  \frac{1}{2} b_{2h}(\vu_h, \vu_h \cdot \vv_h)- \sum_{e \in  \Gamma_h^i} \int_e\{\vu_h\}_e\cdot \vn_e [\vu_h]_e \cdot \{\vv_h\}_e. $$
	Thanks to the antisymmetry of $d_h$, we have
	\begin{equation} \label{eq:dh.deg1}
	d_h(\vu_h;\vu_h,\vv_h) = -\sum_{E \in \Th} \int_E  [ (\vu_h\cdot\nabla)\vv_h  ] \cdot \vu_h  + \frac{1}{2} b_{2h}(\vu_h, \vu_h \cdot \vv_h)+ \sum_{e \in  \Gamma_h^i } \int_e\{\vu_h\}_e\cdot \vn_e [\vv_h]_e \cdot \{\vu_h\}_e.
	\end{equation}
	For the first term, the strong convergence of $\vu_h$ in $L^4(\Omega)^d$ and the strong convergence of the broken gradient $\nabla_h\vv_h$ in $L^2(\Omega)^{ d\times d }$ imply that
	$$ -\lim_{h \to 0} \sum_{E \in \Th} \int_E [ (\vu_h\cdot\nabla)\vv_h  ] \cdot \vu_h = -\int_\Omega [ (\bar \vu\cdot\nabla)\vv ] \cdot \bar \vu = \int_\Omega  [ (\bar \vu\cdot\nabla) \bar \vu   ] \cdot  \vv , $$
	since $\bar \vu \in \Vdiv$. For the second term, take any piecewise constant  approximation $\bar \vv_h$ of $\vv$. Then
	$$ b_{2h}(\vu_h, \vu_h \cdot \vv_h) = b_{2h}(\vu_h, \vu_h \cdot (\vv_h-\bar \vv_h)) + b_{2h}(\vu_h, \vu_h \cdot \bar \vv_h). $$
	The boundedness of $\vu_h$ in $V_h$ and the convergence to zero of $\vv_h-\bar \vv_h$ in $L^\infty (\Omega)^d$ imply that the first term tends to zero. For the second term, we deduce from the definition of $G_h^{\rm div}(\vu_h)$ that
	\begin{equation} \label{eq:Gh.deg1}
	b_{2h}(\vu_h, \vu_h \cdot \bar \vv_h) = \int_\Omega  G_h^{\rm div}(\vu_h) (\vu_h \cdot \bar \vv_h).
	\end{equation}
	As $\di(\bar \vu) = 0$, $G_h^{\rm div}(\vu_h)$ tends to zero weakly in $L^2(\Omega)$. Then the strong convergence of $\vu_h$ in $L^2(\Omega)^d$ and that of $\bar \vv_h$ in $L^\infty (\Omega)^d$ show that this second term tends to zero. It remains to examine the last term of \eqref{eq:dh.deg1}. Here we use the fact that, for any $\vv \in W^{2,\infty}(\Omega)^d$,
	$$ \|\vv_h -\vv\|_{L^\infty(e)} \le C h_e^2 | \vv|_{W^{2,\infty}(\Omega)}. $$
	This, with the boundedness of $\vu_h$ in $V_h$, gives that this last term tends to zero.
	Thus, we conclude that
	$$ \lim_{h \to 0}d_h(\vu_h;\vu_h,\vv_h) = \int_\Omega [(\bar \vu\cdot\nabla)\bar \vu]\cdot \vv \quad \forall\, \vv \in W^{2,\infty}(\Omega)^d\cap H^1_0(\Omega)^d. $$
	The conclusion of these limits and a density argument is that  the triple $(\bar \vT,\bar \vu,\bar p)$ satisfies the first equation of \eqref{pb:weak_cont}
	\begin{equation} \label{eq:limit.balance}
	\int_\Omega  [ (\bar \vu\cdot\nabla)\bar \vu  ] \cdot  \vv + \int_{\Omega}  \bar \vT:\vD(\vv) -\int_\Omega \bar p \di(\vv) =  \int_{\Omega}\vf\cdot\vv \quad \forall\, \vv \in H^1_0(\Omega)^d.
	\end{equation}

	\textbf{Step 2}. The argument for recovering the constitutive relation $\bar \vT = {\cal G}(\bar \vu)$ is close to that for the conforming case, up to some changes. On the one hand, we observe that
 	$$ \lim_{h \to 0} \big( b_{1h}(\vT_h,\vu_h) + J_h(\vu_h, \vu_h) \big) = \int_\Omega \vf \cdot \bar \vu $$
	and, since $J_h(\vu_h, \vu_h)$ is positive and bounded, this implies that
 	$$ \lim_{h \to 0} b_{1h}(\vT_h,\vu_h) \le \int_\Omega \vf \cdot \bar \vu. $$
	On the other hand, we infer from \eqref{eq:limit.balance} that
 	$$ \int_{\Omega} \bar \vT:\vD(\bar \vu) = \int_{\Omega}\vf\cdot \bar \vu. $$
	Hence
	\begin{equation} \label{eq:limit.ah}
	\lim_{h \to 0} b_{1h}(\vT_h,\vu_h) \le \int_{\Omega} \vD(\bar\vu) : \bar \vT.
	\end{equation}
	Next, we set
	$$ \tilde \vT^\bd = {\cal G}(\bar \vu) $$ 
	and define $\tilde \vT_h = {\cal G}_{h,DG}(\bar \vu)$, i.e.,
	$$ \alpha\int_{\Omega}\tilde \vT_h:\vS_h+\gamma\int_{\Omega}\mu(|\tilde \vT_h|)\tilde \vT_h:\vS_h  =   b_{1h}(\vS_h,\bar\vu)  = \int_{\Omega} \vD(\bar \vu) : \vS_h \quad \forall\, \vS_h\in M_h, $$
	where the second equality holds thanks to the fact that $\bar \vu$ belongs to $H^1_0(\Omega)^d$.
	The fact that $\di(\bar \vu) = 0$ implies that the trace of $\tilde \vT^{\bd}$ is zero and justifies the above superscript. Therefore
	$$ \alpha \int_\Omega (\tilde\vT_h-\tilde\vT^\bd):\vS_h + \gamma\int_{\Omega}\big(\mu(|\tilde\vT_h|)\tilde\vT_h -\mu(|\tilde\vT^\bd|)\tilde\vT^\bd\big) :\vS_h = 0\quad \forall\, \vS_h \in M_h, $$
	and, as in the conforming case, we conclude that
	\begin{equation} \label{eq:limit.Th-tildTh}
	\lim_{h \to 0} \|\tilde\vT_h-\tilde\vT^\bd\|_{L^2(\Omega)} =0.
	\end{equation}
	Finally, the difference between the equations satisfied by $\vT_h$ and $\tilde\vT_h$ yields
	$$ \alpha \int_\Omega (\vT_h-\tilde\vT_h):\vS_h + \gamma\int_{\Omega}\big(\mu(|\vT_h|)\vT_h - \mu(|\tilde\vT_h|)\tilde\vT_h\big) :\vS_h = b_{1h}(\vS_h,\vu_h) - \int_{\Omega} \vD(\bar \vu) :  \vS_h\quad \forall\, \vS_h \in M_h. $$
	By testing this equation with $\vS_h = \vT_h-\tilde\vT_h$ and using the monotonicity property \eqref{eqn:mu_mon}, we deduce that
	\begin{equation} \label{eq:Th-tildTh}
	\alpha \| \vT_h-\tilde\vT_h\|^2_{L^2(\Omega)} \le b_{1h}(\vT_h,\vu_h)- b_{1h}(\tilde\vT_h,\vu_h) - \int_{\Omega} \vD(\bar \vu) :(\vT_h-\tilde\vT_h).
	\end{equation}
	However, by \eqref{eq:Ghsym},
	$$ b_{1h}(\tilde\vT_h,\vu_h) = \int_\Omega G_h^{\rm sym}(\vu_h) : \tilde\vT_h, $$
	and it follows from Proposition \ref{prop: Ghsym-eps(u)} and \eqref{eq:limit.Th-tildTh} that
	$$ \lim_{h \to 0} b_{1h}(\tilde\vT_h,\vu_h) = \int_{\Omega} \vD(\bar \vu) :\tilde\vT^\bd. $$
	Then, by passing to the limit in \eqref{eq:Th-tildTh}, we obtain in view of \eqref{eq:limit.ah} the inequality
	$$ \alpha \lim_{h \to 0}  \| \vT_h-\tilde\vT_h\|^2_{L^2(\Omega)} \le \int_{\Omega} \vD(\bar \vu) :\bar\vT - \int_{\Omega} \vD(\bar \vu) :\tilde\vT^\bd -  \int_{\Omega} \vD(\bar \vu) :(\bar \vT-\tilde\vT^\bd)  =0, $$
	whence
	$$ \lim_{h \to 0}  \| \vT_h-\tilde\vT_h\|_{L^2(\Omega)} =0, $$
	and uniqueness of the limit yields
	$$ \bar\vT = \tilde\vT^\bd = {\cal G}(\bar \vu). $$
	This proves that $(\bar\vT, \bar \vu)$ satisfies the second equation of \eqref{pb:weak_cont}.
\end{proof}

	
\subsection{The tetrahedral case} \label{subsec:tetra}

Here we study briefly two examples of finite element discretisations on tetrahedral meshes, the triangular case being simpler. Many of the details are skipped because they follow closely those in the previous subsection. The family of meshes $\Th$ is assumed to be regular as in \eqref{eq:reg_mesh}. Let us start with the same spaces $V_h$, $Q_h$, and $M_h$ defined on $\Th$ by \eqref{eq:Vhnon-k}, \eqref{eq:Qhnon-k}, and \eqref{eq:Mhnon-k}, respectively, and the same bilinear forms $b_{1h}(\vS_h,\vv_h)$, $d_h(\vu_h;\vv_h,\vw_h)$, and $b_{2h}(\vv_h,q_h)$ defined by  \eqref{eq:epshnon}, \eqref{eq:dhnon}, and  \eqref{eq:b2hnon}, respectively. Then the scheme is again given by \eqref{pb:nonconfscheme} and, under assumption \eqref{eq:reg_mesh}, all proofs from the previous subsections are valid in this case, except possibly the proof of the inf-sup condition. In fact, Theorem \ref{thm:inf-sup3D} holds with a much simpler proof. Indeed, take any tetrahedron $E$. Recalling that the case of interest is $k =1$, a polynomial of $\polP_1$ is uniquely determined in $E$ by its values at the centre points $\vb_e$ of its four faces $e$. Then, instead of \eqref{eqn:corrct3}, we can use the sufficient condition
\begin{equation} \label{eqn:corrct3tetra}
\restriction{\vc_h(\vb_e)}{E}= \frac{1}{|e|} \int_e \restriction{(R_h(\vv) - \vv)}{E} \quad  \forall\, E \in \Th, \quad \forall\, e \in \partial E,
\end{equation}
and this defines uniquely the correction $\vc_h$. Furthermore, thanks to \eqref{eq:reg_mesh}, the stability of this correction follows from the fact that $E$ is the image of the unit tetrahedron $\hat E$ by an invertible affine mapping whose matrix satisfies the same properties as the matrix $\vB$ used above. Thus the conclusion of Theorem \ref{thm:inf-sup3D} is valid in this case.

\medskip

As a second example, it would be tempting to use the Crouzeix--Raviart element of degree one on tetrahedra; see~\cite{CroRa}. This would be possible if the analysis did not invoke Korn's inequality (with respect to the broken symmetric gradient), because it is not satisfied by the Crouzeix--Raviart element; cf~\cite{ref:falk91}. Thus, the simplest way to bypass this difficulty is to introduce the jump penalty term $J_h(\vu_h,\vv_h)$ defined in \eqref{eq:Jh}. Let us describe this discretisation.
Again, we suppose that \eqref{eq:reg_mesh} holds. The discrete spaces $Q_h$ and $M_h$ are the same, with $k=1$, as in  \eqref{eq:Qhnon-k} and \eqref{eq:Mhnon-k}, respectively. However, instead of $V_h$, we now use the space $V_{h}^{CR}$ whose elements are also piecewise polynomials of degree one in each element, but in contrast with \eqref{eq:Vhnon-k}, they  are continuous at the centre points of all interior faces $e \in \Gamma_h^i$, and are set to zero at the centre points of all boundary faces $e \in \Gamma_h^b$. Thanks to this pointwise continuity and boundary condition,  the scheme now involves the following bilinear/trilinear forms, compare with \eqref{eq:epshnon}, \eqref{eq:dhnon}, \eqref{eq:b2hnon}:
\begin{equation} \label{eq:epshnonCR}
\int_\Omega \vD(\vv) : \vS \simeq  b_{1h}^{CR}(\vS_h,\vv_h) := \sum_{E \in \Th} \int_E \vD(\vv_h) : \vS_h ,
\end{equation}
\begin{equation} \label{eq:dhnonCR}
d_{h}^{CR}(\vu_h;\vv_h,\vw_h)  :  = \frac{1}{2}\Big[ \sum_{E \in \Th} \int_E  \left[  (\vu_h\cdot\nabla)\vv_h  \right]  \cdot \vw_h - \sum_{E \in \Th} \int_E   \left[  (\vu_h\cdot\nabla)\vw_h  \right]  \cdot \vv_h \Big],
\end{equation}
\begin{equation} \label{eq:b2hnonCR}
b_{2h}^{CR}(\vv_h,q_h)  :  = -\sum_{E \in \Th} \int_E q_h \di(\vv_h).
\end{equation}
With these new forms, analogously to \eqref{pb:nonconfscheme}, the finite element approximation of the problem reads as follows: find a triple
$(\vT_h,\vu_h,p_h) \in  M_h  \times  V_{h}^{CR} \times Q_h$ such that
\begin{alignat}{2} \label{pb:nonconfschemeCR}
\begin{aligned}
d_{h}^{CR}(\vu_h;\vu_h,\vv_h) + b_{1h}^{CR}(\vT_h,\vv_h) + b_{2h}^{CR}(\vv_h,p_h) + J_h(\vu_h,\vv_h) & =  \displaystyle{\int_{\Omega}}\vf\cdot\vv_h &&\quad \forall\, \vv_h\in V_{h}^{CR}, \\
\alpha\displaystyle{\int_{\Omega}}\vT_h:\vS_h+\gamma\displaystyle{\int_{\Omega}}\mu(|\vT_h|)\vT_h:\vS_h & =   b_{1h}^{CR}(\vS_h,\vu_h) &&\quad \forall\, \vS_h\in M_h, \\
b_{2h}^{CR}(\vu_h,q_h) & =  0 &&\quad \forall\, q_h\in Q_h.
\end{aligned}
\end{alignat}
Note that  $b_{1h}^{CR}(\vS_h,\vv_h)$ coincides with $b_{1h}(\vS_h,\vv_h)$ and $b_{2h}^{CR}(\vv_h,q_h)$ coincides with $b_{2h}(\vv_h,q_h)$ because the additional face terms vanish for elements of the space $V_{h}^{CR}$. This is not necessarily the case with $d_{h}^{CR}$ and $d_h$, but $d_{h}^{CR}$ is obviously antisymmetric and is simpler. Although the norm of the broken gradient is a norm on $V_{h}^{CR}$, the mapping $\vv_h \mapsto \|\vD(\vv_h)\|_h$ is not a norm on $V_{h}^{CR}$. According to~\cite{ref:SB03,ref:SB04}, we have instead \eqref{eq:disc-Korn} and \eqref{eq:disc-Kornbis}. That is why we use again the norm $\|\vv_h\|_{V_h}$ defined in \eqref{eq:normVh} and keep the term $J_h(\vu_h,\vv_h)$ in the first line of \eqref{pb:nonconfschemeCR}. Note however that the parameters $\sigma_e$ need not be tuned by Proposition \ref{prop:balance} since there are no surface terms in $b_{1h}^{CR}(\vT_h,\vv_h)$; thus it suffices for instance to take $\sigma_e = 1$ for each face $e$. Moreover, the analysis used for the general discontinuous elements substantially simplifies here. First, as there are no surface terms in the bilinear forms, the bounds are simpler. Next, the operator $\Pi_h$ satisfying the statement of Lemma \ref{thm:lemFort} is constructed directly by setting, for $\vv$ in $H^1_0(\Omega)^d$,
\begin{equation} \label{eqn:CRFortin}
\restriction{\Pi_h(\vv)(\vb_e)}{E}= \frac{1}{|e|} \int_e \vv \quad  \forall\, E \in \Th,\; \forall\,e \in \partial E,
\end{equation}
see~\cite{CroRa}. Clearly, as $\vv \in H^1_0(\Omega)^d$, \eqref{eqn:CRFortin} defines a piecewise polynomial function of degree one in $V_{h}^{CR}$. Finally, convergence of the scheme is derived without the discrete differential operators $G_h^{\rm sym}$ and $G_h^{\rm div}$. Indeed, property \eqref{eq:strng.lim} can be extended as is asserted in the proposition.
\begin{prop} \label{prop:weakconvH1}
	Let the family $\Th$ satisfy \eqref{eq:reg_mesh}. If $\vv_h$ is a sequence in $V_{h}^{CR}$ such that
	$$ \|\vv_h\|_{V_h} \le C $$
	with a constant $C$ independent of $h$, then there exists a function $\bar \vv \in H^1_0(\Omega)^d$ satisfying \eqref{eq:strng.lim} and
	\begin{equation} \label{eq:weak.limCR}
	\lim_{h \to 0} \vD_h(\vv_h)=  \vD(\bar \vv) \quad \mbox{weakly in}\ L^2(\Omega)^{d\times d},
	\end{equation}
	where $\vD_h$ stands for the broken symmetric gradient.
\end{prop}
The proof, contained in~\cite{CroRa}, relies on the fact that the integral average of the jump $[\vv_h]_e$ vanishes on any face $e$ and hence, for any tensor $\vF$ in $H^1(\Omega)^{d\times d}$,
$$ \int_e \vF \vn_e\cdot [\vv_h]_e = \int_e (\vF-\vC) \vn_e\cdot [\vv_h]_e \quad \forall\, \vC \in \polR^{d\times d}. $$
Thus, there is no need for $G_h^{\rm sym}$; the same is true for $G_h^{\rm div}$. This permits to pass directly to the limit in \eqref{pb:nonconfschemeCR}.


\section{Numerical illustrations} \label{sec:numerics}

We introduce two decoupled iterative algorithms. The first one is based on a Lions--Mercier decoupling strategy while the second one is a fixed point algorithm. All the algorithms are implemented using the \textit{deall.ii} library \cite{bangerth2007deal}. For simplicity, we focus on conforming finite element approximations for which an a priori error estimate has been derived in Subsection \ref{subsec:error_estimate}. Performing numerical experiments in the case of the nonconforming approximation scheme will be the subject of future work.

The general setup is the following:

\begin{itemize}
	\item Dirichlet boundary conditions are imposed on the entire domain boundary (not necessarily homogeneous);
	\item A sequence of uniformly refined meshes with square elements of diameter $h=\sqrt{2}/2^{n}$, $n=2,\ldots,6$ (level of refinement) are considered for the mesh refinement analysis;
	\item The finite element spaces $M_h$, $V_h$, and $Q_h$ consist,
	respectively, of discontinuous piecewise polynomials of degree 2, continuous piecewise polynomials of degree 2, and continuous piecewise polynomials of degree 1 (see Subsection \ref{ss:quad}).
\end{itemize}

\noindent Following \cite{ref:BGS18}, we replace the constitutive relation
\begin{equation*}
\alpha\vT^d+\gamma\mu(|\vT^d|)\vT^d-\vD(\vu) = \mathbf{0}
\end{equation*}
by
\begin{equation*}
\alpha\vT^d+\gamma\mu(|\vT^d|)\vT^d-\vD(\vu) = \vg
\end{equation*}
to design an exact solution. Then, given $\vT^d$, $\vu$ and $p$, we compute the corresponding right-hand sides $\vg$ and $\vf$ (forcing term), where we recall that
\begin{equation*}
\vf=(\vu\cdot\nabla)\vu-\frac{1}{\alpha}\di(\vD(\vu))+\nabla p+\frac{\gamma}{\alpha}\di(\mu(|\vT^d|)\vT^d).
\end{equation*}
Finally, we choose $\mu(s)=\frac{1}{\sqrt{1+s^2}}$ which corresponds to \eqref{eq:powerlaw3} with $\beta=1$ and $n=-1/2$.

\subsection{Lions--Mercier decoupled iterative algorithm} \label{subsec:L-M.Alg}

We present here an iterative algorithm to compute approximately the solution to problem \eqref{pb:weak_epsu}, which is based on the formulation \eqref{pb:NS_EL_v3}: find $(\vT_h,\vu_h,p_h)\in M_h\times V_h\times Q_h$ such that
\begin{align} \label{pb:weak_disc_epsu}
\begin{aligned}
d(\vu_h;\vu_h,\vv_h)+\frac{1}{\alpha}\displaystyle{\int_{\Omega}}\vD(\vu_h):\vD(\vv_h)-\displaystyle{\int_{\Omega}}p_h\di(\vv_h) & =  \displaystyle{\int_{\Omega}}\vf\cdot\vv_h+\frac{\gamma}{\alpha}\displaystyle{\int_{\Omega}}\mu(|\vT_h|)\vT_h:\vD(\vv_h), \\
\alpha\displaystyle{\int_{\Omega}}\vT_h:\vS_h+\gamma\displaystyle{\int_{\Omega}}\mu(|\vT_h|)\vT_h:\vS_h & =  \displaystyle{\int_{\Omega}}\vD(\vu_h):\vS_h, \\
\displaystyle{\int_{\Omega}}q_h\di(\vu_h) & =  0
\end{aligned}
\end{align}
for all $(\vS_h,\vv_h,q_h)\in M_h\times V_h\times Q_h$, where $d:V\times V\times V\rightarrow\mathbb{R}$ is defined in \eqref{def:form_d}. Note that problem \eqref{pb:weak_disc_epsu} is equivalent to problem \eqref{pb:weak_disc} analysed in Section \ref{sec:NS_ENL_App}.

To compute the solution to problem \eqref{pb:weak_disc_epsu}, we propose a decoupled algorithm based on a Lions--Mercier splitting algorithm \cite{LM79} (alternating-direction method of the Peaceman--Rachford type \cite{ref:PR55}) applied to the unknown $\vT_h$. Following the discussion in \cite[Section 7]{ref:BGS18}, the algorithm reads, for a pseudo-time step $\tau>0$:

\vspace*{0.2cm}
\noindent \emph{Initialisation}: find $(\vT_h^{(0)},\vu_h^{(0)},p_h^{(0)})\in M_h\times V_h\times Q_h$ such that
\begin{alignat}{2} \label{pb:init}
\begin{aligned}
d(\vu_h^{(0)};\vu_h^{(0)},\vv_h)+\frac{1}{\alpha}\displaystyle{\int_{\Omega}}\vD(\vu_h^{(0)}):\vD(\vv_h)-\displaystyle{\int_{\Omega}}p_h^{(0)}\di(\vv_h) & =  \displaystyle{\int_{\Omega}}\vf\cdot\vv_h &&\quad \forall\, \vv_h\in V_h, \\
\alpha\displaystyle{\int_{\Omega}}\vT_h^{(0)}:\vS_h & =  \displaystyle{\int_{\Omega}}\vD(\vu_h^{(0)}):\vS_h &&\quad \forall\, \vS_h\in M_h, \\
\displaystyle{\int_{\Omega}}q_h\di(\vu_h^{(0)}) & =  0 &&\quad \forall\, q_h\in Q_h.
\end{aligned}
\end{alignat}

\vspace*{0.2cm}
\noindent Then, for $k=0,1,\ldots,$ perform the following two steps:

\vspace*{0.2cm}
\noindent \emph{Step 1:} Find $\vT_h^{(k+\frac{1}{2})}\in M_h$ such that
\begin{equation*}
\frac{1}{\tau}\int_{\Omega}(\vT_h^{(k+\frac{1}{2})}-\vT_h^{(k)}):\vS_h+\gamma\int_{\Omega}\mu(|\vT_h^{(k+\frac{1}{2})}|)\vT_h^{(k+\frac{1}{2})}:\vS_h = \int_{\Omega}\vD(\vu_h^{(k)}):\vS_h-\alpha\int_{\Omega}\vT_h^{(k)}:\vS_h \quad \forall\, \vS_h\in M_h.
\end{equation*}

\vspace*{0.2cm}
\noindent \emph{Step 2:} Find $(\vT_h^{(k+1)},\vu_h^{(k+1)},p_h^{(k+1)})\in M_h\times V_h\times Q_h$ such that
\begin{alignat}{2} \label{pb:step2}
\begin{aligned}
&d(\vu_h^{(k+1)};\vu_h^{(k+1)},\vv_h)+\frac{1}{\alpha}\displaystyle{\int_{\Omega}}\vD(\vu_h^{(k+1)}):\vD(\vv_h)-\displaystyle{\int_{\Omega}}p_h^{(k+1)}\di(\vv_h)\\
& \hspace{7.1cm} =  \displaystyle{\int_{\Omega}}\vf\cdot\vv_h +\frac{\gamma}{\alpha}\displaystyle{\int_{\Omega}}\mu(|\vT_h^{(k+\frac{1}{2})}|)\vT_h^{(k+\frac{1}{2})}:\vD(\vv_h), \\
&\frac{1}{\tau}\displaystyle{\int_{\Omega}}(\vT_h^{(k+1)}-\vT_h^{(k+\frac{1}{2})}):\vS_h+\alpha\displaystyle{\int_{\Omega}}\vT_h^{(k+1)}:\vS_h  =  \displaystyle{\int_{\Omega}}\vD(\vu_h^{(k+1)}):\vS_h-\gamma\displaystyle{\int_{\Omega}}\mu(|\vT_h^{(k+\frac{1}{2})}|)\vT_h^{(k+\frac{1}{2})}:\vS_h,  \\
&\hspace{4.55cm}\displaystyle{\int_{\Omega}}q_h\di(\vu_h^{(k+1)})  =  0
\end{aligned}
\end{alignat}
for all $(\vS_h,\vv_h,q_h)\in M_h\times V_h\times Q_h$.

\vspace*{0.2cm}
The solution to \eqref{pb:init} is obtained by first determining $\vu_h^{(0)}$ and $p_h^{(0)}$ as the solution to a standard steady-state Navier--Stokes equation (first and third equations in \eqref{pb:init}) and then by setting $\vT_h^{(0)}=\frac{1}{\alpha}\vD(\vu_h^{(0)})$.
Similarly, the solution to problem \eqref{pb:step2} can be obtained by first solving the first and third equations for $\vu_h^{(k+1)}$ and $p_h^{(k+1)}$ and then solving the second equation for $\vT_h^{(k+1)}$. A standard argument shows that the above algorithm generates uniformly bounded sequences. Thus they converge up to subsequences. However, the identification of a unique limit for the entire sequence is currently unclear.

Regarding the implementation, we make the following comments:

\begin{itemize}
	\item \emph{Stopping criterion}: For the main loop (Lions--Mercier algorithm), the stopping criterion is
\begin{equation} \label{eqn:stopLM}
\frac{\|\vT_h^{(k+1)}-\vT_h^{(k)}\|_{L^2(\Omega)}+\|\nabla(\vu_h^{(k+1)}-\vu_h^{(k)})\|_{L^2(\Omega)}+\|p_h^{(k+1)}-p_h^{(k)}\|_{L^2(\Omega)}}{\|\vT_h^{(k+1)}\|_{L^2(\Omega)}+\|\nabla\vu_h^{(k+1)}\|_{L^2(\Omega)}+\|p_h^{(k+1)}\|_{L^2(\Omega)}} \leq 10^{-5};
\end{equation}
	\item \emph{Initialisation}: We solve the Navier--Stokes system associated to problem \eqref{pb:init} using Newton's method (the iterates are indexed by $m$) 
	until the following stopping criterion is met:
	\begin{equation*}
	 \frac{\|\nabla(\vu_h^{(m+1)}-\vu_h^{(m)})\|_{L^2(\Omega)}+\|p_h^{(m+1)}-p_h^{(m)}\|_{L^2(\Omega)}}{\|\nabla\vu_h^{(m+1)}\|_{L^2(\Omega)}+\|p_h^{(m+1)}\|_{L^2(\Omega)}} \leq 10^{-6}.
	\end{equation*}
	As an initial guess, we take the solution of the associated Stokes system without the convective term. 
	
	The solution to each saddle-point system of the form
	\begin{equation*}
	\left(\begin{array}{cc}
	A & B^T \\ B & 0
	\end{array}\right)\left(\begin{array}{c}
	\mathbf{U} \\ \mathbf{P}
	\end{array}\right) = \left(\begin{array}{c}
	\mathbf{F} \\ \mathbf{G}
	\end{array}\right)
	\end{equation*}
	is obtained using a Schur complement formulation  
	$$ BA^{-1}B^T \mathbf{P} = BA^{-1} \mathbf{F} - \mathbf{G}, \qquad A\mathbf{U} = \mathbf{F}-B^T \mathbf{P}. $$
	To solve for $\mathbf P$, we use the conjugate gradient algorithm in the case of the Stokes problem and GMRES for the (linearised) Navier--Stokes problems.
	In both cases, the pressure mass matrix is used as preconditioner and the tolerance for the iterative algorithm is set to $10^{-6}\|BA^{-1}\mathbf{F}-\mathbf{G}\|_{\ell_2}$.  
	A direct method is advocated for every occurrence of $A^{-1}$ and also to obtain $\vT^{(0)}$.
	\item \emph{Step 1 (monotone part)}: $\vT_h^{(k+\frac{1}{2})}$ is the zero of the functional
	\begin{equation*}
	F(\vT_h) := \vT_h + \tau\gamma\mu(|\vT_h|)\vT_h-\tau\vD(\vu_h^{(k)})-(1-\alpha\tau)\vT_h^{(k)}.
	\end{equation*}
	Recall that discontinuous piecewise polynomial approximations are used for the stress and so $\vT_h$ is determined locally on  each element $E\in\Th$ as the solution to
	\begin{equation*}
	\int_{E} F(\vT_h):\vS_h = 0 \quad \forall\, \vS_h\in\mathbb{Q}_2.
	\end{equation*}
	We again employ Newton's method starting with $\vT_h^{(0)}=\restriction{\vT_h^{(k)}}{E}$ and use the stopping criterion
	\begin{equation*}
	\|F(\vT_h^{(m)})\|_{L^2(E)} \leq 10^{-6}\frac{\sqrt{|E|}}{\sqrt{|\Omega|}}
	\end{equation*}
	so that the global residual is less than $10^{-6}$.
	Note that in this case, it might happen that no iteration is needed (e.g. when $\gamma=0$), in which case $\restriction{\vT_h^{(k+\frac{1}{2})}}{E}=\restriction{\vT_h^{(k)}}{E}$.
	\item \emph{Step 2}: This step is similar to the initialisation step except that we take $(\vu_h^{(k)},p_h^{(k)})$ as our initial guess for Newton's method for solving the finite element approximation of the Navier--Stokes system.
\end{itemize}


\subsubsection{Case 1: smooth solution}

We consider the case $\Omega=(0,1)^2$ and
\begin{equation*}
\vT^d=\left(\begin{array}{cc}
\frac{\cos(2\pi x)-\cos(2\pi y)}{4} & 0 \\ 0 & \frac{\cos(2\pi y)-\cos(2\pi x)}{4}
\end{array}\right), \quad \vu=\left(\begin{array}{r}
-\cos(\pi x)\sin(\pi y) \\ \sin(\pi x)\cos(\pi y)
\end{array}\right), \quad
p=-\frac{\cos(2\pi x)+\cos(2\pi y)}{4}.
\end{equation*}
Note that $\vT^d$ is the deviatoric part of $\vT$ defined by
\begin{equation*}
\vT=\left(\begin{array}{cc}
\frac{\cos(2\pi x)}{2} & 0 \\ 0 & \frac{\cos(2\pi y)}{2}
\end{array}\right),
\end{equation*}
and in particular it has vanishing trace. We observe that $\vu$ is divergence-free. Moreover, the pressure satisfies $p=-\frac{1}{2}\tr(\vT)$ and has zero mean. We report in Table \ref{tab:case1_gamma0} the error for each component of the solution for the case $\alpha=1$ and $\gamma=0$, while Table \ref{tab:case1_gamma1_001} contains the results for $\alpha=\gamma=1$. Note that we use the $H^1$ semi-norm for the velocity and not the (equivalent) $L^2(\Omega)^{2\times 2}$ norm of the symmetric gradient.
\begin{table}[htbp]
\begin{center}
\begin{tabular}{|c|c|c|c|c|c|}
\hline
$n$ & $h$ & $\|\vT^d-\vT_h\|_{L^2(\Omega)}$ & $\|\nabla(\vu-\vu_h)\|_{L^2(\Omega)}$ & $\|p-p_h\|_{L^2(\Omega)}$ & iter \\
\hline
\hline
2 & 0.354 & 6.04199$\times 10^{-2}$ & 7.51266$\times 10^{-2}$ & 3.02263$\times 10^{-2}$ & 1 \\
3 & 0.177 & 1.44750$\times 10^{-2}$ & 1.82293$\times 10^{-2}$ & 6.18331$\times 10^{-3}$ & 1 \\
4 & 0.088 & 3.58096$\times 10^{-3}$ & 4.52460$\times 10^{-3}$ & 1.46371$\times 10^{-3}$ & 1 \\
5 & 0.044 & 8.92901$\times 10^{-4}$ & 1.12913$\times 10^{-3}$ & 3.60874$\times 10^{-4}$ & 1 \\
6 & 0.022 & 2.23079$\times 10^{-4}$ & 2.82155$\times 10^{-4}$ & 8.99041$\times 10^{-5}$ & 1 \\
\hline
\end{tabular}
\caption{Case 1, $\alpha=1$, $\gamma=0$, $\delta=10^{-5}$, $\tau=0.01$.} \label{tab:case1_gamma0}
\end{center}
\end{table}
\begin{table}[htbp]
\begin{center}
\begin{tabular}{|c|c|c|c|c|c|}
\hline
$n$ & $h$ & $\|\vT^d-\vT_h\|_{L^2(\Omega)}$ & $\|\nabla(\vu-\vu_h)\|_{L^2(\Omega)}$ & $\|p-p_h\|_{L^2(\Omega)}$ & iter \\
\hline
\hline
2 & 0.354 & 3.57579$\times 10^{-2}$ & 8.21275$\times 10^{-2}$ & 3.01953$\times 10^{-2}$ & 183 \\
3 & 0.177 & 7.78829$\times 10^{-3}$ & 1.86706$\times 10^{-2}$ & 6.18695$\times 10^{-3}$ & 182 \\
4 & 0.088 & 2.00882$\times 10^{-3}$ & 4.55378$\times 10^{-3}$ & 1.50017$\times 10^{-3}$ & 182 \\
5 & 0.044 & 8.86597$\times 10^{-4}$ & 1.13687$\times 10^{-3}$ & 4.91418$\times 10^{-4}$ & 182 \\
6 & 0.022 & 7.66438$\times 10^{-4}$ & 3.05389$\times 10^{-4}$ & 3.45733$\times10^{-4}$ & 182 \\
\hline
\end{tabular}
\caption{Case 1, $\alpha=\gamma=1$, $\delta=10^{-5}$, $\tau=0.01$.} \label{tab:case1_gamma1_001}
\end{center}
\end{table}
We observe in Table \ref{tab:case1_gamma1_001} that all three errors are $\mathcal{O}(h^2)$. The deterioration of the convergence rate we observe for $\vT^d$ and $p$ in Table \ref{tab:case1_gamma1_001} is due to the stopping criterion. Indeed, if we use $10^{-6}$ instead of $10^{-5}$ in the stopping criterion \eqref{eqn:stopLM} for the main loop, then for $h=0.044$ ($n=5$) we need 250 iterations and we get $$ \|\vT^d-\vT_h\|_{L^2(\Omega)}=4.66898\times 10^{-4}, \quad \|\vu-\vu_h\|_{L^2(\Omega)}=1.13118\times 10^{-3} \quad \mbox{and} \quad \|p-p_h\|_{L^2(\Omega)}=3.62581\times 10^{-4}, $$
compare with the fourth row of Table \ref{tab:case1_gamma1_001}.

We give in Tables \ref{tab:case1_gamma1_005}, \ref{tab:case1_gamma1_01} and \ref{tab:case1_gamma1_05} the results obtained when a larger pseudo-time step is used.
\begin{table}[htbp]
\begin{center}
\begin{tabular}{|c|c|c|c|c|c|}
\hline
$n$ & $h$ & $\|\vT^d-\vT_h\|_{L^2(\Omega)}$ & $\|\nabla(\vu-\vu_h)\|_{L^2(\Omega)}$ & $\|p-p_h\|_{L^2(\Omega)}$ & iter \\
\hline
\hline
2 & 0.354 & 3.57161$\times 10^{-2}$ & 8.21200$\times 10^{-2}$  & 3.02043$\times 10^{-2}$ & 47 \\
3 & 0.177 & 7.74372$\times 10^{-3}$ & 1.86697$\times 10^{-2}$ & 6.18194$\times 10^{-3}$ & 47 \\
4 & 0.088 & 1.86276$\times 10^{-3}$ & 4.55240$\times 10^{-3}$ & 1.46465$\times 10^{-3}$ & 47 \\
5 & 0.044 & 4.77556$\times 10^{-4}$ & 1.13167$\times 10^{-3}$ & 3.65463$\times 10^{-4}$ & 47 \\
6 & 0.022 & 1.72916$\times 10^{-4}$ & 2.85469$\times 10^{-4}$ & 1.06995$\times 10^{-4}$ & 47 \\
\hline
\end{tabular}
\caption{Case 1, $\alpha=\gamma=1$, $\delta=10^{-5}$, $\tau=0.05$.} \label{tab:case1_gamma1_005}
\end{center}
\end{table}
\begin{table}[htbp]
\begin{center}
\begin{tabular}{|c|c|c|c|c|c|}
\hline
$n$ & $h$ & $\|\vT^d-\vT_h\|_{L^2(\Omega)}$ & $\|\nabla(\vu-\vu_h)\|_{L^2(\Omega)}$ & $\|p-p_h\|_{L^2(\Omega)}$ & iter \\
\hline
\hline
2 & 0.354 & 3.57060$\times 10^{-2}$ & 8.21133$\times 10^{-2}$ & 3.02054$\times 10^{-2}$ & 26 \\
3 & 0.177 & 7.74150$\times 10^{-3}$ & 1.86693$\times 10^{-2}$ & 6.18213$\times 10^{-3}$ & 26 \\
4 & 0.088 & 1.85887$\times 10^{-3}$ & 4.55234$\times 10^{-3}$ & 1.46384$\times 10^{-3}$ & 26 \\
5 & 0.044 & 4.63124$\times 10^{-4}$ & 1.13153$\times 10^{-3}$ & 3.61829$\times 10^{-4}$ & 26 \\
6 & 0.022 & 1.27963$\times 10^{-4}$ & 2.84928$\times 10^{-4}$  & 9.37426$\times 10^{-5}$ & 26 \\
\hline
\end{tabular}
\caption{Case 1, $\alpha=\gamma=1$, $\delta=10^{-5}$, $\tau=0.1$.} \label{tab:case1_gamma1_01}
\end{center}
\end{table}
\begin{table}[htbp]
\begin{center}
\begin{tabular}{|c|c|c|c|c|c|}
\hline
$n$ & $h$ & $\|\vT^d-\vT_h\|_{L^2(\Omega)}$ & $\|\nabla(\vu-\vu_h)\|_{L^2(\Omega)}$ & $\|p-p_h\|_{L^2(\Omega)}$ & iter \\
\hline
\hline
2 & 0.354 & 3.57028$\times 10^{-2}$ & 8.21057$\times 10^{-2}$ & 3.02063$\times 10^{-2}$ & 10 \\
3 & 0.177 & 7.73342$\times 10^{-3}$ & 1.86606$\times 10^{-2}$ & 6.18238$\times 10^{-3}$ & 7 \\
4 & 0.088 & 1.85742$\times 10^{-3}$ & 4.55172$\times 10^{-3}$ & 1.46368$\times 10^{-3}$ & 7 \\
5 & 0.044 & 4.59753$\times 10^{-4}$ & 1.13121$\times 10^{-3}$ & 3.60876$\times 10^{-4}$ & 7 \\
6 & 0.022 & 1.15437$\times 10^{-4}$ & 2.83829$\times 10^{-4}$ & 8.99203$\times 10^{-5}$ & 7 \\
\hline
\end{tabular}
\caption{Case 1, $\alpha=\gamma=1$, $\delta=10^{-5}$, $\tau=0.5$.} \label{tab:case1_gamma1_05}
\end{center}
\end{table}
We see that the larger the pseudo-time step, the fewer the number of iterations. Moreover, for all cases $\tau=0.05$, $\tau=0.1$ and $\tau=0.5$, there is no deterioration of the convergence rate in contrast to what we observed in Table \ref{tab:case1_gamma1_001} (due to the stopping criterion). 

 
\subsubsection{Case 2: non-smooth velocity}

We consider now the $L$-shaped domain $\Omega=(-1,1)^2\setminus [0,1)^2$;  we take $\vT^d$ and $p$ as above, but here
\begin{equation*}
\vu=\left(\begin{array}{r}
y(x^2+y^2)^{\frac{1}{3}} \\ -x(x^2+y^2)^{\frac{1}{3}}
\end{array}\right),
\end{equation*}
which is divergence-free.
The results when $\alpha=1$ and $\gamma=0$ are given in Table \ref{tab:case2_gamma0} while Tables \ref{tab:case2_gamma1_001} and \ref{tab:case2_gamma1_05} contain the results for the case $\alpha=\gamma=1$ with $\tau=0.01$ and $\tau=0.5$, respectively.

\begin{table}[htbp]
\begin{center}
\begin{tabular}{|c|c|c|c|c|c|}
\hline
$n$ & $h$ & $\|\vT^d-\vT_h\|_{L^2(\Omega)}$ & $\|\nabla(\vu-\vu_h)\|_{L^2(\Omega)}$ & $\|p-p_h\|_{L^2(\Omega)}$ & iter \\
\hline
\hline
2 & 0.354 & 3.65187$\times 10^{-2}$ & 3.80529$\times 10^{-2}$ & 5.22497$\times 10^{-2}$ & 1 \\
3 & 0.177 & 5.61550$\times 10^{-3}$ & 6.85310$\times 10^{-3}$ & 1.07102$\times 10^{-2}$ & 1 \\
4 & 0.088 & 1.32332$\times 10^{-3}$ & 1.86233$\times 10^{-3}$ & 2.53671$\times 10^{-3}$ & 1 \\
5 & 0.044 & 3.95496$\times 10^{-4}$ & 5.79343$\times 10^{-4}$ & 6.25652$\times 10^{-4}$ & 1 \\
6 & 0.022 & 1.24435$\times 10^{-4}$ & 1.83765$\times 10^{-4}$ & 1.55951$\times 10^{-4}$ & 1 \\
\hline
\end{tabular}
\caption{Case 2, $\alpha=1$, $\gamma=0$, $\delta=10^{-5}$, $\tau=0.5$.} \label{tab:case2_gamma0}
\end{center}
\end{table}
\begin{table}[htbp]
\begin{center}
\begin{tabular}{|c|c|c|c|c|c|}
\hline
$n$ & $h$ & $\|\vT^d-\vT_h\|_{L^2(\Omega)}$ & $\|\nabla(\vu-\vu_h)\|_{L^2(\Omega)}$ & $\|p-p_h\|_{L^2(\Omega)}$ & iter \\
\hline
\hline
2 & 0.354 & 3.51269$\times 10^{-2}$ & 6.79039$\times 10^{-2}$ & 5.22609$\times 10^{-2}$ & 198 \\
3 & 0.177 & 4.65311$\times 10^{-3}$ & 9.59150$\times 10^{-3}$ & 1.07129$\times 10^{-2}$ & 198 \\
4 & 0.088 & 1.10623$\times 10^{-3}$ & 2.04375$\times 10^{-3}$ & 2.55916$\times 10^{-3}$ & 198 \\
5 & 0.044 & 7.88026$\times 10^{-4}$ & 6.03974$\times 10^{-4}$ & 7.14457$\times 10^{-4}$ & 198 \\
6 & 0.022 & 7.63053$\times 10^{-4}$ & 2.28218$\times 10^{-4}$ & 3.79151$\times 10^{-4}$ & 198 \\
\hline
\end{tabular}
\caption{Case 2, $\alpha=\gamma=1$, $\delta=10^{-5}$, $\tau=0.01$.} \label{tab:case2_gamma1_001}
\end{center}
\end{table}
\begin{table}[htbp]
\begin{center}
\begin{tabular}{|c|c|c|c|c|c|}
\hline
$n$ & $h$ & $\|\vT^d-\vT_h\|_{L^2(\Omega)}$ & $\|\nabla(\vu-\vu_h)\|_{L^2(\Omega)}$ & $\|p-p_h\|_{L^2(\Omega)}$ & iter \\
\hline
\hline
2 & 0.354 & 3.50999$\times 10^{-2}$ & 6.78690$\times 10^{-2}$ & 5.22585$\times 10^{-2}$ & 11 \\
3 & 0.177 & 4.56878$\times 10^{-3}$ & 9.56288$\times 10^{-3}$ & 1.07075$\times 10^{-2}$ & 8 \\
4 & 0.088 & 8.00090$\times 10^{-4}$ & 2.03639$\times 10^{-3}$ & 2.53571$\times 10^{-3}$ & 7 \\
5 & 0.044 & 2.08125$\times 10^{-4}$ & 5.91247$\times 10^{-4}$ & 6.25280$\times 10^{-4}$ & 7 \\
6 & 0.022 & 6.85220$\times 10^{-5}$ & 1.92531$\times 10^{-4}$ & 1.55842$\times 10^{-4}$ & 7 \\
\hline
\end{tabular}
\caption{Case 2, $\alpha=\gamma=1$, $\delta=10^{-5}$, $\tau=0.5$.} \label{tab:case2_gamma1_05}
\end{center}
\end{table}


\subsection{A fixed-point algorithm} \label{subsec:alt_algo}

Instead of the Lions--Mercier type algorithm introduced in Subsection \ref{subsec:L-M.Alg}, we explore the following fixed-point strategy.

\vspace*{0.2cm}
\noindent \emph{Initialisation}: $(\vT_h^{(0)},\vu_h^{(0)},p_h^{(0)})=\mathbf{0}$.

\vspace*{0.2cm}
\noindent Then for $k=0,1,\ldots$, do the following two steps.

\vspace*{0.2cm}
\noindent \emph{Step 1:} Find $(\vu_h^{(k+1)},p_h^{(k+1)})\in V_h\times Q_h$ such that
\begin{align*}
d(\vu_h^{(k+1)};\vu_h^{(k+1)},\vv_h)+\frac{1}{\alpha}\displaystyle{\int_{\Omega}}\vD(\vu_h^{(k+1)}):\vD(\vv_h)-\displaystyle{\int_{\Omega}}p_h^{(k+1)}\di(\vv_h) & =  \displaystyle{\int_{\Omega}}\vf\cdot\vv_h+\frac{\gamma}{\alpha}\displaystyle{\int_{\Omega}}\mu(|\vT_h^{(k)}|)\vT_h^{(k)}:\vD(\vv_h), \\
\displaystyle{\int_{\Omega}}q_h\di(\vu_h^{(k+1)}) & =  0
\end{align*}
for all $(\vu_h,q_h) \in V_h \times  Q_h$.

\vspace*{0.2cm}
\noindent \emph{Step 2:} Find $\vT_h^{(k+1)}\in M_h$ such that
\begin{equation*}
\alpha\int_{\Omega}\vT_h^{(k+1)}:\vS_h+\gamma\int_{\Omega}\mu(|\vT_h^{(k+1)}|)\vT_h^{(k+1)}:\vS_h = \int_{\Omega}\vD(\vu_h^{(k+1)}):\vS_h \quad \forall\, \vS_h\in M_h.
\end{equation*}
It is easy to show that this algorithm produces uniformly bounded sequences.

The solvers used for these two steps are similar to those described in Subsection \ref{subsec:L-M.Alg}. In particular, we take $(\vu_h^{(k)},p_h^{(k)})$ as initial guess for Newton's method for the finite element approximation of the Navier--Stokes system, except when $k=0$, in which case we use the solution of the associated Stokes problem.

The results obtained using the stopping criterion \eqref{eqn:stopLM} are given in Table \ref{tab:case1_alt_algo}. There are similar to those obtained in Table \ref{tab:case1_gamma1_05}.

\begin{table}[htbp]
\begin{center}
\begin{tabular}{|c|c|c|c|c|c|}
\hline
$n$ & $h$ & $\|\vT^d-\vT_h\|_{L^2(\Omega)}$ & $\|\nabla(\vu-\vu_h)\|_{L^2(\Omega)}$ & $\|p-p_h\|_{L^2(\Omega)}$ & iter \\
\hline
\hline
2 & 0.354 & 3.57082$\times 10^{-2}$ & 8.21052$\times 10^{-2}$ & 3.02063$\times 10^{-2}$ & 10 \\
3 & 0.177 & 7.73745$\times 10^{-3}$ & 1.86629$\times 10^{-2}$ & 6.18241$\times 10^{-3}$ & 8 \\
4 & 0.088 & 1.85777$\times 10^{-3}$ & 4.55234$\times 10^{-3}$ & 1.46369$\times 10^{-3}$ & 8 \\
5 & 0.044 & 4.60268$\times 10^{-4}$ & 1.13344$\times 10^{-3}$ & 3.60885$\times 10^{-4}$ & 8 \\
6 & 0.022 & 1.17442$\times 10^{-4}$ & 2.92590$\times 10^{-4}$ & 8.99455$\times 10^{-5}$ & 8 \\
\hline
\end{tabular}
\caption{Case 1, $\alpha=\gamma=1$, $\delta=10^{-5}$.} \label{tab:case1_alt_algo}
\end{center}
\end{table}

Concerning the computational cost when similar results are obtained, i.e., when $\tau=0.5$ for the Lions--Mercier type algorithm, we note that the latter requires the solution of one more equation per iteration, namely  the linear equation for $\vT_h^{(k+1)}$ in Step 2.

\begin{table}[htbp]
\begin{center}
\begin{tabular}{|c|r|r|r|r|r|r|}
\cline{2-7}
\multicolumn{1}{c}{ } & \multicolumn{3}{|c|}{Lions--Mercier, $\tau=0.5$} & \multicolumn{3}{c|}{Alt. Algo} \\
\hline
$h$ & iter & CPU time $[s]$ & wall time $[s]$ & iter & CPU time $[s]$ & wall time $[s]$ \\
\hline
0.354 & 10 & 10.33 & 14.76 & 10 & 7.30 & 10.22 \\
0.177 & 7 & 32.56 & 39.92 & 8 & 25.57 & 32.70 \\
0.088 & 7 & 134.28 & 162.05 & 8 & 94.76 & 111.77 \\
0.044 & 7 & 474.41 & 565.44 & 8 & 364.79 & 414.95 \\
0.022 & 7 & 1542.63 & 1695.55 & 8 & 1307.07 & 1539.07 \\
\hline
\end{tabular}
\caption{Time (in seconds) needed to meet the stopping criteria \eqref{eqn:stopLM} for the Lions--Mercier type algorithm (setup of Table \ref{tab:case1_gamma1_05}) and the algorithm of Section \ref{subsec:alt_algo}.} \label{tab:comp_cost}
\end{center}
\end{table}

\bibliographystyle{siam} 
\bibliography{limited_strain}

\end{document}